\documentclass[12pt]{amsart}
\usepackage[utf8]{inputenc}
 
\usepackage{amssymb}
\usepackage{bm}
\usepackage{graphicx}
\usepackage[centertags]{amsmath}
\usepackage{amsfonts}
\usepackage{amsthm}
\usepackage{enumerate}
\usepackage{tocvsec2}
\usepackage{xcolor}
\usepackage[margin=1in]{geometry}

\newtheorem{theorem}{Theorem}[section]
\newtheorem*{theorem*}{Theorem}

\newtheorem{corollary}[theorem]{Corollary}
\newtheorem{lemma}[theorem]{Lemma}
\newtheorem{rem}[theorem]{Remark}

\newtheorem{proposition}[theorem]{Proposition}

\theoremstyle{definition}

\newenvironment{remark}[1][Remark]{\begin{trivlist}
\item[\hskip \labelsep {\bfseries #1}]}{\end{trivlist}}

\newcommand{\rr}{\mathbb{R}}
\newcommand{\nn}{\mathbb{N}}

\newcommand{\ee}{\varepsilon}

\newcommand{\hhh}{\mathcal{H}}

\newcommand{\sep}{\text{sep}}
\newcommand{\cat}{^\smallfrown}
\newcommand{\ord}{\textbf{Ord}}

\begin{document}

\title[$\xi$-AUS, $\xi$-AUC, and $(\beta)$ operators]{$\xi$-asymptotically uniformly smooth, $\xi$-asymptotically uniformly convex, and $(\beta)$ operators}
\author{Ryan M. Causey}

\address{Department of Mathematics\\ University of South Carolina\\
Columbia, SC 29208\\ U.S.A.} \email{causey@math.sc.edu}

\author{Stephen J. Dilworth}

\address{Department of Mathematics\\ University of South Carolina\\
Columbia, SC 29208\\ U.S.A.} \email{dilworth@math.sc.edu}

\thanks{The second author was supported by the National Science Foundation under Grant Number DMS--1361461.}

\begin{abstract} For each ordinal $\xi$, we define the notion of $\xi$-asymptotically uniformly smooth and $w^*$-$\xi$-asymptotically uniformly convex operators.  When $\xi=0$, these extend the notions of asymptotically uniformly smooth and $w^*$-asymptotically uniformly convex Banach spaces.  We give a complete description of renorming results for these properties in terms of the Szlenk index of the operator, as well as a complete description of the duality between these two properties.  We also define the notion of an operator with property $(\beta)$ of Rolewicz which extends the notion of property $(\beta)$ for a Banach space.  We characterize those operators the domain and range of which can be renormed so that the operator has property $(\beta)$ in terms of the Szlenk index of the operator and its adjoint.

\end{abstract}

\maketitle

\section{Introduction} The notions of asymptotic uniform convexity and asymptotic uniform smoothness, first introduced by different names in \cite{Milman}, have since become important concepts in the geometry of Banach spaces, particularly in renorming theory.  Knaust, Odell, and Schlumprecht \cite{KOS} proved that a separable Banach space admits an equivalent asymptotically uniformly smooth norm if and only if its Szlenk index does not exceed $\omega$, the first infinite ordinal.  Raja \cite{Raja} extended this result to non-separable Banach spaces, and Godefroy, Kalton, and Lancien \cite{GKL} gave the optimal relationship between the $\ee$-Szlenk indices of a separable Banach space with Szlenk index $\omega$ and the power type moduli which can be obtain under some equivalent asymptotically smooth norm on that space.  The results in this direction as well as the duality between the modulus of asymptotic uniform smoothness of $X$ and the $w^*$-asymptotic uniform convexity of $A^*$, and in particular the characterization of those Banach spaces admitting an equivalent norm with property $(\beta)$ of Rolewicz given in \cite{DKLR},  parallel the well-known corresponding renorming results concerning superreflexive Banach spaces, uniform smoothness, and uniform convexity. It was shown in \cite{DKLR} that a Banach space admits an equivalent norm with property $(\beta)$ of Rolewicz if and only if it is reflexive and $X$ and $X^*$ both have Szlenk index not exceeding $\omega$.  Moreover, in \cite{LPR}, for every countable ordinal $\alpha$, the notion of $\omega^\alpha$-UKK$^*$ was defined, which generalizes the notion of $w^*$-asymptotic uniform convexity of $X^*$.   These authors also proved the analogous renorming theorems to that of Knaust, Odell, and Schlumprecht:  For each countable ordinal $\alpha$, a  separable Banach space admits an equivalent norm whose dual norm is $\omega^\alpha$-UKK$^*$ if and only if the Szlenk index of the Banach space does not exceed $\omega^{\alpha+1}$.   

The purpose of this paper is to introduce for every ordinal $\xi$ the notions of $\xi$-asymptotic uniform smoothness, $\xi$-asymptotic uniform convexity, and $w^*$-$\xi$-asymptotic uniform convexity of an operator.  In the case that $\xi=0$ and the operator in question is an identity, these recover the usual notions of asymptotic uniform smoothness, asymptotic uniform convexity, and $w^*$-$\xi$-asymptotic uniform smoothness.  The notion of $w^*$-$\xi$-asymptotic uniform convexity we define extends to operators the notion of $\omega^\xi$-UKK$^*$ norms defined in \cite{LPR}.  To our knowledge, the notion of $\xi$-asymptotic uniform smoothness when $\xi>0$ has not been considered previously in the literature, even for Banach spaces, and $0$-asymptotic uniform smoothness has not been considered for operators.  We also define what it means for an operator to have property $(\beta)$ of Rolewicz.  In what follows, $Sz(A)$ denotes the Szlenk index of $A$.  Any unknown terminology will be defined in the second section. 

Our first result is a complete description of the duality between $\xi$-asymptotic uniform smoothness and $w^*$-$\xi$-asymptotic uniform convexity.  

\begin{theorem} For any operator $A:X\to Y$ and any ordinal $\xi$, $A$ is $\xi$-asymptotically uniformly smooth if and only if $A^*$ is $w^*$-$\xi$-asymptotically uniformly convex.   Furthermore, for $1<p\leqslant \infty$ and $1/p+1/q=1$,  the modulus of $\xi$-asymptotic uniform smoothness of $A$ is of power type $p$ if and only if the modulus of $w^*$-$\xi$-asymptotic uniform convexity of $A^*$ is of power type $q$.

\end{theorem}

We obtain a number of renorming results, typified by the following. When we say an operator admits an equivalent norm with a certain property, we mean there exist equivalent norms on the domain and range yielding the stated property.

\begin{theorem} \begin{enumerate}[(i)]\item For any ordinal $\xi$, an operator $A:X\to Y$ admits an equivalent $\xi$-asymptotically uniformly smooth norm if and only if $Sz(A)\leqslant \omega^{\xi+1}$.  \item For any ordinals $\xi, \zeta$ and any weakly compact operator $A:X\to X$ such that $Sz(A)\leqslant \omega^{\xi+1}$ and $Sz(A^*)\leqslant \omega^{\zeta+1}$, there exists a single, equivalent norm $|\cdot|$ on $X$ such that $A:(X, |\cdot|)\to (X, |\cdot|)$ is $\xi$-asymptotically uniformly smooth and $\zeta$-asymptotically uniformly convex.  \end{enumerate} \label{main2}
\end{theorem}

We also say $A:X\to Y$ has \emph{property} $(\beta)$ if for any $\ee>0$, there exists $\beta=\beta(\ee)>0$ such that for any $x\in B_X$ and any sequence $(x_n)\subset B_X$ such that $\inf_{m\neq n}\|Ax_m-Ax_n\|\geqslant \ee$, there exists $n\in \nn$ such that $\|x+x_n\|\leqslant 2(1-\beta)$.   This is an operator version of property $(\beta)$ of Rolewicz, in that a Banach space has property $(\beta)$ of Rolewicz if and only if $I_X$ has property $(\beta)$ by our definition.  As mentioned above, in \cite{DKLR} it was shown that $X$ can be renormed to have property $(\beta)$ if and only if $X$ is reflexive and $Sz(X), Sz(X^*)\leqslant \omega$.  In complete analogy, we prove the following.  

\begin{theorem} An operator $A:X\to Y$ admits an equivalent norm with property $(\beta)$ if and only if $A$ is weakly compact and $Sz(A), Sz(A^*)\leqslant \omega$.

\end{theorem}

In \cite{BKL}, a metric characterization was given of those reflexive Banach spaces which can be renormed to have property $(\beta)$ as those reflexive Banach spaces into which one cannot embed the countably infinitely branching tree $T=\cup_{n=0}^\infty \nn^n$ with its graph distance, where each sequence is adjacent to its extensions by one term.  In complete analogy, we prove the following operator analogue concerning the preservation of the metric space through $T$ an operator.  Let us say that $T$ \emph{factors through} $A:X\to Y$ provided there exist a function $\varphi:T\to X$  and $D>0$ such that for every $s,t\in T$, $$d(s,t)/D\leqslant \|A\varphi(s)-A\varphi(t)\|, \|\varphi(s)-\varphi(t)\| \leqslant d(s,t).$$  

\begin{theorem} Let $A:X\to Y$ be a weakly compact operator. Then exactly one of the following alternatives holds: 

\begin{enumerate}[(i)]\item $T$ factors through $A$. \item $A$ admits an equivalent norm with property $(\beta)$. \end{enumerate}

\end{theorem}

We also include a discussion of the different power types possible.  It is known \cite{KOS}, \cite{Raja} that if a Banach space admits an asymptotically uniformly smooth norm, it admits an asymptotically uniformly smooth norm with power type modulus. This is in contrast to both the operator case and to the higher ordinal case.

\begin{theorem} Fix an ordinal $\xi$. \begin{enumerate}[(i)]\item For any $1<p\leqslant \infty$, there exists a Banach space which is $\xi$-asymptotically uniformly smooth with modulus of power type $p$ and which cannot be renormed to be $\xi$-asymptotically uniformly smooth with modulus of power type better than $p$. \item There exists a Banach space with Szlenk index $\omega^{\xi+1}$ which cannot be renormed to be $\xi$-asymptotically uniformly smooth with any power type modulus if and only if $\xi>0$. \item There exists an operator with Szlenk index $\omega^{\xi+1}$ which cannot be renormed to be $\xi$-asymptotically uniformly smooth with any power type modulus. \item All of the appropriate dual statements for $w^*$-$\xi$-asymptotically uniformly convex Banach spaces and operators also hold. \end{enumerate}

\end{theorem}

\section{Preliminaries}

Given a set $\Lambda$, we let $\Lambda^{<\nn}$ denote the finite sequences in $\Lambda$, including the empty sequence $\varnothing$.  We order $\Lambda^{<\nn}$ by initial segments, denoted $\prec$.  That is, $s\prec t$ if $s$ is a proper initial segment of $t$.  We let $|s|$ denote the length of $s$, $s|_i$ the initial segment of $s$ having length $i$ (provided $0\leqslant i\leqslant |s|$), and $s^-$ the maximal, proper initial segment of $s$ (provided $s\neq \varnothing$).  We let $s\cat t$ denote the concatenation of $s$ with $t$.   Given a subset $T$ of $\Lambda^{<\nn}$, we let $MAX(T)$ denote the maximal members of $T$ with respect to the order $\prec$, and $T'=T\setminus MAX(T)$.    We define the transfinite derived sets $T^0=0$, $T^{\xi+1}=(T^\xi)'$, and if $\xi$ is a limit ordinal, $T^\xi=\cap_{\zeta<\xi} T^\zeta$.    We let $o(T)=\min\{\xi: T^\xi=\varnothing\}$ if this class is non-empty, and we write $o(T)=\infty$ otherwise.   We refer to $o(T)$ as the \emph{order} of $T$.    We say $o(T)$ is \emph{ill-founded} if $o(T)=\infty$ and $T$ is \emph{well-founded} otherwise.  

We say $T\subset \Lambda^{<\nn}$ is a \emph{tree} provided that for any $s\prec t\in T$, $s\in T$.    In this case, $T^\xi$ is also a tree for every ordinal $\xi$.  We say $T\subset \Lambda^{<\nn}\setminus\{\varnothing\}$ is a $B$-\emph{tree} provided that for every $\varnothing \neq s\prec t\in T$, $s\in T$.  Then $T$ is a $B$-tree if and only if $T\cup\{\varnothing\}$ is a tree.   In this case, $T^\xi$ is also a $B$-tree for every $\xi$.   We note that if $T$ is a $B$-tree on the set $\Lambda$, if $t\in \Lambda^{<\nn}$, and if $T_t$ is the set of all non-empty sequences $u\in \Lambda^{<\nn}$ such that $t\cat u\in T$, $T_t$ is a $B$-tree and $o(T_t)=\xi$ if and only if $t$ is maximal in $T^\xi$ \cite{Causey}.

If $U\subset \Lambda_1^{<\nn}$ and $V\subset \Lambda_2^{<\nn}$, we say a function $\theta:U\to V$ is \emph{monotone} provided that for every $s,t\in U$ with $s\prec t$, $\theta(s)\prec \theta(t)$.    If $S,T$ are trees (resp. $B$-trees), then there exists a monotone map $\theta:S\to T$ if and only if $o(S)\leqslant o(T)$ (where we agree that $\xi<\infty$ for any ordinal $\xi$).    Moreover, this map may be taken to preserve lengths  \cite{Causey2}.

Given a set $T\subset \Lambda^{<\nn}$ and a directed set $D$, we let $T D=\{(t, s): t\in T, s\in D^{<\nn}, |t|=|s|\}$.  Here, we identify pairs of sequences $(t,s)$ such that $|t|=|s|$  with sequences of pairs by identifying $(\varnothing, \varnothing)$ with $\varnothing$ and $((a_i)_{i=1}^n, (b_i)_{i=1}^n)$ with $((a_i, b_i))_{i=1}^n$.  With this identifcation, we will think of $T D$ as sequences of pairs, not pairs of sequences. In this case, the length of the sequence $(t,s)$ is the common length of $s$ and $t$. In the case that $D$ also consists of pairs, we will think of $T D$ as sequences of triples by identifying $((a_i)_{i=1}^n, ((b_i, c_i))_{i=1}^n)$ with $((a_i, b_i, c_i))_{i=1}^n$ and $(\varnothing, \varnothing)$ with $\varnothing$.   Note that if $T$ is a tree (resp. $B$-tree), so is $T D$.  Moreover, for any ordinal $\xi$, $(T D)^\xi=T^\xi D$, so $o(T D)=o(T)$.   If $U\subset SD$ and $V\subset TD$, we say $\theta:U\to V$ is a \emph{pruning} provided that $\theta$ is monotone, and for each $s=s_1\cat (a, d)\in SD$, if $\theta(s)=t_1\cat (b, d')$, then $d\leqslant_D d'$, where $\leqslant_D$ denotes the order on $D$. An \emph{extended pruning} will be a pair $(\theta, e)$ of maps such that $\theta:U\to V$ is a pruning and $e:MAX(U)\to MAX(V)$ is a function such that for every $t\in MAX(U)$, $\theta(t)\preceq e(t)$.   We note that if $S,T$ are trees (resp. $B$-trees) with $o(S)\leqslant o(T)$, then by the previous paragraph there exists a length-preserving, monotone map $\theta:S\to T$.  We may then define $\phi:S D\to T D$ by letting $\phi(\varnothing)=\varnothing$ (which we omit in the case of $B$-trees) and by letting $\phi(((\lambda_i, d_i))_{i=1}^n) = ((\mu_i, d_i))_{i=1}^n$, where $(\mu_i)_{i=1}^n= \theta((\lambda_i)_{i=1}^n)$.  Thus we see that if $o(S)\leqslant o(T)$, then for any directed set $D$, there exists a length-preserving pruning $\phi:SD\to TD$.   

We define some particularly useful trees for later use.  We let $T_0=\{\varnothing\}$.  If $T_\xi$ has been defined, we let $$T_{\xi+1}=\{\varnothing\}\cup \{(\xi+1)\cat t: t\in T_\xi\}.$$  If $\xi$ is a limit ordinal and $T_\zeta$ has been defined for every $\zeta<\xi$, we let $T_\xi=\cup_{\zeta<\xi}T_{\zeta+1}$.    For each ordinal $\xi$, we let $B_\xi=T_\xi\setminus\{\varnothing\}$.  It is easy to see that $T_\xi$ is a tree with $o(T_\xi)=\xi+1$, and $B_\xi$ is a $B$-tree with $o(B_\xi)=\xi$.   

Suppose $T$ is a tree, $X$ is a Banach space, and $\tau$ is some topology on $X$.  We say a collection $(x_t)_{t\in T}$ is $\tau$-\emph{closed} provided that for every ordinal $\xi$ and every $t\in T^{\xi+1}$, $$x_t\in \overline{\{x_s: s\in T^\xi, s^-=t\}}^\tau.$$   If $A:X\to Y$ is an operator and $\ee>0$, we say the collection $(x_t)_{t\in T}$ is $(A, \ee)$-\emph{separated} provided that for every $\varnothing\neq t\in T$, $\|Ax_t-Ax_{t^-}\|\geqslant \ee$.   

If $B$ is a $B$-tree, $X$ is a Banach space, and $\tau$ is some topology on $X$, we say a collection $(x_t)_{t\in B}$ is $\tau$-\emph{null} provided that for every ordinal $\xi$ and every $t\in (B\cup \{\varnothing\})^{\xi+1}$, $$0\in \overline{\{x_s: s\in B^\xi, s^-=t\}}^\tau.$$   If $A:X\to Y$ is an operator and $\ee>0$, we say the collection $(x_t)_{t\in B}$ is $(A, \ee)$-\emph{large} provided that for every $t\in B$, $\|Ax_t\|\geqslant \ee$.

In the remainder of this section and in the next section, $A:X\to Y$ is a fixed operator, $M$ is the set of all weakly open sets in $X$ which contain $0$, and $N$ is the set of all $w^*$-open sets in $Y^*$ which contain $0$.  Direct $D=M\times N$ by $(U,V)\leqslant_D (U', V')$ if $U'\subset U$ and $V'\subset V$.    Let $S_\xi=T_\xi D$ and $R_\xi=B_\xi D$.  Note that $R_\xi$ is the collection of all finite sequences of triples $s=((\zeta_i, U_i, V_i))_{i=1}^n$ such that $(\zeta_i)_{i=1}^n\in B_\xi$, $U_i\in M$, and $V_i\in N$.  The tree $S_\xi$ is $R_\xi$ together with the empty sequence.      We say a collection $(x_t)_{t\in R_\xi}\subset X$ is \emph{normally weakly null} provided that for any $s=s_1\cat (\zeta, U,V)\in R_\xi$, $x_s\in U$.  We say a collection $(y^*_t)_{t\in R_\xi}\subset Y^*$ is \emph{normally} $w^*$-\emph{null} if for every $s=s_1\cat (\zeta, U,V)\in R_\xi$, $y^*_s\in V$.  We say a collection $(y^*_t)_{t\in S_\xi}\subset Y^*$ is \emph{normally} $w^*$-\emph{closed} if for every $s=s_1\cat (\zeta, U,V)\in R_\xi$, $y^*_s-y^*_{s_1}\in V$.   It is easy to see that normally weakly null collections are weakly null, normally $w^*$-null collections are $w^*$-null, and normally $w^*$-closed collections are $w^*$-closed.  Moreover, if $\theta:R_\xi\to R_\zeta$ is a pruning for some ordinals $\xi, \zeta$, and if $(x_t)_{t\in R_\zeta}$ is normally weakly null, then $(x_{\theta(t)})_{t\in R_\xi}$ is also normally weakly null, and a similar statement holds for normally $w^*$-null collections $(y^*_t)_{t\in R_\zeta}$. When we say that $(x_{\theta(t)})_{t\in R_\xi}$ is normally weakly null, we mean that the collection $(u_t)_{t\in R_\xi}:=(x_{\theta(t)})_{t\in R_\xi}$ is normally weakly null.  More precisely, if $t=t_1\cat (\eta, U,V)$, $x_{\theta(t)}\in U$.   Moreover, if $\theta:S_\xi\to S_\zeta$ is a pruning which is length-preserving, then $(y^*_{\theta(t)})_{t\in T_\xi}$ is normally $w^*$-closed if $(y^*_t)_{t\in T_\xi}$ is normally $w^*$-closed.  This is because in this case, since $\theta$ preserves lengths, $\theta(t^-)=\theta(t)^-$ for any $\varnothing\neq t\in T_\xi$.

Given an ordinal $\xi$, $\sigma>0$, and a topology $\tau$ (either the weak or $w^*$ topology) on $X$, we let $$\rho_\xi^\tau(\sigma;A)= \sup\Bigl\{\inf\{ \|y+\sigma Ax\|-1: t\in B, x\in \text{co}(x_s: \varnothing\prec s\preceq t)\}\Bigr\},$$ where the supremum is taken over all $y\in B_Y$, all $B$-trees $B$ with $o(B)=\omega^\xi$, and all $\tau$-null collections $(x_t)_{t\in B}\subset B_X$. If $\xi=0$, since $B$-trees of order $1$ are totally incomparable unions of length $1$ sequences, $\rho^w_0(\sigma;A)$ can be defined by nets by $$\sup\Bigl\{\underset{\lambda}{\lim\sup}\|y+\sigma Ax_\lambda\|-1: y\in B_Y, (x_\lambda)\subset B_X, x_\lambda\underset{\tau}{\to}0\Bigr\}.$$   We let $$\delta_\xi^\tau(\sigma;A)=\inf\Bigl\{ \sup\{\|x+\sigma\sum_{\varnothing\prec s\preceq t} x_s\|-1: t\in B\}\Bigr\},$$ where the infimum is taken over all $x\in X$ with $\|x\|\geqslant 1$, all $B$-trees $B$ with $o(B)=\omega^\xi$, and all $\tau$-null, $(A,1)$-large collections $(x_t)_{t\in B}$.  Again, if $\xi=0$, this modulus can be computed by nets as $$\delta^{w^*}_0(\sigma;A^*)=\inf\{\underset{\lambda}{\lim\inf} \|x+\sigma x_\lambda\|-1: \|x\|\geqslant 1, \|Ax\|\geqslant 1, x_\lambda\underset{w^*}{\to}0\}.$$   We note that if $Y=\{0\}$, then $\rho^\tau_\xi(\sigma;A)=-1$, and otherwise $\rho^\tau_\xi(\sigma;A)\geqslant 0$ for all $\sigma$. Indeed, in this case we may fix $y\in S_Y$ and let $x_t=0$ for some $B$-tree $B$ with $o(B)=\omega^\xi$ and all $t\in B$.  We note also that $\delta^\tau_\xi(\sigma;A)\geqslant 0$.  Indeed, if $A$ is such that there does not exist any $B$-tree $B$ with $o(B)=\omega^\xi$ and a $\tau$-null, $(A,1)$-large collection $(x_t)_{t\in B}$, then we are taking the infimum of the empty set.  Otherwise for any $B$-tree $B$ with $o(B)=\omega^\xi$,  any $\tau$-null, $(A,1)$-large collection $(x_t)_{t\in B}\subset X$, and any $x\in X$ with $\|x\|\geqslant 1$, the collection $\{x_t: t\in B, |t|=1\}$ contains a $\tau$-null net, say $(z_\lambda)$, whence $$\sup_{t\in B} \|x+\sigma\sum_{s\preceq t}x_s\| \geqslant \underset{\lambda}{\lim\inf} \|x+\sigma z_\lambda\| \geqslant \|x\|\geqslant 1.$$

We say $A:X\to Y$ is $\xi$-\emph{asymptotically uniformly smooth} if $\rho^w_\xi(\sigma;A)/\sigma\to 0$ as $\sigma\to 0$, where $w$ denotes the weak topology on $X$. We abbreviate this as $\xi$-AUS. We say $A:X\to Y$ is $\xi$-\emph{asymptotically uniformly convex} if $\delta^w_\xi(\sigma;A)>0$ for every $\sigma>0$.   If $A:X\to Y$ is an adjoint, we say $A$ is $w^*$-$\xi$-AUS (resp. $w^*$-$\xi$-AUC) if $\rho^{w^*}_\xi(\sigma;A)/\sigma\to 0$ as $\sigma\to 0$ (resp. $\delta^{w^*}_\xi(\sigma;A)>0$ for every $\sigma>0$).    In the case that $\xi=0$, we simply write AUS, AUC, etc. in place of $0$-AUS, $0$-AUC, etc.   One easily checks that if $A$ is the identity operator on $X$, then the definition of AUS or AUC coincides with the usual definition of what it means for a Banach space to be AUS or AUC, and analogous statements hold for $w^*$-AUS and $w^*$-AUC if $X$ is a dual space.   

It is easy to see that while renorming $Y$ may change the moduli $\delta^w_\xi(\sigma;A)$, $\delta^{w^*}_\xi(\sigma;A)$, whether or not $A$ is $\xi$-AUC or $w^*$-$\xi$-AUC is invariant under renorming $Y$.  Similarly, whether $A$ is $\xi$-AUS or $w^*$-$\xi$-AUS is invariant under renorming $X$.

The purpose of normal weak or $w^*$-nullity is that it will be particularly convenient to pass to subtrees having prescribed biorthogonal behavior during dualization.  The content of our first proposition is that in order to compute the moduli above, it is sufficient to consider these normally weakly or $w^*$ null trees.

\begin{proposition} \begin{enumerate}[(i)]\item For any well-founded $B$-tree $B$ and any ordinal $\xi$ such that $0<\xi\leqslant o(B)$, if $(x_t)_{t\in B}$ is weakly null, there exists a length-presrving, monotone $\theta:R_\xi\to B$ such that $(x_{\theta(t)})_{t\in R_\xi}$ is normally weakly null. \item For any well-founded $B$-tree $B$ and any ordinal $\xi$ such that $0<\xi\leqslant o(B)$, if $(z^*_t)_{t\in B}$ is $w^*$-null, there exists a length-preserving, monotone $\theta:R_\xi\to B$ such that $(y^*_t)_{t\in R_\xi}$ is normally $w^*$-null.   \end{enumerate}

\label{normal prop}
\end{proposition}

\begin{proof} We only prove $(i)$, with the proof of $(ii)$ being similar.   We will prove by induction on $\Omega=\{(\zeta, \xi)\in \ord^2: 0<\xi\leqslant \zeta\}$ ordered lexicographically that for any pair $(\zeta, \xi)\in \Omega$, any $B$-tree with $o(B)=\zeta$, and any weakly null collection $(x_t)_{t\in B}$, there exists a length-preserving, monotone $\theta:R_\xi\to B$ such that $(x_{\theta(t)})_{t\in R_\xi}$ is normally weakly null.  To obtain a contradiction, suppose the claim fails and let $(\zeta, \xi)$ be a minimal pair for which the claim fails.  Let $B$ be a $B$-tree with $o(B)=\zeta$ and $(x_t)_{t\in B}$ a weakly null collection witnessing the failure of the claim.   First suppose that $\xi=\gamma+1$.    Since $o(B\cup\{\varnothing\})=o(B)+1=\zeta+1$, $\varnothing\in (B\cup \{\varnothing\})^{\gamma+1}$.  Thus if $R$ is the collection of sequences in $B^\gamma$ having length $1$, $0\in \overline{\{x_t: t\in R\}}^w$.  Fix $(U,V)\in D$.   Let $\theta((\xi, U,V))=t$ for some $t\in R$ such that $x_t\in U$.    With this $U,V$ and $t$ still fixed, since $t\in B^\gamma$, if $B_t$ denotes those non-empty sequences such that $t\cat u\in B$, $o(B_t)\geqslant \gamma$.    Moreover, $(x_{t\cat u})_{u\in B_t}$ is weakly null.  We may then fix $\theta_{U,V}:R_\gamma\to B_t$ which is length-preserving and monotone and such that $(x_{t\cat \theta_{U,V}(u)})_{u\in B_t}$ is normally weakly null.  We then define $$\theta((\xi, U,V)\cat u)=t\cat \theta_{U,V}(u)$$ for each $u\in R_\gamma$.  This defines $\theta$ on all of $R_\xi$.   Moreover, it is easy to see from the construction that $\theta$ satisfies all of the requirements, yielding a contradiction.    This shows that $\xi$ cannot be a successor ordinal, and since $\xi\neq 0$, it follows that $\xi$ must be a limit ordinal.    Then $R_\xi=\cup_{\gamma<\xi}R_{\gamma+1}$.   By minimality of $\xi$, for each $\gamma<\xi$, there exists a length-preserving, monotone $\theta_\gamma:R_{\gamma+1}\to B$ such that $(x_{\theta(t)})_{t\in R_{\gamma+1}}$ is normally weakly null.  Then let $\theta:R_\xi\to B$ be given by $\theta|_{R_{\gamma+1}}=\theta_\gamma$.  Again, this $\theta$ satisfies the conclusions, and we reach a contradiction.

\end{proof}

\begin{remark} Note that if $\theta:R_\xi\to B$ is length-preserving and monotone, then for any $(z^*_t)_{t\in B}\subset Y^*$ and $t\in R_\xi$, $\sum_{\varnothing\prec s\preceq t} z^*_{\theta(s)}= \sum_{\varnothing\prec s\preceq \theta(t)} z^*_s$.  With this fact, it is easy to see that if an operator $A^*:Y^*\to X^*$ is $w^*$-$\xi$-AUC, it is $w^*$-$\zeta$-AUC for any $\zeta>\xi$, and a similar statement holds for $\xi$-AUC.  Moreover, it is even more readily apparent that any operator which is  $\xi$-AUS (resp. $w^*$-$\xi$-AUS) is $\zeta$-AUS (resp. $w^*$-$\zeta$-AUS) for any $\zeta>\xi$.   

It is also easy to see that any compact operator is AUS and AUC, and therefore $\xi$-AUS and $\xi$-AUC for every $\xi$.  Moreover, any compact adjoint is $w^*$-AUS and $w^*$-AUC.    

\end{remark}

It is not clear that $\rho^\tau_\xi(\cdot;A)$ is convex when $\xi>0$. However, convexity of $\rho^\tau_\xi(\cdot;A)$ is a consequence of the fact that in order to check the infimum over all convex combinations of the branches of a $\tau$-null tree, it is sufficient to check over ``special convex combinations.'' This phenomenon will be used again in the sequel.

\begin{proposition}
For every $\xi$, every operator $A:X\to Y$, and topology $\tau$ either a weak or $w^*$ topology on $X$, $\rho_\xi^\tau(\cdot;A)$ is a convex function.  
\end{proposition}

\begin{proof}  Suppose $A:X\to Y$ is an operator, $\xi$ is an ordinal, $\tau$ is either a weak or $w^*$-topology on $X$, $\sigma_1, \sigma_2, \ee>0$, and $\lambda\in (0,1)$ are such that $$\rho^\tau_\xi(\lambda \sigma_1+(1-\lambda)\sigma_2;A)=5\ee+\lambda \rho^\tau_\xi(\sigma_1;A)+(1-\lambda)\rho^\tau_\xi(\sigma_2;A).$$ Let $a=\rho^\tau_\xi(\sigma_1;A)$ and $b=\rho^\tau_\xi(\sigma_2;A)$.  Fix $y\in B_Y$, a $B$-tree $B$ with $o(B)=\omega^\xi$, and a weakly null collection $(x_t)_{t\in B}\subset B_X$ such that for every $t\in B$ and every $x\in \text{co}(x_s:\varnothing\prec s\preceq t)$, $\|y+(\lambda \sigma_1 +(1-\lambda)\sigma_2)Ax\|>1+4\ee+\lambda a+(1-\lambda)b= 4\ee+\lambda (1+a)+(1-\lambda)(1+ b)$.  

Let $\Gamma_\xi$ be the $B$ tree of order $\omega^\xi$ defined in \cite{Causey2}.  Let $D$ is a $\tau$-neighborhood basis at $0$ in $X$.  We may first fix a monotone $\theta:\Omega_\xi:=\Gamma_\xi D\to B$ such that $(x_{\theta(t)})_{t\in \Omega_\xi}$ is normally weakly null.  By relabeling, we may assume that $B=\Omega_\xi$.  Let $\mathbb{P}_\xi:\Omega_\xi\to [0,1]$ be the function defined in \cite{Causey2}. We recall that $\mathbb{P}_\xi$ has the property that for every $t\in MAX(\Omega_\xi)$, $\sum_{s\preceq t}\mathbb{P}_\xi(s)=1$.   Let $\Pi\Omega_\xi=\{(s,t)\in \Omega_\xi\times MAX(\Omega_\xi): s\preceq t\}$.

For each $t\in MAX(\Omega_\xi)$, let $z_t=\sum_{\varnothing\prec s\preceq t} \mathbb{P}_\xi(s)x_s\in \text{co}(x_s: \varnothing\prec s \preceq t)$ and fix $y^*_t\in B_{Y^*}$ such that \begin{align*} \|y+(\lambda \sigma_1 +(1-\lambda)\sigma_2)Az_t\| & =\text{Re\ }y^*_t(y+(\lambda \sigma_1+(1-\lambda)\sigma_2)Az_t) \\ & > 4\ee+\lambda (1+a)+(1-\lambda)(1+b).\end{align*} Define $f,g,h:\Pi\Omega_\xi\to \rr$ by $$f(s,t)=\text{Re\ }z^*_t(\lambda y+\lambda \sigma_1 Ax_s),$$ $$g(s,t)=\text{Re\ }z^*_t((1-\lambda) y+(1-\lambda) \sigma_2 Ax_s),$$ $h(s,t)=f(s,t)+ g(s,t)$.  Note that for any $t\in MAX(\Omega_\xi)$, $$\sum_{\varnothing\prec s\preceq t}\mathbb{P}_\xi(s)h(s,t)= \text{Re\ }y^*_t(y+(\lambda \sigma_1+(1-\lambda)\sigma_2)Az_t)>4\ee+\lambda (1+a)+(1-\lambda)(1+b).$$ This means the function $h$ is $4\ee+\lambda (1+a)+(1-\lambda)(1+b)$-large, as defined in \cite{Causey2}.  By Theorem $1.5$ of \cite{Causey2}, there exists an extended pruning $(\theta,e):\Omega_\xi\to \Omega_\xi$ such that for every $(s,t)\in \Pi\Omega_\xi$, $h(\theta(s), e(t))>3\ee+\lambda (1+a)+ (1-\lambda)(1+b)$.

Next, fix $R>\max\{\|A\|, 1\}$ and fix a finite partition $\mathcal{P}$ of $[-R, R]^2$ into subsets of diameter (with respect to $\ell_\infty^2$ metric) less than $\delta$, where \begin{enumerate}[(i)]\item $(1+\sigma_1)\delta <\frac{\ee}{2\lambda}$, \item $(1+\sigma_2)\delta< \frac{\ee}{2(1-\lambda)}$, \item $(1+\lambda \sigma_1+(1-\lambda)\sigma_2)\delta<\ee$. \end{enumerate} Define $\chi:\Pi\Omega_\xi\to \mathcal{P}$ by letting $\chi(s,t)$ denote the member $S$ of $\mathcal{P}$ such that $$(\text{Re\ }y^*_{e(t)}(y), \text{Re\ }y^*_{e(t)}(Ax_{\theta(s)}))\in S.$$   By \cite[Proposition $4.6(iii)$]{Causey1}, there exists an extended pruning $(\theta', e'):\Omega_\xi\to \Omega_\xi$ and $S\in \mathcal{P}$ such that for every $(s,t)\in \Pi\Omega_\xi$, $\chi(\theta'(s), e'(t))=S$.   Fix any $(\eta, \mu)\in S$.  For each $s\in \Omega_\xi$, let $u_s=x_{\theta\circ \theta'(s)}$ and for each $t\in MAX(\Omega_\xi)$, let $u^*_t=y^*_{e\circ e'(t)}$.   Note that \begin{enumerate}[(i)]\item for every $(s,t)\in \Pi\Omega_\xi$, $$3\ee+\lambda (1+a)+(1-\lambda)(1+b)< h(\theta\circ \theta'(s), e\circ e'(t))= \text{Re\ }u^*_t(y+(\lambda \sigma_1+(1-\lambda)\sigma_2)Au_s),$$ \item for any $t\in MAX(\Omega_\xi)$, $|\eta-\text{Re\ }u^*_t(y)|<\delta$,  \item for any $(s,t)\in \Pi\Omega_\xi$, $|\mu-\text{Re\ }u^*_t(Au_s)|<\delta$.   \end{enumerate}

Combining these yields that \begin{align*}3\ee+\lambda (1+a)+(1-\lambda)(1+b) & \leqslant  \text{Re\ }u^*_t(y+(\lambda \sigma_1+(1-\lambda)\sigma_2)Au_s) \\ & < (1+\lambda \sigma_1+(1-\lambda)\sigma_2)\delta + \eta+(\lambda \sigma_1+(1-\lambda)\sigma_2)\mu \\ & < \ee+ \eta+(\lambda\sigma_1 +(1-\lambda)\sigma_2 )\mu. \end{align*} 

From this it follows that $$2\ee+ \lambda (1+a)+(1-\lambda)(1+b) < \lambda(\eta+\sigma_1 \mu)+(1-\lambda)(\eta+\sigma_2 \mu).$$  Either $\ee+\lambda (1+a)<\lambda (\eta+\sigma_1 \mu)$ or $\ee+(1-\lambda)(1+b) < (1-\lambda)(\eta+\sigma_2\mu)$.   First assume that $\ee+\lambda (1+a)<\lambda(\eta+\sigma_1\mu)$.  Fix a maximal member $t$ of $\Omega_\xi$ and note that for any $u\in \text{co}(u_s:\varnothing\prec s\preceq t)$, \begin{align*} \|y+\sigma_1 Au\| & \geqslant \text{Re\ }u^*_t(y+\sigma_1 Au) \geqslant \eta+\sigma_1 \mu - (1+\sigma_1)\delta \\ & > (1+a+ \frac{\ee}{\lambda})-(1+\sigma_1) \delta> 1+a + \frac{\ee}{2\lambda}= 1+\rho^w_\xi(\sigma_1;A)+ \frac{\ee}{2\lambda}.\end{align*}   But this is a contradiction of the definition of $\rho^w_\xi(\sigma_1;A)$, since $(u_s)_{s\in \Omega_\xi}$ is $\tau$-null.   The assumption that $\ee+(1-\lambda)(1+b)<(1-\lambda)(\eta+\sigma_2 \mu)$ yields a similar contradiction, and these contradictions finish the proof.

\end{proof}

\section{Duality}

In this section, as in the previous section, $A:X\to Y$ is a fixed operator.  The main result of this section is the following.  

\begin{theorem} For any ordinal $\xi$, $A$ is $\xi$-AUS if and only if $A^*$ is $w^*$-$\xi$-AUC. 

\label{duality}
\end{theorem}

Note that in order to check this fact, we only need to consider the behavior of $\rho_\xi^w(\sigma;A)$ and $\delta_\xi^{w^*}(\sigma;A^*)$ for $\sigma$ sufficiently close to $0$. As in \cite{DKLR}, we will do this by showing that $\delta_\xi^{w^*}(\cdot;A^*)$ is equivalent to the Young dual of $\rho_\xi^w(\cdot;A)$.    

\begin{proposition} Fix $0<\sigma, \tau\leqslant 1$ and any ordinal $\xi$.  \begin{enumerate}[(i)]\item If $\rho^w_\xi(\sigma;A)\leqslant \sigma\tau$, then $\delta^{w^*}_\xi(6\tau;A^*)\geqslant \sigma\tau$.  \item If $\delta^{w^*}_\xi(\tau;A^*)\geqslant \sigma\tau$, then $\rho^w_\xi(\sigma;A)\leqslant \sigma\tau$.   \end{enumerate}

\label{dualprop}
\end{proposition}

Let us first show how this yields the result. If $A=0$, $A$ is $\xi$-AUS and $A^*$ is $w^*$-$\xi$-AUC, so assume $A\neq 0$.   First suppose that $\lim_{\sigma\to 0^+} \rho^w_\xi(\sigma;A)/\sigma=0$.   Fix $\tau\in (0,1)$ and choose $\sigma_0\in (0,1)$ such that for every $0<\sigma\leqslant \sigma_0$, $\rho^w_\xi(\sigma;A)\leqslant \sigma\tau/6$.   Then by $(i)$ of Proposition \ref{dualprop}, $\delta^{w^*}_\xi(\tau;A^*)\geqslant \sigma_0\tau>0$.   Next, assume $\delta^{w^*}_\xi(\tau;A^*)>0$ for all $\tau>0$.  Fix $\tau\in (0,1)$ and let $\sigma\in (0,1)$ be a number not exceeding $\delta^{w^*}_\xi(\tau;A^*)/\tau$.   Then $\delta^{w^*}_\xi(\tau;A^*)\geqslant \sigma \tau$, whence $\rho^w_\xi(\sigma;A)\leqslant \sigma\tau$.  Therefore it follows that $\lim\sup_{\sigma\to 0^+} \rho^w_\xi(\sigma;A)/\sigma\leqslant \tau$.  Since $\tau>0$ was arbitrary, $\lim_{\sigma\to 0^+}\inf \rho^w_\xi(\sigma;A)/\sigma=0$.

The remainder of this section is devoted to the proof of this proposition, for which we require some preparation.  

\begin{lemma} Fix  $\ee>0$ and a $w^*$-null net $(z^*_\lambda)_{\lambda\in Q}\subset Y^*$ such that $\|A^*z^*_\lambda\|\geqslant \ee$ for all $\lambda\in Q$.  Then for any $\eta>0$, any $U\in M$, any $V\in N$, and any $\lambda_0\in Q$, there exist $x\in B_X\cap U$ and $\lambda\in Q$ such that $z^*_\lambda\in V$, $\text{\emph{Re}\ }z^*_\lambda(Ax)>\ee/2-\eta$, and $\lambda_0\leqslant \lambda$. 

\label{lemma1}
\end{lemma}

\begin{proof}  We may assume $U$ is convex and symmetric. For every $\lambda\in Q$, fix $x_\lambda\in B_X$ such that $A^*z^*_\lambda(x_\lambda)>\ee-\eta$.   By passing to a further subnet which converges $w^*$ in $X^{**}$, we may assume that for every $\lambda_1, \lambda_2\in Q$, $x_{\lambda_2}-x_{\lambda_1}\in U$.    Fix $\lambda_1\in Q$ and then fix $\lambda_2\geqslant \lambda_0$ such that $|z^*_{\lambda_2}(Ax_{\lambda_1})|<\eta$ and $z^*_{\lambda_2}\in V$.  Then $$ \text{Re\ }z^*_{\lambda_2}\bigl(\frac{Ax_{\lambda_2}-Ax_{\lambda_1}}{2}\bigr) > (\ee-\eta)/2 - |z^*_{\lambda_2}(Ax_{\lambda_1})|/2>\ee/2-\eta.$$    Thus $(x_{\lambda_2}-x_{\lambda_1})/2$ is the desired $x$ and $\lambda_2$ is the desired $\lambda$.

\end{proof}

\begin{lemma} Fix an ordinal $\xi>0$.  Suppose $(z^*_t)_{t\in R_\xi}\subset Y^*$ is a collection which is $(A^*, 1)$-large and normally $w^*$-null.  Then for any $\ee>0$, there exist a normally weakly null collection $(x_t)_{t\in R_\xi}\subset B_X$ and a length-preserving pruning $\theta:R_\xi\to R_\xi$ such that for every $t\in R_\xi$, $\text{\emph{Re\ }}z^*_{\theta(t)}(Ax_t)>1/2-\ee$.  Moreover, for any sequence $(\ee_n)$ of positive numbers, any $W\in M$, and any sequence $(W_n)\subset N$, these choices may be made such that \begin{enumerate}[(i)]\item $(x_t)_{t\in R_\xi}\subset W$, \item for each $t\in R_\xi$, $z^*_{\theta(t)}\in W_{|t|}$, \item for any $\varnothing\prec s\prec t\in R_\xi$, $|z^*_{\theta(t)}(Ax_s)|, |z^*_{\theta(s)}(Ax_t)|<\ee_{|t|}$. \end{enumerate}

\label{lemma2} 

\end{lemma}

\begin{proof} We work by induction on $\xi$. 

The base and successor cases are similar, so we perform them simultaneously.  Suppose that $\xi=\zeta+1$, and if $\zeta>0$, assume the result holds for $\zeta$.  Note that $(z^*_{(\xi, U,V)})_{(U,V)\in D}$ is $w^*$-null and $\|A^*z^*_{(\xi, U,V)}\|\geqslant 1$.  By Lemma \ref{lemma1}, for any $(U,V)\in D$, there exist $(U', V')\in D$ and $u_{(\xi, U', V')}\in U\cap W\cap B_X$ such that $U'\subset U$, $V'\subset V$, $\text{Re\ }z^*_{(U', V')}(u_{(\xi, U', V')})>1/2-\ee$, and $z^*_{(\xi, U', V')}\in V\cap W_1$.   Let $\theta_0((\xi, U,V))=(\xi, U', V')$.    If $\zeta=0$, we let $\theta=\theta_0$ and $x_t=u_t$, which finishes this case.  Otherwise, fix $(U,V)\in D$ and apply the inductive hypothesis to $(z^*_{\theta_0(\xi, U,V)\cat t})_{t\in R_\zeta}$ to find a length-preserving pruning $\theta_{U,V}:R_\zeta\to R_\zeta$, a normally weakly null collection $(u^{U,V}_t)_{t\in R_\zeta}$ such that for every $t\in R_\zeta$, $$\text{Re\  }z^*_{\theta_0(\xi, U,V)\cat \theta^{U,V}(t)}(u^{U,V}_t)>1/2-\ee,$$    and $(z^*_{\theta_0(\xi, U,V)\cat \theta_{U,V}(t)})_{t\in R_\zeta}$ is normally $w^*$-null.    For each $t\in R_\zeta$, let $u_{(\xi, U,V)\cat t}=u^{U,V}_t$ and $\theta_0((\xi, U, V)\cat t)=\theta_0((\xi, U,V))\cat \theta_{U,V}(t)$.  This defines a normally weakly null collection $(u_t)_{t\in R_\xi}$ and $\theta_0:R_\xi\to R_\xi$ such that for every $t\in R_\xi$, $\text{Re\ }z^*_{\theta(t)}(u_t)>1/2-\ee$ and such that $(z^*_{\theta_0(t)})_{t\in R_\xi}$ is normally $w^*$-null.  We next define a pruning $\phi:R_\xi\to R_\xi$ for $t\in R_\xi$ by induction on $|t|$ such that for each $t=((\zeta_i, U_i, V_i))_{i=1}^n\in R_\xi$, $\phi(t)=((\zeta_i, U_i', V_i'))_{i=1}^n$ for some $U_i'\subset U_i$ and $V_i'\subset V_i$.    Given $t\in R_\xi$ with length $1$, let $\phi(t)=t$.    Next, suppose $t=((\zeta_i, U_i, V_i))_{i=1}^n\in R_\xi$ and $\phi(t^-)=((\zeta_i, U_i', V_i'))_{i=1}^{n-1}$ has been defined.    Let $$U_n'=U_n\cap W\cap \{x\in X: (\forall \varnothing\prec s\prec t)(|z^*_{ \theta_0\circ\phi(s)}(x)|< \ee_n)\}$$ and $$V_n'=V_n\cap W_n\cap \{y^*\in Y^*: (\forall \varnothing\prec s\prec t)(|y^*(u_{\phi(s)})|<\ee_n)\}.$$  Let $x_t=u_{\phi(t)}$ and let $\theta= \theta_0\circ\phi$.   This completes the successor case.

Last, assume $\xi$ is a limit ordinal.  Then for every $\zeta<\xi$, we apply the result to $(z^*_t)_{t\in R_{\zeta+1}}$ to obtain $(x_t)_{t\in R_{\zeta+1}}$ and $\theta_\zeta:R_{\zeta+1}\to R_{\zeta+1}\subset R_\xi$.   Then let $\theta|_{R_{\zeta+1}}=\theta_\zeta$.

\end{proof}

\begin{proposition}   Suppose that for some $\xi$, $(y^*_t)_{t\in MAX(S_\xi)}$ is given, and contained in a $w^*$-compact subset $K$ of $Y^*$. Then there exists a collection $(y^*_t)_{t\in S_\xi'}\subset K$ and a length-preserving pruning $\theta:S_\xi\to S_\xi$ such that $(y^*_{\theta(t)})_{t\in S_\xi}$ is normally $w^*$-closed.   

\label{tedious}
\end{proposition}

\begin{proof} We work by induction.   The $\xi=0$ case is trivial, since we may take $\theta:S_0\to S_0$ to be the identity.  Assume the result holds for $\zeta$ and suppose $\xi=\zeta+1$.  Fix $(y^*_t)_{t\in MAX(S_\xi)}\subset K$ as in the statement.  For each $U\in M$ and $V\in N$, we may apply the inductive hypothesis to $(x_{(\xi,U,V)\cat t})_{t\in R_\zeta}$ and $(y^*_{(\xi, U,V)\cat t})_{t\in MAX(S_\zeta)}$ to obtain a pruning $\theta_{U,V}:S_\zeta\to S_\zeta$ and a collection $(y^*_{(\xi, U,V)\cat t})_{t\in S_\zeta}$ such that $(y^*_{(\xi, U,V)\cat \theta_{U,V}(t)})_{t\in S_\zeta}$ is normally $w^*$-null.   Let $y^*_\varnothing$ be any $w^*$-limit of a $w^*$-convergent subnet of $(y^*_{(\xi, U,V)})_{(U,V)\in M\times N}$.   This means that for any $(U,V)\in M\times N$, there exists $(U', V')\in M\times N$ such that $U'\subset U$, $V'\subset V$, and $y^*_{(\xi, U',V')}-y^*_\varnothing\in V$.  For any $t\in S_\zeta$, let $\theta(\varnothing)=\varnothing$ and $$\theta((\xi, U,V)\cat t)=(\xi, U', V')\cat \theta_{U',V'}(t).$$

Assume that $\xi$ is a limit ordinal and the result holds for every $\zeta<\xi$.  Then applying the inductive hypothesis to $(x_t)_{t\in R_{\zeta+1}}$ and $(y^*_t)_{t\in MAX(S_{\zeta+1})}$, we obtain $(y^{*, \zeta}_t)_{t\in S_{\zeta+1}}$ and a pruning $\theta_\zeta:S_{\zeta+1}\to S_{\zeta+1}$ satisfying the conclusions.  Let $y^*_\varnothing$ be any $w^*$-limit of a $w^*$-convergent subnet of $(y^{*, \zeta}_\varnothing)_{\zeta<\xi}$.    Let $\theta(\varnothing)=\varnothing$.  Fix $U\in M$, $V\in N$, and $\zeta<\xi$.   Let $W\in N$ be convex, symmetric with $2W\subset V$. Note that there exists $\eta_\zeta<\xi$ such that $\zeta<\eta_\zeta$ and $y^{*, \eta_\zeta}_\varnothing- y^*_\varnothing\in W$.  Fix a length-preserving pruning $\phi_\zeta:R_\zeta\to R_{\eta_\zeta}$. Let $\theta((\zeta+1, U, V))=(\eta_\zeta+1, U, W)$ and for $t\in R_\zeta$, let $$\theta((\zeta+1, U, V)\cat t)=\theta_{\eta_\zeta}((\eta_\zeta+1, U, W)\cat \phi_\zeta(t)).$$   It is straightforward to check that the conclusions are satisfied by this construction.

\end{proof}

\begin{corollary} Suppose $\sigma, \ee, \mu>0$ and $(\ee_n)$ is a sequence of positive numbers.   Suppose $0<\xi$, $y\in Y$, and $(u_t)_{t\in R_\xi}\subset B_X$ is a normally weakly null tree such that for every $t\in R_\xi$ and every convex combination $u$ of $(u_s: \varnothing\prec s\preceq t)$, $\|y+\sigma A u\|\geqslant \mu$.  Then there exist collections $(x_t)_{t\in R_\xi}\subset B_X$, $(y^*_t)_{t\in S_\xi}\subset B_{Y^*}$ such that \begin{enumerate}[(i)]\item $(x_t)_{t\in R_\xi}$ is normally weakly null, \item $(y^*_t)_{t\in S_\xi}$ is normally $w^*$-closed,  \item for every $s,t\in R_\xi$ with $s\preceq t$, $\text{\emph{Re}\ }y^*_t(y+\sigma A x_s)\geqslant \mu$.  \end{enumerate}

Moreover, we may make these choices such that for every $s\prec t$, $s,t\in S_\xi$, $|A^*y^*_s(x_t)|<\ee_{|t|}$ and for every $t\in R_\xi$, $|(y^*_t- y^*_{t^-})(y)|<\ee_{|t|}$.   
\label{thiscorollary}
\end{corollary}

\begin{proof}  By the Hahn-Banach theorem, for every $t\in MAX(S_\xi)$, we may find $z^*_t\in B_{Y^*}$ such that for every $\varnothing\prec s\preceq t$, $\text{Re\ }z^*_t(y+\sigma Au_s)\geqslant \mu$.   By Proposition \ref{tedious}, there exists a length-preserving pruning $\phi:S_\xi\to S_\xi$ and a collection $(z^*_t)_{t\in S_\xi}$ such that $(z^*_{\phi(t)})_{t\in S_\xi}$ is normally $w^*$-closed.  Note that $(u_{\phi(t)})_{t\in R_\xi}$ is still normally weakly null, since $\phi$ is a pruning.    Note that with this choice, $\text{Re\ }z^*_t(y+\sigma A u_s)\geqslant \mu$ for every $\varnothing\prec s\preceq t$ from the $w^*$-closedness. Indeed, to obtain a contradiction, suppose that $\varnothing\prec s\preceq t$ are such that $\text{Re\ }z^*_t(y+\sigma Au_s)<\mu$. By well-foundedness, by replacing $t$ with a proper extension if necessary, we may assume that there is no proper extension $v$ of $t$ such that $\text{Re\ }z^*_v(y+\sigma Au_s)<\mu$.  If $t$ is maximal in $R_\xi$, then $z^*_t(y+\sigma Au_s)\geqslant \mu$ by our choice of $z^*_t$.  Thus $t$ is not maximal, whence $z^*_t\in \overline{\{z^*_v: t\prec v\}}^{w^*}$.  But since every proper extension $v$ of $t$ satisfies $\text{Re\ }z^*_v(y+\sigma Au_s)\geqslant \mu$, and since $z^*_t$ lies in the $w^*$-closure of these functionals, $\text{Re\ }z^*_t(y+\sigma Au_s)\geqslant \mu$.    

We next define a length-preserving $\phi_0:S_\xi\to S_\xi$ as in the proof of Lemma \ref{lemma2} and let $\theta= \phi\circ \phi_0$, $x_t=u_{\theta(t)}$, $y^*_t=z^*_{\theta(t)}$ to obtain the properties in the last line of the corollary.

\end{proof}

Before stating the next results, we require some facts concerning ordinals which can be found in \cite{Monk}.  If $\xi$ is any ordinal, then there exist $k\in \nn$, $\gamma_1>\ldots >\gamma_k$, and non-negative integers $n_1, \ldots, n_k$ such that $\xi=\omega^{\gamma_1}n_1+\ldots + \omega^{\gamma_k}n_k$.  Moreover, if $\xi$ is non-zero, and if we omit any summands such that $n_i=0$, this is the unique \emph{Cantor normal form} of $\xi$.    By including zero terms in the expressions above, if $\xi, \zeta$ are any two ordinals, there exist $k\in \nn$, $\gamma_1>\ldots >\gamma_k$, and non-negative natural numbers $m_1, n_1, \ldots, m_k, n_k$ such that $\xi=\omega^{\gamma_1}m_1+\ldots +\omega^{\gamma_k}m_k$ and $\zeta=\omega^{\gamma_1}n_1+\ldots +\omega^{\gamma_k}n_k$.  Then the \emph{Hessenberg sum} of $\xi$ and $\zeta$, denoted $\xi\oplus \zeta$, is given by $$\xi\oplus \zeta=\omega^{\gamma_1}(m_1+n_1)+\ldots +\omega^{\gamma_k}(m_k+n_k).$$   This is independent of the particular representations of $\xi$ and $\zeta$, since such representations differ only in zero summands.     It is clear that for any $\xi, \zeta$, $(\xi+1)\oplus \zeta= (\xi\oplus \zeta)+1$.    Note also that for any ordinal $\zeta$, the set of all pairs $(\xi_0, \xi_1)$ such that $\xi_0\oplus \xi_1=\zeta$ is finite.  

We require the following result from \cite{Causey}.    

\begin{proposition} Let $\xi$ be a limit ordinal.  Suppose that for every $\zeta<\xi$, there exist ordinals $\xi_{0, \zeta}, \xi_{1, \zeta}$ such that $\xi_{0, \zeta}\oplus \xi_{1, \zeta}=\zeta+1$.   Then there exist a subset $S$ of $[0, \xi)$, ordinals $\xi_0, \xi_1$ such that $\xi_0\oplus \xi_1=\xi$, and $\ee\in \{0,1\}$ such that $\xi_\ee$ is a limit ordinal, $$\sup_{\zeta\in S} \xi_{\ee, \zeta}\geqslant \xi_\ee, \text{\ and\ }\min_{\zeta\in S}\xi_{1-\ee, \zeta} \geqslant \xi_{1-\ee}.$$  

\label{proximity}
\end{proposition}

\begin{lemma} Suppose $(y^*_t)_{t\in S_\xi}$ is a normally $w^*$-closed collection, and that $f:R_\xi\to 2=\{0,1\}$ is a function.  Then there exist ordinals $\xi_0, \xi_1$ with $\xi_0\oplus \xi_1=\xi$, and for each $\ee\in 2$, there exists a pruning $\theta_\ee:S_{\xi_\ee}\to S_\xi$ such that $f\circ \theta_\ee|_{R_{\xi_\ee}}\equiv \ee$, $(y^*_{\theta_\ee(t)})_{t\in R_{\xi_\ee}}$ is normally $w^*$-closed,  and $\varnothing=\theta_\ee(\varnothing)$.

\label{stabilization}
\end{lemma}

\begin{proof} By induction on $\xi$.  The $\xi=0$ case is trivial.  Assume $\xi=\zeta+1$ and the conclusion holds for $\zeta$. Let $H$ denote the set of all triples $(\ee, \xi_0, \xi_1)\in 2\times \ord^2$ such that $\xi_0\oplus \xi_1=\zeta$.  Note that $H$ is finite.     For every pair $(U,V)\in M\times N$, let $f_{U,V}:R_\zeta\to 2 $ be given by $f_{U,V}(t)=f((\xi, U, V)\cat t)$.  Applying the inductive hypothesis to $f_{U,V}$ and $(y^*_{(\xi, U, V)\cat t})_{t\in R_\zeta}$, for every $(U,V)\in M\times N$, there exists a triple $\alpha_{U,V}=(\ee(U,V), \xi_0(U,V), \xi_1(U,V))$ and for each $\ee\in 2$ there exists a pruning $\theta^{U,V}_\ee:S_\zeta\to S_\zeta$ such that $f_{U,V}\circ \theta^{U,V}_\ee\equiv \ee$, $\theta^{U,V}_\ee(\varnothing)=\varnothing$, and $(y^*_{(\xi, U, V)\cat \theta^{U,V}_\ee(t)})_{t\in R_\zeta}$ is normally $w^*$-closed.  For each $\alpha\in H$, let $G_\alpha=\{(U,V)\in M\times N: \alpha_{U,V}=\alpha\}$.    Since $H$ is finite, there exists $\alpha\in H$ such that $G_\alpha$ is cofinal in $M\times N$.  Without loss of generality we assume $\alpha=(0, \xi_0, \xi_1)$.   We will define $\theta_0:S_{\xi_0+1}\to S_\xi$ and $\theta_1:S_{\xi_1}\to S_\xi$ to satisfy the conclusions.

Let $\theta_0(\varnothing)=\varnothing$.  For $(U,V)\in M\times N$, fix $(U', V')\in G_\alpha$ such that $U'\subset U$ and $V'\subset V$.   Let $$\theta_0((\xi_0+1, U, V)\cat t)= (\xi, U', V')\cat \theta^{U',V'}_0(t)$$ for $t\in R_{\xi_0+1}$.    It is easily verified that the requirements on $\theta_0$ are satisfied.   

Next, for each $(U,V)\in M\times N$, fix some $(U', V')\in G_\alpha$ such that $U'\subset U$ and $ V'+V'\subset V$.  Define $\varphi:S_{\xi_1}\to S_{\xi_1}$ by letting $\varphi(\varnothing)=\varnothing$ and $\varphi((\eta, U, V)\cat t)=(\eta, U', V')\cat t$ for each $(\eta, U,V)\cat t\in R_{\xi_1}$.  Let $\theta_1(\varnothing)=\varnothing$ and for $t=(\eta, U, V)\cat t_1\in R_{\xi_1}$, let $$\theta_1(t)= (\xi, U', V')\cat \theta_1^{U',V'}(\varphi(t)).$$   The only requirement to check which is not straightforward is that for any $t=(\eta, U,V)\in S_{\xi_1}$, $y^*_{\theta_1(t)}-y^*_\varnothing\in V$.    To see this, note that since $\varphi(t)=(\eta, U', V')=:t_0$, $(y^*_{(\xi, U', V')\cat \theta^{U',V'}_1(s)})_{s\in S_\zeta}$ is normally $w^*$-closed, and $\theta_1^{U', V'}(\varnothing)=\varnothing$, $$y^*_{\theta_1(t)}- y^*_{(\xi, U', V')}=y^*_{(\xi, U', V')\cat \theta^{U', V'}_1(t_0)}-y^*_{(\xi, U', V')\cat \theta^{U', V'}_1(t_0^-)}\in  V'.$$  Since $(y^*_s)_{s\in S_\xi}$ is normally $w^*$-closed, $y^*_{(\xi, U', V')}-y^*_\varnothing\in V'$.  Therefore $$y^*_{\theta_1(t)}-y^*_\varnothing = y^*_{\theta_1(t)} - y^*_{(\xi, U', V')}+y^*_{(\xi, U', V')}-y^*_\varnothing\in V'+V'\subset V.$$

Last, suppose the result holds for every $\zeta<\xi$, $\xi$ a limit ordinal. Fix $f:R_\xi\to 2$ and $(y^*_t)_{t\in S_\xi}$ normally $w^*$-closed.    Apply the result to the function $f|_{R_{\zeta+1}}:R_{\zeta+1}\to 2$ and $(y^*_t)_{t\in S_{\zeta+1}}$.   By the inductive hypothesis, there exist ordinals $\xi_{0, \zeta}$ and $\xi_{1, \zeta}$ such that $\xi_{0, \zeta}\oplus \xi_{1, \zeta}=\zeta+1$ and for each $\ee\in 2$, a pruning $\theta_{\ee, \zeta}:S_{\xi_{\ee, \zeta}}\to S_{\zeta+1}$ such that $f\circ \theta_{\ee, \zeta}|_{R_{\xi_{\ee, \zeta}}}\equiv \ee$ and $\varnothing=\theta_{\ee, \zeta}(\varnothing)$.   By Proposition \ref{proximity}, there exists a subset $A\subset[0, \xi)$ and ordinals $\xi_0, \xi_1$ with $\xi_0\oplus\xi_1=\xi$ and such that, without loss of generality, $\xi_0$ is a limit ordinal, $\sup_{\zeta\in A}\xi_{0, \zeta}\geqslant \xi_0$, and $\min_{\zeta\in A}\xi_{1, \zeta}\geqslant \xi_1$.    Fix any $\zeta\in A$ and any length-preserving pruning $\phi:S_{\xi_1}\to S_{\xi_{1, \zeta}}$, which we may, since $\xi_1\leqslant \xi_{1, \zeta}$.  Define $\theta_1:S_{\xi_1}\to S_\xi$ by letting $\theta= \theta_{1, \zeta}\circ\phi$. Next, for any $\zeta<\xi_0$, fix $\eta_\zeta\in A$ such that $\zeta+1<\eta_\zeta$ and a length-preserving pruning $\phi_\zeta:S_{\zeta+1}\to S_{\eta_\zeta}$.  Define $\theta_0$ on $S_{\zeta+1}$ by letting $\theta_0|_{S_{\zeta+1}}= \theta_{0, \eta_\zeta}\circ\phi_\zeta$.  This defines $\theta_0$ on all of $S_\xi$.  It is straightforward to check that the conclusions are satisfied in this case.

\end{proof}

\begin{corollary} Fix $\sigma, \mu, \ee>0$, $y\in Y$, and $\xi\in \emph{\textbf{Ord}}$.  Suppose that $(u_t)_{t\in R_{\omega^\xi}}\subset B_X$ is a normally weakly null collection such that for every $t\in R_{\omega^\xi}$ and every convex combination $u$ of $(x_s: \varnothing \prec s\preceq t)$, $\|y+\sigma A u\|\geqslant \mu$.  Then there exist collections $(x_t)_{t\in R_{\omega^\xi}}\subset B_X$, $(y^*_t)_{t\in S_{\omega^\xi}}\subset B_{Y^*}$, and a number $c\in \rr$ such that \begin{enumerate}[(i)]\item $(x_t)_{t\in R_{\omega^\xi}}$ is normally weakly null, \item $(y^*_t)_{t\in S_{\omega^\xi}}$ is normally $w^*$-closed,  \item for every $\varnothing\preceq s\prec t\in R_{\omega^\xi}$, $|y^*_s(Ax_t)|<\ee$,   \item for every $\varnothing\prec s\preceq t$, $\text{\emph{Re\ }}y^*_t(y+\sigma Ax_s)\geqslant \mu$,  \item for any $\varnothing\preceq s\preceq t\in S_{\omega^\xi}$, $|(y^*_t-y^*_s)(y)|<\ee$, \item for every $t\in R_{\omega^\xi}$, $c\leqslant \text{\emph{Re}\ }y^*_t(Ax_t)\leqslant c+\ee$.   \end{enumerate}

\label{big corollary}

\end{corollary}

\begin{proof}  Fix $(\ee_n)\subset (0,1)$ such that $\sum \ee_n<\ee$.   By Corollary \ref{thiscorollary}, after relabeling, we may assume that $(u_t)_{t\in R_{\omega^\xi}}$ is normally weakly null, and that we have a collection $(z^*_t)_{t\in S_{\omega^\xi}}\subset B_{Y^*}$ which is normally $w^*$-closed such that $\text{Re\ }z^*_t(y+\sigma Au_s)\geqslant \mu$ for all $\varnothing\prec s\preceq t\in S_{\omega^\xi}$, if $s\prec t\in S_{\omega^\xi}$, then $|A^*z^*_s(u_t)|<\ee_{|t|}$, and if $\varnothing\prec s\preceq t\in S_{\omega^\xi}$, $|(z^*_t-z^*_{t^-})(Au_s)|<\ee_{|t|}$.  We define $\varphi:S_{\omega^\xi}\to S_{\omega^\xi}$ by $\varphi(\varnothing)=\varnothing$ and $$\varphi(s\cat (\eta, U,V))= \varphi(s)\cat (\eta, U, V\cap \{y^*\in Y^*: |y^*(y)|< \ee_{|s|+1} \}).$$ By replacing $(u_t)_{t\in R_{\omega^\xi}}$ and $(z^*_t)_{t\in S_{\omega^\xi}}$  with $(u_{\varphi(t)})_{t\in R_{\omega^\xi}}$ and $(z^*_{\varphi(t)})_{t\in R_{\omega^\xi}}$, respectively, we may assume that $(u_t)_{t\in R_{\omega^\xi}}$ and $(z^*_t)_{t\in S_{\omega^\xi}}$ satisfy properties $(i)$-$(v)$.

 If $D=\text{diam}\{\text{Re\ }z^*_t(Au_t):t\in R_{\omega^\xi}\}$, we claim that there exists a pruning $\theta:S_{\omega^\xi}\to S_{\omega^\xi}$ such that $\theta(\varnothing)=\varnothing$, $\text{diam}\{\text{Re\ }z^*_{\theta(t)}(Au_{\theta(t)})\}\leqslant D/2$, $(u_{\theta(t)})_{t\in R_{\omega^\xi}}$ is normally weakly null, and $(z^*_{\theta(t)})_{t\in S_{\omega^\xi}}$ is normally $w^*$-closed.   Iterating this claim and composing the resulting prunings will yield a pruning $\phi:S_{\omega^\xi}\to S_{\omega^\xi}$ such that $\phi(\varnothing)=\varnothing$, $(x_t)_{t\in R_{\omega^\xi}}:=(u_{\phi(t)})_{t\in R_{\omega^\xi}}$ is normally weakly null, and $(y^*_t)_{t\in S_{\omega^\xi}}:=(z^*_{\phi(t)})_{t\in S_{\omega^\xi}}$ is normally $w^*$-closed, and $\text{diam}\{y^*_t(Ax_t): t\in R_{\omega^\xi}\}<\ee$, and these collections retain properties $(iii)$-$(v)$.  Then if $c$ is the infimum of this set, the collections $(x_t)_{t\in R_{\omega^\xi}}$ and $(y^*_t)_{t\in S_{\omega^\xi}}$ satisfy the conclusions.  We prove this claim.  Let $I=[a,b]$, where $a=\inf\{\text{Re\ }y^*_t(Ax_t): t\in R_{\omega^\xi}\}$, $a'= \sup \{\text{Re\ }y^*_t(Ax_t):t\in R_{\omega^\xi}\}$, and $2b=a+a'$.   Let $f:R_{\omega^\xi}\to 2$ be given by $f(t)= 1_I(\text{Re\ }y^*_t(Ax_t))$.  Then by Lemma \ref{stabilization}, since the only pairs $(\xi_0, \xi_1)$ with $\xi_0\oplus\xi_1=\omega^\xi$ are $(0, \omega^\xi)$ and $(\omega^\xi, 0)$, there exists $\ee\in 2$ and a monotone map $\theta:S_{\omega^\xi}\to S_{\omega^\xi}$ such that $f\circ \theta|_{R_{\omega^\xi}}\equiv \ee$, $(x_{\theta(t)})_{t\in R_{\omega^\xi}}$ is normally weakly null, and $(y^*_{\theta(t)})_{t\in S_{\omega^\xi}}$ is normally $w^*$-closed.  Moreover, the set $\{\text{Re\ }y^*_{\theta(t)}(Ax_{\theta(t)}): t\in R_{\omega^\xi}\}$ is contained in $I$ if $\ee=1$ and in $[a,a']\setminus I$ if $\ee=0$. Since each of these sets has diameter not exceeding half of the diameter of $[a,a']$, this gives the claim.

\end{proof}

\begin{proposition} If  $\sigma, \tau>0$, $\xi\in \textbf{\emph{Ord}}$ are such that $\delta_\xi^{w^*}(\tau;A^*)\geqslant  \sigma\tau$, if $T$ is a tree with $o(T)=\omega^\xi+1$, and if $(y^*_t)_{t\in T}\subset B_{Y^*}$ is $w^*$-closed and $(A^*, c)$-separated for some $c\geqslant \tau$, then either $y^*_\varnothing=0$ or $\|y^*_\varnothing\|\leqslant 1-\sigma c$.    

\label{finish prop}
\end{proposition}

\begin{proof} Let $c$ and $(y^*_t)_{t\in T}$ be as in the statement.  Assume $y^*_\varnothing\neq 0$. For $t\in B=T\setminus \{\varnothing\}$, let $z^*_t=\frac{1}{c}(y^*_t-y^*_{t^-})$. Note that $(z^*_t)_{t\in B}$ is $(A^*, 1)$-large and $w^*$-null.  From this, it follows that  $$\sup_{t\in B}\|\frac{y^*_\varnothing}{\|y^*_\varnothing\|} + \tau\sum_{\varnothing \prec s\preceq t} z^*_s\| \geqslant 1+\sigma \tau,$$ and \begin{align*} (1+\sigma \tau)\|y^*_\varnothing\| & \leqslant \sup_{t\in B} \|y^*_\varnothing + \tau\sum_{\varnothing\prec s\preceq t}\|y^*_\varnothing\| z^*_s\| = \sup_{t\in B} \|y^*_\varnothing + \frac{\|y^*_\varnothing\|\tau}{c}(y^*_t-y^*_\varnothing)\| \\ &  = \|\bigl(1-\frac{\tau\|y^*_\varnothing\|}{c}\bigr) y^*_\varnothing + \frac{\tau\|y^*_\varnothing\|}{c} y^*_t\| \leqslant \bigl(1-\frac{\tau\|y^*_\varnothing\|}{c}\bigr) \|y^*_\varnothing\|+ \frac{\tau\|y^*_\varnothing\|}{c}. \end{align*}
 
Rearranging this inequality gives $\|y^*_\varnothing\|\leqslant 1-c\sigma$.

\end{proof}

\begin{proof}[Proof of Proposition \ref{dualprop}]$(i)$ Suppose not.  Then there exist $\mu<\sigma \tau$, $y^*\in Y^*$ with $\|y^*\|\geqslant 1$, and a $w^*$-null, $(A^*,1)$-large collection $(z^*_t)_{t\in R_{\omega^\xi}}$ such that for every $t\in R_{\omega^\xi}$, $$\|y^*+6\tau \sum_{\varnothing \prec s\preceq t} z^*_s \| \leqslant 1+\mu.$$  Since $\frac{1+3\sigma\tau}{1+\mu}>1+\sigma\tau$, we may fix $\ee>0$ such that $\frac{1+3\sigma\tau-20\ee}{1+\mu}>1+\sigma\tau$.  Fix a decreasing sequence $(\ee_n)$ of positive numbers such that $\sum \ee_n<\ee$. Fix $y\in S_Y$ such that $y^*(y)>1-\ee$.  We apply Lemma \ref{lemma2} with $W=\{x\in X: |A^*y^*(x)|<\ee\}$, $W_n=\{z^*\in Y^*: |z^*(y)|<\ee_n\}$ to obtain $(x_t)_{t\in R_{\omega^\xi}}\subset B_X$ and a length-preserving $\theta:R_{\omega^\xi}\to R_{\omega^\xi}$ such that $(z^*_{\theta(t)})_{t\in R_{\omega^\xi}}$ is normally $w^*$-null and for any $\varnothing\prec s\prec t\in R_{\omega^\xi}$, $|A^*z^*_{\theta(t)}(x_s)|, |z^*_{\theta(s)}(Ax_t)|<\ee_{|t|}$, and for any $t\in R_{\omega^\xi}$, $\text{Re\ }z^*_{\theta(t)}(Ax_t)\geqslant 1/2-\ee$. By relabeling, we assume $\theta$ is the identity.   For each $t\in R_{\omega^\xi}$, let $y^*_t=y+ 6\tau\sum_{\varnothing\prec s\preceq t} z_s^*$, noting that $\|y^*_t\|\leqslant 1+\mu$.   Then for any $\varnothing\prec s\preceq t$, \begin{align*} \text{Re\ }y^*_t(y+\sigma Ax_s) & \geqslant y^*(y)+6\sigma\tau\text{Re\ }z^*_s(Ax_s)-|y^*(Ax_s)| \\ & -6\sigma\tau\sum_{s\neq v\preceq t} |z^*_v(Ax_s)| - 6\tau\sum_{v\preceq t} |z^*_v(y)|  \\ & > 1-\ee +6\sigma\tau(1/2-\ee) - \ee-6\ee-6\ee \\ & = 1+3\sigma\tau-20\ee.\end{align*}   Therefore for any $t\in R_{\omega^\xi}$ and any convex combination $x$ of $(x_s: \varnothing\prec s\preceq t)$, $$\|y+\sigma Ax\| \geqslant (1+\mu)^{-1}\text{Re\ }y^*_t(y+\sigma Ax)\geqslant \frac{1+3\sigma\tau-20\ee}{1+\mu}>1+\sigma\tau.$$    But since $\rho^w_\xi(\sigma;A)\leqslant \sigma\tau$, it follows that the infimum of $\|y+\sigma Ax\|$ over all such convex combinations is at most $1+\sigma\tau$, a contradiction.

$(ii)$ Suppose that $\rho^w_\xi(\sigma;A)>\mu-1>\sigma\tau$ and $\delta^{w^*}_\xi(\tau;A^*)\geqslant \sigma\tau$.   Fix $\ee>0$ such that $1+\sigma(\tau+2\ee)+\ee<\mu$.  There exists $y\in B_Y$, a normally weakly null collection $(x_t)_{t\in R_{\omega^\xi}}\subset B_X$ such that for every $t\in B$ and every convex combination $x$ of $(x_s: \varnothing\prec s\preceq t)$, $\|y+\sigma Ax\|\geqslant \mu$.  Then by Corollary \ref{big corollary}, there exist $c\in \rr$, a normally weakly null collection $(x_t)_{t\in R_{\omega^\xi}}\subset B_X$, a normally $w^*$-closed collection $(y^*_t)_{t\in S_{\omega^\xi}}$ such that \begin{enumerate}[(i)]\item for any $\varnothing\preceq s\prec t\in R_{\omega^\xi}$, $|A^*y^*_s(x_t)|<\ee$, \item for every $t\in R_{\omega^\xi}$, $c\leqslant \text{Re\ }y^*_t(Ax_t)\leqslant c+\ee$, \item for every $\varnothing\prec s\preceq t\in R_{\omega^\xi}$, $\text{Re\ }y^*_t(y+\sigma Ax_s)\geqslant \mu$, \item for any $\varnothing\preceq s\preceq t\in S_{\omega^\xi}$, $|(y^*_s-y^*_t)(y)|<\ee$.  \end{enumerate}

Consider two cases. 

Case $1$: $c-\ee\geqslant \tau$.  In this case, $(y^*_t)_{t\in S_{\omega^\xi}}$ is $(A^*, c-\ee)$-separated, since for any $t\in R_{\omega^\xi}$, $$\|A^*y^*_t-A^*y^*_{t^-}\|\geqslant \text{Re\ }(A^*y^*_t-A^*y^*_{t^-})(x_t) \geqslant \text{Re\ }y^*_t(Ax_t)- |y^*_{t^-}(Ax_t)|>c-\ee.$$   It follows from Proposition \ref{finish prop} that either  $\|y^*_\varnothing\|\leqslant 1-\sigma(c-\ee)$ or $y^*_\varnothing=0$.    Then for any $t\in R_{\omega^\xi}$, $$\mu\leqslant \text{Re\ }y^*_t(y+\sigma Ax_t) \leqslant \text{Re\ }y^*_\varnothing(y) +\sigma(c+\ee) +\ee \leqslant 1-\sigma(c-\ee)+\sigma(c+\ee)+\ee= 1+2\sigma\ee+\ee<\mu$$ if $\|y^*\|\leqslant 1-\sigma(c-\ee)$, and $$\mu \leqslant \text{Re\ }y^*_t(y+\sigma Ax_t) \leqslant \sigma(c+\ee)+\ee <\mu$$ if $y^*_\varnothing=0$.  Either of these alternatives yields a contradiction.   

Case $2$: $c-\ee<\tau$.    In this case, for any $t\in R_{\omega^\xi}$, $$\mu \leqslant \text{Re\ }y^*_t(y+\sigma Ax_t) \leqslant 1+\sigma(c+\ee) \leqslant 1+\sigma(\tau+2\ee)<\mu,$$ another contradiction.

\end{proof}

\section{Szlenk index and measures of non-compactness}

In this section, we discuss some remarkable recent results from \cite{LPR}.  We show how to combine their results with the main result from \cite{Causey2} to obtain renorming results.  We work with complex scalars, while \cite{LPR} worked with real scalars, but it is easy to modify either set of proofs to work for the other case.    Our proofs also extend their results to uncountable ordinals and to include renormings involving operators.

For a Banach space $X$, recall that a subset $K$ of $X^*$ is called $w^*$-\emph{fragmentable} provided that for every $\ee>0$ and every non-empty, $w^*$-closed subset $L$ of $K$, there exists a $w^*$-open set $V$ such that $V\cap L\neq \varnothing$ and $\text{diam}(V\cap L)<\ee$.  For a fixed Banach space $X$, we let $\Phi$ denote the $w^*$-compact subsets of $X^*$.  We say an assignment $i:\Phi\times (0, \infty)\to \ord\cup\{\infty\}$ is a \emph{fragmentation index} if \begin{enumerate}[(i)]\item  $\sup_{\ee>0}i(K, \ee)<\infty$ if and only if $K$ is $w^*$-fragmentable, \item for every $K,L\in \Phi$ with $K\subset L$, $i(K,\ee)\leqslant i(L, \ee)$ for every $\ee>0$, \item for every $K\in \Phi$ and $0<\delta\leqslant \ee$, $i(K, \ee)\leqslant i(K, \delta)$. \end{enumerate} 
Recall the convention that $\xi<\infty$ for every ordinal $\xi$.   

We say that two fragmentation indices $i,j$ are \emph{equivalent} if there exists $\rho >0$ such that for every $\ee>0$ and $K\in \Phi$, $$i(K, \rho \ee) \leqslant j(K, \ee)\leqslant i(K, \ee/\rho).$$  Note that equivalence is transitive.

Next, we recall the definition of the Szlenk derivation.  If $X$ is a Banach space and $K\subset X^*$ is $w^*$-compact, for $\ee>0$, we let $s_\ee(K)$ denote those $x^*\in K$ such that for all $w^*$-neighborhoods $V$ of $x^*$, $\text{diam}(V\cap K)>\ee$.    We define the transfinite derivations $s_\ee^0(K)=K$, $s_\ee^{\xi+1}(K)=s_\ee(s^\xi_\ee(K))$, and if $\xi$ is a limit ordinal, $s^\xi_\ee(K)=\cap_{\zeta<\xi}s_\ee^\zeta(K)$.  Note that $s_\ee^\xi(K)$ is $w^*$-compact for every $\xi$ and every $\ee>0$.   For each $\ee>0$, we let $Sz(K, \ee)$ denote the minimum ordinal $\xi$ such that $s_\ee^\xi(K)=\varnothing$ if such an ordinal exists, and we write $Sz(K, \ee)=\infty$ otherwise.  We let $Sz(K)=\sup_{\ee>0} Sz(K, \ee)$. 
If $A:X\to Y$ is an operator, we let $Sz(A)=Sz(A^*B_{Y^*})$.     It follows directly from the definition that $(K, \ee)\mapsto Sz(K, \ee)$ is a fragmentation index.   Moreover, it is well-known and easy to see that $(K, \ee)\mapsto Sz(S_\mathbb{F}K, \ee)$ is also a fragmentation index equivalent to $Sz(\cdot, \cdot)$, where $S_\mathbb{F}K=\{t x^*: t\in S_\mathbb{F}, x^*\in K\}$.      

If $\hhh\subset X^{<\nn}$, we define the \emph{weak derivative} of $\hhh$, denoted $(\hhh)_w'$, to be those sequences $t\in \hhh$ such that for any weak neighborhood $U$ of $0$ in $X$, there exists $x\in U$ such that $t\cat (x)\in \hhh$.  We define the higher weak derivatives $(\hhh)_w^\xi$ as usual, and the weak order of $\hhh$ by $o_w(\hhh)=\min\{\xi: (\hhh)_w^\xi=\varnothing\}$, assuming this class is non-empty, and we write $o_w(\hhh)=\infty$ otherwise.  Given $K\subset X^*$ and $\ee>0$, we let $\hhh^K_\ee=\varnothing $ if $K=\varnothing$, and otherwise we let $\hhh^K_\ee$ consist of $\varnothing$ together with those sequences $(x_i)_{i=1}^n\in B_X^{<\nn}$ such that there exists $x^*\in K$ such that for each $1\leqslant i\leqslant n$, $\text{Re\ }x^*(x_i)\geqslant \ee$.    We may then define the fragmentation index $o(\cdot, \cdot)$ by $(K, \ee)\mapsto o_w(\hhh^K_\ee)$.  Conditions (ii) and (iii) of fragmentation index are easily seen to hold, while the main theorem of \cite{Causey1} yields that $o(\cdot, \cdot)$ is indeed a fragmentation index and is equivalent to the Szlenk index.   

We also recall the definitions of  fragmentation and slicing derivations associated with measures of non-compactness, defined in \cite{LPR}.   A function $\eta:\Phi\to [0, \infty)$ is called a \emph{measure of non-compactness} provided that \begin{enumerate}[(i)]\item $\eta(K)=\varnothing$ whenever $K$ is finite, \item $\eta\bigl(\cup_{i=1}^n K_i) = \max_{1\leqslant i\leqslant n} \eta(K_i)$ for any $K_1, \ldots, K_n\in \Phi$, \item for any $K, L\in \Phi$ with $K\subset L$, $\eta(K)\leqslant \eta(L)$, \item there exists a constant $b\geqslant 0$ such that for any $K\in \Phi$ and any $r>0$, $$\eta(K+r B_{X^*})\leqslant \eta(K)+rb.$$ 
  \end{enumerate}

Additionally, we say the measure of non-compactness $\eta$ is \emph{convexifiable} provided that there exists a constant $\kappa>0$ such that for any $K\in \Phi$, $\eta(\overline{\text{co}}^{w^*}(K))\leqslant \kappa \eta(K)$.    We refer to the smallest such $\kappa$ as the \emph{convexifiability constant} of $\eta$.

Given $\ee>0$ and $K\in \Phi$,  we let $[\eta]_\ee'(K)$ consist of those $x^*\in K$ such that for every $w^*$-neighborhood $V$ of $x^*$, $\eta(\overline{V}^{w^*}\cap K)\geqslant \ee$.  The set $[\eta]_\ee'(K)$ is the \emph{fragmentation derivation} of $K$.     We define the higher order derivations as usual, and let $i_{\eta}(K, \ee)=\min\{\xi: [\eta]_\ee^\xi(K)=\varnothing\}$ if such an ordinal exists, and $i_\eta(K, \ee)=\infty$ otherwise.   Recall that a $w^*$-open slice in $X^*$ is a set of the form $\{x^*\in X^*: \text{Re\ }x^*(x)>a\}$, where $x\in X$ and $a\in \rr$.   We let $\langle \eta\rangle_\ee'(K)$ denote those $x^*\in K$ such that for every $w^*$-open slice $S$ containing $x^*$, $\eta(\overline{S}^{w^*}\cap K)\geqslant \ee$. This is the \emph{slicing derivation} of $K$.   We define the higher order derivations as usual.  We remark that $[\eta]_\ee^\xi(K)$ and $\langle \eta\rangle_\ee^\xi(K)$ are $w^*$-compact for every $\ee>0$ and every ordinal $\xi$.  Moreover, if $K$ is convex, so is $\langle \eta\rangle_\ee^\xi(K)$ for every $\ee>0$ and $\xi\in \ord$.   

Recall the definition of the Kuratowski measure of non-compactness on $X^*$.  We let $\alpha(K)$ consist of the infimum of all positive $\ee$ such that there exists a finite set $F\subset X^*$ such that $K\subset \cup_{x^*\in F} (x^*+\ee B_{X^*})$.   It is easily verified that this is a measure of non-compactness.   Moreover, it is easy to see that for $K\in \Phi$, every $\ee_1>\ee>0$, and every ordinal $\xi$, $$[\alpha]_{\ee_1}^\xi(K)\subset s_\ee^\xi(K) \subset [\alpha]_{\ee/4}^\xi(K),$$ from which it follows that $i_\alpha(\cdot, \cdot)$ is a fragmentation index which is equivalent to $Sz(\cdot,\cdot)$.   Similarly, it is straightforward to verify that if $x^*\in [\alpha]_\ee'(K)$, then for any $\ee_1\in (0, \ee)$, there exists a subset $S$ of $K$ such that $\|y^*-x^*\|>\ee_1$ for every $y^*\in S$ and such that $x^*\in \overline{S}^{w^*}$.   Conversely, if there exists a subset $S$ of $K$ such that $x^*\in \overline{S}^{w^*}$ and $\|x^*-y^*\|\geqslant 2\ee$ for every $y^*\in S$, then $x^*\in [\alpha]_\ee'(K)$.  We will use these facts often in the sequel.

We will define other  measures of non-compactness which will be important for us.    For every ordinal $\xi>0$, let $$\alpha^\xi(K)=\inf\{\ee>0: [\alpha]_\ee^\xi(K)=\varnothing\}.$$  These are defined differently than similar measures of non-compactness from \cite{LPR}, but the proofs of the following facts are similar to the proofs given there. We omit those proofs which follow from inessential modifications from the work done there.   The primary difference is our proof that for every ordinal $\xi$, $\alpha^{\omega^\xi}$ is convexifiable.  This was shown in \cite{LPR} for $\xi$ finite.    

\begin{proposition} Let $\xi>0$ be an ordinal.  \begin{enumerate}[(i)]\item The function $\alpha^\xi$ is a measure of non-compactness. \item For any $\ee_1>\ee>0$ and any $w^*$-compact set $K$, $$[\alpha^\xi]_{\ee_1}'(K)\subset [\alpha]_\ee^\xi(K)\subset [\alpha^\xi]_\ee'(K).$$     \item The measure of non-compactness $\alpha^\xi$ is convexifiable if and only if $\xi=\omega^\zeta$ for some $\zeta$. Moreover, there exists a constant $\kappa$ such that for every ordinal $\zeta$, the convexifiability constant of $\alpha^{\omega^\zeta}$ does not exceed $\kappa$.  \end{enumerate}\label{convexifiable}\end{proposition}

\begin{proof}$(i)$ Properties (i) and (iii) of measure of non-compactness are easily checked.   For (ii) and (iv), we note that in \cite{LPR} it was shown that for any $\ee>0$, any $\xi\in \ord$, any $K, K_1, \ldots, K_n$, and any $r>0$, $$[\alpha]_\ee^\xi\bigl(\cup_{i=1}^n K_i\bigr)= \cup_{i=1}^n [\alpha]_\ee^\xi(K_i)$$ and $$[\alpha]_{\ee+r}^\xi(K+r B_{X^*})\subset [\alpha]_\ee^\xi(K)+r B_{X^*}.$$  From these facts it follows that $\alpha^\xi(\cup_{i=1}^n K_i)=\max_{1\leqslant i\leqslant n}\alpha^\xi(K_i)$ and $\alpha^\xi(K+r B_{X^*})\leqslant \alpha^\xi(K)+r$ for any $r>0$ and $K, K_1, \ldots, K_n\in \Phi$.

$(ii)$ We first show by induction on $\xi$ that for any $w^*$-open set $V$, any $w^*$-compact $K$, and any $\ee>0$, $V\cap [\alpha]_\ee^\xi(K)\subset [\alpha]_\ee^\xi(\overline{V}^{w^*}\cap K)$.  The base case and limit ordinal cases are trivial.  Suppose $V\cap [\alpha]_\ee^\xi(K)\subset [\alpha]_\ee^\xi(\overline{V}^{w^*}\cap K)$ and $x^*\in V\cap [\alpha]_\ee^{\xi+1}(K)$.   Let $U$ be any $w^*$-neighborhood of $x^*$ and let $W$ be a $w^*$-neighborhood of $x^*$ such that $\overline{W}^{w^*}\subset U\cap V$.   Then $\alpha(\overline{W}^{w^*}\cap [\alpha]_\ee^\xi(K))\geqslant \ee$.  Using the inductive hypothesis, we see that $$\overline{W}^{w^*}\cap [\alpha]_\ee^\xi(K) \subset \overline{U}^{w^*}\cap V\cap [\alpha]^\xi_\ee(K) \subset \overline{U}^{w^*}\cap [\alpha]_\ee^\xi(\overline{V}^{w^*}\cap K).$$   From this it follows that $\alpha(\overline{U}^{w^*}\cap [\alpha]_\ee^\xi(\overline{V}^{w^*}\cap K))\geqslant \ee$.  Since $U$ was arbitrary, $x^*\in [\alpha]_\ee^{\xi+1}(\overline{V}^{w^*}\cap K)$.  This yields the claim.  

Now fix $K$ $w^*$-compact and $\ee_1>\ee>0$.  Suppose that for some ordinal $\xi>0$ and $x^*\in X^*$ that $x^*\notin [\alpha^\xi]_\ee'(K)$.  This means there exists a $w^*$-neighborhood $V$ of $x^*$ such that $\alpha^\xi(\overline{V}^{w^*}\cap K)<\ee$.  By the definition of $\alpha^\xi$, $[\alpha]^\xi_\ee(\overline{V}^{w^*}\cap K)=\varnothing$.  By the previous claim, $V\cap [\alpha]_\ee^\xi(K)=\varnothing$, whence $x^*\notin [\alpha]_\ee^\xi(K)$.   This yields that $[\alpha]_\ee^\xi(K)\subset [\alpha^\xi]_\ee'(K)$.   

We next prove by induction on $\xi$ that $[\alpha^\xi]_{\ee_1}'(K)\subset [\alpha]_\ee^\xi(K)$.  First suppose that $x^*\notin [\alpha]_\ee'(K)$.  This means there exists a $w^*$-neighborhood $V$ of $x^*$ such that $\alpha(\overline{V}^{w^*}\cap K)<\ee$.    This means that $$[\alpha]_{\ee_1}^1(\overline{V}^{w^*}\cap K) \subset [\alpha]_\ee^1 (\overline{V}^{w^*}\cap K)=\varnothing,$$  whence $\alpha^1(\overline{V}^{w^*}\cap K)\leqslant \ee<\ee_1$.   Thus $x^*\notin [\alpha^1]_{\ee_1}'(K)$.    Next suppose that $\xi$ is a limit ordinal and the result holds for every $0<\zeta<\xi$.  Note that for any $0<\zeta<\xi$ and for any $w^*$-compact $K$, $\alpha^\xi(K)\leqslant \alpha^\zeta(K)$, whence $[\alpha^\xi]_{\ee_1}'(K)\subset [\alpha^\zeta]_{\ee_1}'(K)$.  Therefore $$[\alpha^\xi]_{\ee_1}'(K)\subset \underset{0<\zeta<\xi}{\bigcap} [\alpha^\zeta]_{\ee_1}'(K) \subset \underset{0<\zeta<\xi}{\bigcap}[\alpha]_\ee^\zeta(K)=[\alpha]_\ee^\xi(K).$$   Next, suppose that for some ordinal $\xi$, $[\alpha^\xi]_{\ee_1}'(K)\subset [\alpha]^\xi_\ee(K)$.   To reach a contradiction, assume that $x^*\in [\alpha^{\xi+1}]_{\ee_1}'(K)\setminus [\alpha]_\ee^{\xi+1}(K)$.     Since $x^*\notin [\alpha]^{\xi+1}_\ee(K)$, there exists a $w^*$-neighborhood $V$ of $x^*$ such that $\alpha(\overline{V}^{w^*}\cap [\alpha]_\ee^\xi(K))<\ee$.  From this it follows that $V\cap [\alpha]_\ee^\xi(K)=\varnothing$.  Fix a $w^*$-neighborhood $U$ of $x^*$ such that $\overline{U}^{w^*}\subset V$.  Since $x^*\in [\alpha^{\xi+1}]_{\ee_1}'(K)$, we see that $\alpha^{\xi+1}(\overline{U}^{w^*}\cap K)\geqslant \ee_1>\ee$, whence $[\alpha]^{\xi+1}_\ee(\overline{U}^{w^*}\cap K)\neq \varnothing$.    But  $$[\alpha]^{\xi+1}_\ee(\overline{U}^{w^*}\cap K) \subset \overline{U}^{w^*}\cap [\alpha]_\ee^{\xi+1}(K)\subset V\cap [\alpha]_\ee^{\xi+1}(K)=\varnothing,$$ a contradiction.

$(iii)$ First we show that if $\xi=\omega^\zeta$, $\alpha^\xi$ is convexifiable.  For a fragmentation index $i:\Phi\times (0, \infty)\to \ord\cup\{\infty\}$, let $\overline{i}(K, \ee)=i(S_\mathbb{F}K, \ee)$.  As we have already mentioned,  $Sz(\cdot, \cdot)$ is equivalent to $\overline{Sz}(\cdot, \cdot)$, and $Sz(\cdot, \cdot)$, $i_\alpha(\cdot, \cdot)$, and $o(K, \ee):=o_w(\hhh^K_\ee)$ are all equivalent.  By transitivity, each of these fragmentation indices is equivalent to $\overline{i}_\alpha(\cdot, \cdot)$ and $\overline{o}(\cdot, \cdot)$.  Then there exists a constant $\rho>1$ such that if $i, j$ are any two of these six fragmentation indices $Sz, \overline{Sz}, i_\alpha, \overline{i}_\alpha, o, \overline{o}$, for any $K\in \Phi$, and $\ee>0$, $i(K, \ee\rho)\leqslant j(K, \ee)$.   By Corollary $4.3$ of \cite{Causey2} and the remarks following it, if $K$ is balanced and $\ee_1>\ee>0$, if $o(\overline{\text{co}}^{w^*}(K), \ee_1)>\omega^\zeta$, then $o(K, \ee)>\omega^\zeta$.   We claim that $\alpha^{\omega^\zeta}$ is convexifiable with convexifiability constant not exceeding $\rho^4$.  We work by contradiction, so assume this is not true.  Then there exists $K\in \Phi$ such that $\alpha^{\omega^\zeta}(\overline{\text{co}}^{w^*}(K))> \rho^4 \alpha^{\omega^\zeta}(K)$.   Fix $\ee>0$ such that $\alpha^{\omega^\zeta}(\overline{\text{co}}^{w^*}(K))>\rho^4 \ee$ and $\ee>\alpha^{\omega^\zeta}(K)$.  This means that $[\alpha]^{\omega^\zeta}_{\rho^4 \ee}(\overline{\text{co}}^{w^*}(K))\neq \varnothing$, so that $$i_\alpha(\overline{\text{co}}^{w^*}(S_\mathbb{F} K), \rho^4 \ee)\geqslant i_\alpha(\overline{\text{co}}^{w^*}(K), \rho^4 \ee)>\omega^\zeta.$$  From this it follows that  $$o(\overline{\text{co}}^{w^*}(S_\mathbb{F}K), \rho^3\ee) \geqslant i_\alpha(\overline{\text{co}}^{w^*}(S_\mathbb{F} K), \rho^4 \ee)>\omega^\zeta.$$   Then since $S_\mathbb{F}K$ is balanced and $\rho>1$, $$o(S_\mathbb{F}K, \rho^2\ee)>\omega^\zeta.$$   Finally, $$i_\alpha(K,\ee)\geqslant o(K, \rho \ee) \geqslant \overline{o}(K, \rho^2 \ee)= o(S_\mathbb{F} K, \rho^2\ee)>\omega^\zeta,$$ which means $[\alpha]^{\omega^\zeta}_\ee(K)\neq \varnothing$.     But this is a contradiction, since $\ee>\alpha^{\omega^\zeta}(K)$, and this contradiction shows that $\alpha^{\omega^\zeta}$ is convexifiable with convexifiability constant not exceeding $\rho^4 $.   

In order to see that $\alpha^\xi$ is not convexifiable when $\xi>0$ is not of the form $\omega^\zeta$, suppose $\xi$, $\zeta$ are such that $\omega^\zeta<\xi<\omega^{\zeta+1}$.   Let $X=C([1, \omega^{\omega^\zeta}])$.  Let $K=\{\ee \delta_\eta: \eta\in [1, \omega^{\omega^\zeta}], |\ee|=1\}$.   Then $\overline{\text{co}}^{w^*}(K)=B_{C([1, \omega^{\omega^\zeta}])}$.  It is well-known and easy to see that $Sz(K, \ee)=\omega^\zeta+1$ for every $\ee\in (0,2)$, so that $[\alpha]_\ee^\xi(K)=\varnothing$ for every $\ee>0$.  This shows that $\alpha^\xi(K)=0$.  But since $Sz(C([1, \omega^{\omega^\zeta}]))=\omega^{\zeta+1}$ \cite{Brooker1}, $\alpha^\xi(\overline{\text{co}}^{w^*}(K))>0$.  This example shows that $\alpha^\xi$ is not convexifiable.

\end{proof}

\begin{lemma}\begin{enumerate}[(i)]\item Fix an operator $A:X\to Y$, $K\subset Y^*$  $w^*$-compact, $V\subset X^*$  $w^*$-open, $\ee_1>\ee>0$.  If $x^*\in V\cap [\alpha]_{\ee_1}^\xi(A^*K)$, then there exists a tree $T$ with $o(T)=\xi+1$ and a collection $(y^*_t)_{t\in T}\subset K$ such that $(y^*_t)_{t\in T}$ is $w^*$-closed and $(A^*, \ee)$-separated, $(A^* y^*_t)_{t\in T}\subset V$, and $A^*y^*_\varnothing=x^*$.  \item If $T$ is a well-founded, non-empty tree, $(x^*_t)_{t\in T}$ is $w^*$-closed, and if $P$ is a finite cover of $T\setminus T'$, there exists a subtree $S$ of $T$ with $o(S)=o(T)$ and $M\in P$ such that $(x^*_s)_{s\in S}$ is $w^*$-closed and $ S\setminus S'\subset M$.  \item Suppose $S$ is a non-empty, well-founded $B$-tree. Suppose $C$ is a compact set of positive numbers and let $b=\max C$.   Suppose $\ee, \delta, R>0$ are such that $\ee/2b>\delta R$.  Suppose also that  $K\subset X^*$ is $w^*$-compact and $C$ is a compact set of positive numbers such that $C^{-1}=\{c^{-1}:c\in C\}$ has diameter not exceeding $\delta$, and $(x^*_\sigma)_{\sigma\in S}\subset CK\cap R B_{X^*}$ is $w^*$-closed and $\ee$-separated.  Then $x^*_\varnothing\in C[\alpha]_{\ee/2b- \delta R}^\xi(K)$.  Here, $CL=\{ c x^*: c\in C, x^*\in L\}$.    \end{enumerate}

\label{szlenk characterization}

\end{lemma}

We will often apply $(iii)$ when $C=\{1\}$.

\begin{rem}\upshape Note that the proof given below of $(ii)$ also works if we replace a $w^*$-closed tree with a weakly null or $w^*$-null $B$-tree.   

\label{remark}
\end{rem}

\begin{proof}$(i)$ Let $N$ denote the $w^*$-open sets in $X^*$ containing $0$. Note that $N$ ordered by reverse inclusion is a directed set.  We prove by induction on $\xi$ that there  exists a collection $(x^*_t)_{t\in T_\xi N}$ satisfying the requirements.  The $\xi=0$ case is trivial since $T_0  N=\{\varnothing\}$.  We take $y^*_\varnothing\in K$ such that $A^* y^*_\varnothing =x^*$.   Assume the result holds for a given $\xi$ and $x^*\in V\cap [\alpha]^{\xi+1}_{\ee_1}(A^*K)$.   For every $U\in N$, there exists $x^{*,U}\in V\cap (x^*+U)\cap [\alpha]^{\xi}_{\varepsilon_1}(A^*K)$ such that $\|x^{*,U}-x^*\|>\ee$.    By the inductive hypothesis, for each $U\in N$ there exists a $w^*$-closed, $(A^*, \ee)$-separated collection $(y^{*,U}_t)_{t\in T_\xi N}\subset K$ such that $(A^*y^{*,U}_t)_{t\in T_\xi N}\subset V$  and $A^*y^{*,U}_\varnothing=x^{*,U}$.   Let $y^*_\varnothing$ be any $w^*$-limit of a $w^*$-converging subnet of $(y^{*,U}_\varnothing)_{U\in N}$,  and for each $U\in N$,  $y^*_{(\xi+1, U)\cat t}=y^{*,U}_t$.   One easily checks that the requirements are satisfied.   Last, assume the result has been shown for every ordinal $\zeta<\xi$.  Then if $\zeta<\xi$, $x^*\in [\alpha]^\xi_{\ee_1}(A^*K)\subset [\alpha]^{\zeta+1}_{\ee_1}(A^*K)$, whence there exists $(y^{*, \zeta}_t)_{t\in T_{\zeta+1} N}$ satisfying the conslusions.  Let $y^*_\varnothing$ be a $w^*$-limit of a $w^*$-converging subnet $(y^{*,\zeta}_\varnothing)_{\zeta\in D}$ of $(y^{*,\zeta}_\varnothing)_{\zeta<\xi}$.   For $t\in T_\xi N\setminus \{\varnothing\}$, let $y^*_t=y^{*, \zeta}_t$, where $\zeta$ is the unique ordinal less than $\xi$ such that $t\in T_{\zeta+1} N$.

$(ii)$ We prove the result for trees $T$ such that $o(T)=\xi+1$ by induction on $\xi$.     The $\xi=0$ case is trivial, since $T=\{\varnothing\}=T\setminus T'$ in this case.    Suppose $o(T)=\xi+1$ where $\xi=\zeta+1$ and the result holds for $\zeta$.   Let $R$ consist of all length $1$ sequences in $T^\zeta$. Since $o(T)>\zeta+1$, $R$ is non-empty.  Since $T$ is closed, $x^*_\varnothing\in \overline{\{x^*_s: s\in R\}}^{w^*}$.   For every $s\in R$, let $T_s$ consist of all those sequences $u$ such that $s\cat u\in T$.   Note that $o(T_s)=\zeta+1$ and the collection $(x^*_{s\cat u})_{u\in T_s}$ is closed.  There exists a subtree $S_s\subset T_s$ with $o(S_s)=\zeta+1$ such that $(x^*_{s\cat u})_{u\in S_s}$ is closed and $M_s\in P$ such that $S_s\setminus S_s'\subset M_s$.    Then with $R_M=\{s\in R: M_s=M\}$, since $\{x^*_s: s\in R\}=\cup_{M\in P} \{x^*_s:s\in R_M\}$, and since this union is finite, there must exist $M\in P$ such that $x^*_\varnothing\in \overline{\{x^*_s: s\in R_M\}}^{w^*}$.  Let $$S=\{\varnothing\}\cup \bigcup_{s\in R_M} \{s\cat u: u\in S_s\}.$$   One easily checks that the conditions are satisfied in this case.   Last, assume $\xi$ is a limit ordinal and the conclusion holds for every $\zeta<\xi$.  Suppose $o(T)=\xi+1$.    Let $R$ consist of all length $1$ sequences in $T$ and for each $s\in R$, let $T_s$ consist of all sequences such that $s\cat u\in T$.   For each $s$, find $M_s$ and $S_s\subset T_s$ as in the successor case.   For each $M\in P$, let $R_M=\{t\in R: M_s=M\}$.  We claim that there exists $M\in P$ such that for all $\zeta<\xi$, $y^*_\varnothing\in \overline{\{y^*_s: s\in R_M\cap T^\zeta\}}^{w^*}$.  If it were not so, for each $M\in P$ there would exist some $\zeta_M<\xi$ such that $y^*_\varnothing\notin \overline{\{y^*_s: s\in R_M\cap T^{\zeta_M}\}}^{w^*}$.  Since $\xi$ is a limit and $P$ is finite, there exists some $\zeta$ such that for each $M\in P$, $\zeta_M<\zeta<\xi$.  Then $\varnothing\in T^{\zeta+1}$, since $\zeta+1<\xi$, but $$y^*_\varnothing \notin \cup_{M\in P}\overline{\{y^*_s: s\in R_M\cap T^{\zeta_M}\}}^{w^*}\supset \cup_{M\in P}\overline{\{y^*_s: s\in R_M\cap T^\zeta\}}^{w^*}=\overline{\{y^*_s: s\in T^\zeta, s^-=\varnothing\}}^{w^*},$$ a contradiction.  Here we have used that the sets $\overline{\{y^*_s: s\in R_M\cap T^\gamma\}}^{w^*}$ decrease as $\gamma$ increases.    Thus there must be such an $M$.   We then let $$S=\{\varnothing\}\cup \bigcup_{s\in R_M}\{s\cat u: u\in S_s\}.$$   Of the conditions to be satisfied by $S$, the only non-trivial one to check is that $y^*_\varnothing\in \overline{\{y^*_s: s\in S^\zeta, s^-=\varnothing\}}^{w^*}$.   But note that $\{y^*_s: s\in S^\zeta, s^-=\varnothing\}$ is equal to $\{y^*_s: s\in R_M\cap T^\zeta\}$, and so the $w^*$-closure of this set contains $y^*_\varnothing$ by construction.  In order to see the equality of the two sets, we note that the length $1$ sequences in $S$, and therefore in $S^\zeta$, lie both in $T^\zeta$ and $R_M$, yielding one inclusion.  For the reverse inclusion, fix $s\in R_M\cap T^\zeta$.  Since $o(S_s)=o(T_s)$ and $s\in T^\zeta$, it follows that $o(S_s)=o(T_s)>\zeta$.   Using the map $u\mapsto s\cat u$ from $S_s$ into $S$, we see that $s=s\cat\varnothing\in S^\zeta$, since $\varnothing\in S_s^\zeta$.   Since $R_M$ consists of length $1$ sequences, $s^-=\varnothing$, which finishes the reverse inclusion.

$(iii)$ First, note that the closed condition and compactness of $CK$ guarantees that $x^*_\sigma\in CK$ for all $\sigma\in S$.  This statement is the base case of the following claim which we prove by induction: For any ordinal $\zeta$ and $\sigma\in S^\zeta$, $x^*_\sigma\in C[\alpha]^\zeta_{\ee/2b-\delta R}(K)$.   The limit ordinal case is trivial.  Suppose $\sigma\in S^{\zeta+1}\subset S^\zeta$ and the result holds for $\zeta$.  Then $x^*_\sigma\in C[\alpha]_{\ee/2b-\delta R}^\zeta(K)$.  Since the collection is closed, $x^*_\sigma$ lies in the $w^*$-closure of $\{x^*_s: s\in D\}$, where $D=\{s\in S^\zeta: s^-=\sigma\}$.   By the $\ee$-separated condition, $\|x^*_\sigma- x^*_s\|>\ee$ for every $s\in D$.  For every $s\in D$, $x^*_s\in C[\alpha]^\zeta_{\ee/2b- \delta R}(K)$, so there exists $c_s\in C$ such that $c_s^{-1}x^*_s\in [\alpha]^\zeta_{\ee/2b-\delta R}(K)$.   We may fix a net $(z^*_\lambda)$ in $\{c^{-1}_s x^*_s: s\in D\}$ and $c\in C$ such that $z^*_\lambda\underset{w^*}{\to} c^{-1}x^*_\sigma$.   Note that for every $\lambda$, if $z^*_\lambda = c^{-1}_s x^*_s$, $$\|c^{-1}x^*_\sigma - c^{-1}_s x^*_s\| \geqslant b^{-1}\|x^*_\sigma-x^*_s\| -|b^{-1}-c^{-1}|\|x^*_\sigma\| - |b^{-1}- c^{-1}_s| \|x^*_s\|> \ee/b- 2 \delta R.$$ It follows that $c^{-1}x^*_\sigma\in [\alpha]^{\zeta+1}_{\ee/2b-\delta R}(K)$.   This yields the inductive claim.  Since $o(S)=\xi+1$, $\varnothing \in S^\xi$, and $(iii)$ follows.

\end{proof}

The next four results were stated in \cite{LPR}.  Some of these results were stated in the case that the Banach space $X$ is separable.   The general cases require only the substitution of Lemma \ref{szlenk characterization} for their Lemmas $5.5$ and $5.6$.

\begin{lemma}\cite[Proposition $3.2$]{LPR} Suppose that $\eta$ is a measure of non-compactness which is convexifiable with convexifiability constant $\kappa$.  Suppose that $\lambda\in (0,1)$, $\ee>0$, and $L\subset X^*$ are such that $L$ is $w^*$-compact, balanced, and radial, and $[\eta]_\ee'(L)\subset \lambda L$.  Then if $M=\overline{\text{\emph{co}}}^{w^*}(L)$, then for any $\ee_1>\ee$, there exists $n\in \nn$ such that $\langle \eta\rangle_{\kappa\ee_1}^n(M)=\varnothing$.

\label{lpr1}
\end{lemma}

\begin{proposition}\cite[Lemma $2.7$]{LPR} Suppose that $\eta$ is a measure of non-compactness.  For any $w^*$-compact, convex $K\subset X^*$ and any $\ee>0$, $$\langle \eta\rangle_\ee'(K)=\overline{\text{\emph{co}}}^{w^*}([\eta]_\ee'(K)).$$   
\label{lpr2}
\end{proposition}

\begin{lemma}\cite[Proposition $5.7$]{LPR} Suppose that $K$ is $w^*$-compact and balanced, and $m\in \nn$, $\ee>0$, $0<\xi\in \emph{\textbf{Ord}}$, $[\alpha]_\ee^{\xi m}(K)=\varnothing$.  Then there exists a balanced, radial, $w^*$-compact set $L$ containing $K$ such that for any $\ee_1>4\ee$, $$[\alpha]_{\ee_1}^\xi(L)\subset \Bigl(1-\frac{1}{32(m+1)}\Bigr)L.$$

\label{lpr3}
\end{lemma}

\begin{corollary} Let $K\subset X^*$ be a $w^*$-compact set with $Sz(K)\leqslant \omega^{\xi+1}$.  For $\ee>0$, let $K_0=\overline{\text{\emph{co}}}^{w^*}(K)$, and given $K_n$, let $K_{n+1}=\overline{\text{\emph{co}}}^{w^*}(s^{\omega^\xi}_\ee(K_n))$.    Then there exists $n=n(\ee)\in \nn$ such that $K_n=\varnothing$.    

\label{main corollary}

\end{corollary}

\section{General renorming results}

We will later need the following elementary observation.  

\begin{proposition} Suppose $A:X\to Y$ is weakly compact, $x^{**}\in X^{**}$, and $V$ is a $w^*$-neighborhood of $x^{**}$.  Then for any $\ee>0$, there exists $x\in X\cap V$ with $\|x\|\leqslant \|x^{**}\|$ such that $\|Ax-A^{**}x^{**}\|<\ee$.  
\label{easy obs}
\end{proposition}

\begin{proof} By homogeneity, we may assume $\|x^{**}\|=1$.   We may also assume $V$ is convex.   By Goldstine's theorem, we may fix a net $(x_\lambda)\subset B_X$ such that $x_\lambda\underset{w^*}{\to}x^{**}$.  We may assume that for all $\lambda$, $x_\lambda\in V$.   By weak compactness of $A$, $Ax_\lambda\underset{w}{\to}A^{**}x^{**}$.  Thus there exists a convex combination $x$ of $(x_\lambda)$ such that $\|Ax-A^{**}x^{**}\|<\ee$, and this $x$ is the one we seek.

\end{proof}

\begin{corollary} If $A:X\to Y$ is a weakly compact operator, then for any ordinal $\xi$, $\rho^w_\xi(\cdot;A:X\to Y)=\rho^{w^*}_\xi(\cdot; A^{**}:X^{**}\to Y)$. \label{obs corollary}  \end{corollary}

\begin{proof} Suppose that there exists a weakly compact operator $A:X\to Y$, an ordinal $\xi$, and a positive number $\sigma$ such that $\rho^w_\xi(\sigma;A:X\to Y)\neq \rho^{w^*}_\xi(\sigma;A^{**}:X^{**}\to Y)$.  It is clear that $\rho^w_\xi(\sigma;A:X\to Y)\leqslant \rho^{w^*}_\xi(\sigma;A^{**}:X^{**}\to Y)$, so our supposition implies that $\rho^w_\xi(\sigma;A:X\to Y)<\rho^{w^*}_\xi(\sigma;A^{**}:X^{**}\to Y)$.  Fix some numbers $\mu, \eta$ such that $$\rho^w_\xi(\sigma;A:X\to Y)+1<\mu<\eta<\rho^{w^*}_\xi(\sigma;A^{**}:X^{**}\to Y)+1.$$      Let $D$ be the set of non-empty, finite subsets of $X^*$ directed by inclusion.  Given $F\in D$, let $$V_F=\{x^{**}\in X^{**}: (\forall x^*\in F)(|x^{**}(x^*)|<|F|^{-1})\}.$$  Arguing as in Proposition $1.1$, there exist $y\in B_Y$ and a collection $(x^{**}_t)_{t\in B_{\omega^\xi} D}\subset B_{X^{**}}$ such that \begin{enumerate}[(i)]\item if $t=t_1\cat (\gamma, F)\in B_{\omega^\xi} D$, $x^{**}_t\in V_F$, \item for every $t\in B_{\omega^\xi} D$ and every $x^{**}\in \text{co}(x^{**}_s: \varnothing\prec s\preceq t)$, $\|y+\sigma A^{**}x^{**}\|\geqslant \eta$.   \end{enumerate}

For each $t\in B_{\omega^\xi} D$, by Proposition \ref{easy obs}, there exists $x_t\in B_X$ such that $x_t\in V_F$ and $\|Ax_t-A^{**}x^{**}_t\|< \eta-\mu$.    Then for any $t\in B_{\omega^\xi} D$ and any convex combination $x=\sum_{\varnothing\prec s\preceq t}a_s x_s$ of $(x_s: \varnothing\prec s \preceq t)$, if $x^{**}=\sum_{\varnothing\prec s\preceq t}a_sx_s^{**}$, $$\|y+\sigma Ax\| \geqslant \|y+\sigma A^{**}x^{**}\|- \sum_{\varnothing \prec s\preceq t} a_s\|Ax_s-Ax_s^{**}\| \geqslant \eta- (\eta-\mu)\sum a_s =\mu>\rho^w_\xi(\sigma;A:X\to Y).$$  Since $o(B_{\omega^\xi})=\omega^\xi$, $(x_t)_{t\in B_{\omega^\xi} D}\subset B_X$ is a weakly null collection, and since we have the lower estimate for every convex combination, we arrive at a contradiction.

\end{proof}

\begin{theorem} For any operator $A:X\to Y$ and any ordinal $\xi$, there exists an equivalent norm $|\cdot|$ on $Y$ such that $A:X\to (Y, |\cdot|)$ is $\xi$-AUS if and only if $Sz(A)\leqslant \omega^{\xi+1}$.    
\label{renorming1}
\end{theorem}

In order to verify Theorem \ref{renorming1}, we first observe the following.  

\begin{lemma} For any ordinal $\xi$,  the following are equivalent. 
 
\begin{enumerate}[(i)]\item $A^*$ is $w^*$-$\xi$-AUC.  \item For every $\ee>0$, there exists $\delta>0$ such that if $T$ is a tree with order $\omega^\xi+1$ and if $(y^*_t)_{t\in T}\subset B_{Y^*}$ is $w^*$-closed and $(A^*, \ee)$-separated, $\|y^*_\varnothing\|\leqslant 1-\delta$.  \end{enumerate} It follows that if $p,q$ are two equivalent norms on $Y$ such that $A^*:(Y^*,p^*)\to X^*$ is $w^*$-$\xi$-AUC, then $A^*:(Y^*, p^*+q^*)\to X^*$ is $w^*$-$\xi$-AUC.   

\end{lemma}

\begin{proof} First, let condition $1$ be that $\delta^{w^*}_\xi(\tau;A)<\infty$ for some $\tau>0$ (or, equivalently, for every $\tau>0$), and let condition $2$ be the condition that there exists a tree $T$ with $o(T)=\omega^\xi+1$, $\ee>0$, and a $w^*$-closed, $(A^*, \ee)$-separated collection $(y^*_t)_{t\in T}\subset B_{Y^*}$.  We claim that conditions $1$ and $2$ are equivalent.  By Lemma \ref{szlenk characterization}$(i)$ and $(iii)$, condition $2$ is equivalent to the existence of some $\ee>0$ such that  $Sz_\ee(A^*B_{Y^*})>\omega^\xi$, or equivalently, $Sz(A)>\omega^\xi$.   We prove the equivalence of conditions $1$ and $2$. If a collection $(y^*_t)_{t\in T}$ exists as in condition $2$, we may fix any $y^*$ with $\|y^*\|=1$ and for each $t\in B=T\setminus\{\varnothing\}$ let $z^*_t=\ee^{-1}(y^*_t-y^*_{t^-})$.  Then $(z^*_t)_{t\in B}$ is $w^*$-null and $(A^*, 1)$-large.  Moreover, for any $t\in B$ and $\tau>0$, $$\|y^*+\tau \sum_{\varnothing\prec s\preceq t}y^*_s\|=\|y^*+(\tau/\ee)(y^*_t-y^*_\varnothing)\|\leqslant 1+2\tau/\ee,$$ whence $\delta^{w^*}_\xi(\tau;A^*)<\infty$.   

Conversely, suppose that $\delta^{w^*}_\xi(\tau;A^*)<\infty$.  Fix $y^*\in Y^*$ and a $w^*$-null, $(A^*, 1)$-large, $(z_t)_{t\in B}$, where $B$ is a $B$-tree with $o(B)=\omega^\xi$ such that $$\sup_{t\in B}\|y^*+\tau\sum_{\varnothing\prec s\preceq t} z^*_s\|=C <\infty.$$  Then with $y^*_\varnothing=C^{-1}y^*$ and $y^*_t=C^{-1}(y^*+\tau\sum_{\varnothing\prec s\preceq t}z^*_s)$ for $t\in B$, $(y^*_t)_{t\in B\cup\{\varnothing\}}\subset B_{Y^*}$ is $w^*$-closed and $\tau/C$-separated.   

This shows that $(i)$ and $(ii)$ are equivalent, and in fact both vacuously true, in the case that condition $1$ fails, which is equivalent to the case that $Sz(A)\leqslant \omega^\xi$. Thus it suffices to prove the equivalence of conditions $(i)$ and $(ii)$ under the assumption that condition $2$ holds.

Suppose $A^*$ is $w^*$-$\xi$-AUC.   Suppose $(y^*_t)_{t\in T}\subset B_{Y^*}$ and $\ee>0$ are as in $(ii)$.   Then by Proposition \ref{finish prop} with $c=\tau=\ee$ and $\sigma=\delta^{w^*}_\xi(\tau;A^*)/\tau>0$, we deduce that either $y^*_\varnothing=0$ or $\|y^*_\varnothing\|\leqslant 1-\sigma \tau=1-\delta_\xi^{w^*}(\tau;A^*)$.  Thus $(ii)$ is satisfied with $\delta=\min \{1,\delta^{w^*}_\xi(\tau;A^*)\}$.

Next suppose that $(ii)$ is satisfied.   Fix $\tau>0$ and fix $\delta\in (0,1/2)$ as in the conclusion of $(ii)$ with $\ee=\tau/2$.    We claim that $\delta_\xi^{w^*}(\tau;A^*)\geqslant \delta/(1-\delta)=:D<1$.  If it were not so, there would exist $y^*\in Y^*$ with $\|y^*\|\geqslant 1$, $\mu<1+D$, and a $w^*$-null, $(A^*, 1)$-large collection $(z^*_t)_{t\in B}$ such that $$\sup_{t\in B}\|y^*+\tau\sum_{\varnothing\prec s\preceq t} z^*_s\|\leqslant \mu<1+D<2,$$ where $B$ is a $B$-tree with $o(B)=\omega^\xi$.   Let $y^*_\varnothing=\mu^{-1} y^*$, $y^*_t=\mu^{-1}(y^*+\tau\sum_{\varnothing\prec s\preceq t}z^*_s)\subset B_{Y^*}$ for each $t\in B$.  Then $(y^*_t)_{t\in B\cup\{\varnothing\}}$ is $w^*$-closed and $(A^*, \tau/2)$-separated, since $\mu<2$.    From this it follows that $\|y^*_\varnothing\|\leqslant 1-\delta$, whence $$1\leqslant \|y^*\|=\mu\|y^*_\varnothing\|\leqslant \mu(1-\delta)<(1+D)(1-\delta)=1,$$ a contradiction.

Next, suppose that $p,q$ are two equivalent norms on $Y$ such that $A^*:(Y^*, p^*)\to X^*$ is $w^*$-$\xi$-AUC.  Fix $C>0$ such that $C^{-1}p^*\leqslant q^*\leqslant Cp^*$.  Suppose that for $\ee>0$, $\delta>0$ is as in $(ii)$ with respect to the norm $p^*$.  Suppose that $T$ is a tree with $o(T)=\omega^\xi+1$ and that $(y^*_t)_{t\in T}\subset B_{Y^*}^{p^*+q^*}$ is $w^*$-closed and $\ee$-separated.  Fix $\mu>0$ such that $2\mu<1/C$.  By passing to a subtree and using Lemma \ref{szlenk characterization}, assume that $a,b$ are such that for every $t\in MAX(T)$, $$a-\mu\leqslant p^*(y^*_t)\leqslant a, \hspace{5mm} b-\mu \leqslant q^*(y^*_t)\leqslant b.$$ Note that $a+b\leqslant 1+2\mu$.   Since $$(1+1/C)a \leqslant a+b \leqslant 1+2\mu,$$ it follows that $a<1$.   Then $(a^{-1}y_t^*)_{t\in T}\subset B_{Y^*}^{p^*}$ is $w^*$-closed and $(A^*, \ee)$-separated.   From this it follows that $p^*(a^{-1}y^*_\varnothing)\leqslant 1-\delta$ and $p^*(y^*_\varnothing)\leqslant a(1-\delta)$.   Then \begin{align*} p^*(y^*_\varnothing) + q^*(y^*_\varnothing) & \leqslant a(1-\delta)+b = a+b-a\delta \\ & \leqslant a+b-\delta (a+b)/(1+C) \\ & \leqslant (1+2\mu)\bigl(1-\frac{\delta}{1+C}\bigr).\end{align*} Since $\mu\in (0, 1/2C)$ was arbitrary, we are done.

\end{proof}

\begin{remark} In the sequel, if $p,q$ are two equivalent norms on a Banach space $Y$, we let $\text{eq}(p,q)$ denote the smallest $C>0$ such that for every $y\in Y$, $p(y)\leqslant Cq(y)$ and $q(y)\leqslant Cp(y)$.

Note that if $p,q$ are two norms on $Y$, then $Q:(Y,p)\oplus_\infty (Y,q)\to Y$ given by $Q(y_1,y_2)=y_1+y_2$ is a quotient map which induces the quotient norm $$|y|=\inf\{\max\{p(y_1), q(y_2)\}: y_1+y_2=y\}.$$  Moreover, the dual norm to this quotient norm is $p^*+q^*$, which is the norm induced by the isomorhpic embedding $Q^*:Y^*\to (Y^*, p^*)\oplus_1 (Y^*, q^*)$, $Q^*y=(y,y)$.    Thus $p^*+q^*$ is the dual norm to an equivalent norm on $Y$.  

Similarly, if $(r_n)$ is a sequence of norms on $Y$ such that $\sup_n \text{eq}(\|\cdot\|_Y, r_n)<\infty$, and if $(\ee_n)$ is a summable sequence of positive numbers, then $r^*=\sum \ee_n r_n^*$ is dual to some equivalent norm $r$ on $Y$.  Indeed, if $Q:(\oplus_n (Y, r_n))_{c_0}\to Y$ is the surjection given by $Q((y_n))=\sum \ee_n y_n$, then the quotient norm on $Y$ induced by $Q$ has $r^*$ as its dual.   

Finally, if $p,q$ are equivalent norms on $Y$, then the norm $r^*=((p^*)^2+(q^*)^2)^{1/2}$ is also dual to the equivalent norm $$r=\inf\{(p(y_1)^2+q(y_2)^2)^{1/2}: y_1+y_2=y\}$$ on $Y$.

\end{remark}

The proof of Theorem \ref{renorming1} is a modification of the main theorem of \cite{LPR} to operators on possibly non-separable domains and to uncountable ordinals.  

\begin{proof}[Proof of Theorem \ref{renorming1}] Assume there exists an equivalent norm $|\cdot|$ on $Y$ such that $A:X\to (Y, |\cdot|)$ is $\xi$-AUS.  We may assume the original norm of $Y$ has this property, since this does not change the Szlenk index of $A$.  Then $A^*:Y^*\to X^*$ is $w^*$-$\xi$-AUC.  By the previous lemma, for any $\ee>0$, there exists $\delta=\delta(\ee)\in (0,1)$ such that if $(y^*_t)_{t\in T}\subset B_{Y^*}$ is $w^*$-closed and $(A^*, \ee/2)$-separated, $\|y^*_\varnothing\|\leqslant 1-\delta$.   If $s_\ee^{\omega^\xi}(A^*B_{Y^*})=\varnothing$ for all $\ee>0$, $Sz(A)\leqslant \omega^\xi<\omega^{\xi+1}$.  Otherwise for any $\ee>0$ such that $s_\ee^{\omega^\xi}(A^*B_{Y^*})\neq \varnothing$ and any $x^*\in s_\ee^{\omega^\xi}(A^*B_{Y^*})$, by Lemma \ref{szlenk characterization}, there exists a tree $T$ with $o(T)=\omega^\xi+1$ and a  $w^*$-closed, $(A^*, \ee/2)$-separated collection $(y^*_t)_{t\in T}\subset B_{Y^*}$ such that $A^*y^*_\varnothing=x^*$.   Then $\|y^*_\varnothing\|\leqslant 1-\delta$, so that $x^*\in (1-\delta)A^*B_{Y^*}$.   We have shown that $s_\ee^{\omega^\xi}(A^*B_{Y^*})\subset (1-\delta)A^*B_{Y^*}$.  By a standard homogeneity argument, for any $n\in \nn$, $s_\ee^{\omega^\xi n}(A^*B_{Y^*})\subset (1-\delta)^n A^*B_{Y^*}$.  It easily follows that $Sz_\ee(A^*B_{Y^*})<\omega^{\xi+1}$.  Therefore $Sz(A)\leqslant \omega^{\xi+1}$.  

If $Sz(A)<\omega^{\xi+1}$, then $Sz(A)\leqslant \omega^\xi$.  This means that for any $\ee>0$, $s_\ee^{\omega^\xi}(A^*B_{Y^*})=\varnothing$.    Therefore by our previous lemma, $A^*$ is already $w^*$-$\xi$-AUC, and it suffices to assume $Sz(A)=\omega^{\xi+1}$.   Let $a=\max\{\|A\|, 1\}$ and fix a sequence $(\ee_n)$ of positive numbers decreasing to $0$ such that $a\sum \ee_n\leqslant 1$ and $s_{\ee_1}^{\omega^\xi}(A^*B_{Y^*})\neq \varnothing$.    For each $n\in \nn$, $K_{n,0}=A^*B_{Y^*}$ and $K_{n,m+1}=\overline{\text{co}}^{w^*}(s_{\ee_n}^{\omega^\xi}(A^*B_{Y^*}))$.  By Corollary \ref{main corollary} and our choice of $\ee_1$, there exists $M_n\in \nn$ such that $K_{n,M_n}\neq \varnothing$ and $K_{n, M_n+1}=\varnothing$.    We let $f:Y^*\to \rr$ be given by $$f(y^*)=\|y^*\|+ \sum_{n=1}^\infty \frac{\ee_n}{M_n} \sum_{m=1}^{M_n} d(A^*y^*, K_{n,m}).$$   Since $f$ is $w^*$-lower semi-continuous and $\|y^*\|\leqslant f(y^*)\leqslant 2\|y^*\|$, we deduce that the Minkowski functional $|\cdot|$ of $C=\{y^*\in Y^*: f(y^*)\leqslant 1\}$ is an equivalent norm on $Y^*$ with unit ball $C$ which is dual to some norm on $Y$.  Arguing as in \cite{LPR}, for $\ee>0$, we deduce the existence of some $\delta_0=\delta_0(\ee)$ such that if $(y^*_t)_{t\in T}\subset C$ is $w^*$-closed and $(A^*, \ee)$-separated with $o(T)=\omega^\xi+1$, $f(y^*_\varnothing)\leqslant 1-\delta_0$.  Since $C\subset B_{Y^*}$ and $f$ is $2$-Lipschitz with respect to the original norm on $Y^*$,  $$f((1+\delta_0/2)y^*_\varnothing) \leqslant f(y^*_\varnothing)+ \delta_0\leqslant 1.$$  Thus it follows that $|y^*_\varnothing|\leqslant (1+\delta_0/2)^{-1}=1-\delta$ for some $\delta>0$.    We have shown that $A^*:(Y^*, |\cdot|)\to X^*$ is $w^*$-$\xi$-AUC, from which it follows that $A:X\to (Y, |\cdot|_0)$ is $\xi$-AUS, where $|\cdot|_0$ is the predual norm to $|\cdot|$.

\end{proof}

\begin{corollary} If $A:X\to Y$ is weakly compact and $Sz(A^*)\leqslant \omega^{\xi+1}$, then there exists an equivalent norm $|\cdot|$ on $X$ such that $A:(X, |\cdot|)\to Y$ is $\xi$-AUC, $A^*:Y^*\to (X^*, |\cdot|)$ is $\xi$-AUS, and $A^{**}:(X^{**}, |\cdot|)\to Y^{**}$ is $w^*$-$\xi$-AUC.   
\label{drew corollarymore}
\end{corollary}

\begin{proof} For $(\ee_n$) chosen as in the proof of Theorem \ref{renorming1}, we let $K_{n,0}=A^{**}B_{X^{**}}$ and $K_{n, m+1}=\overline{\text{co}}^{w^*}(K_{n,m})$.  We define the function $f:X^{**}\to \rr$ by $$f(x^{**})=\|x^{**}\|+\sum_{n=1}^\infty \frac{\ee_n}{M_n}\sum_{m=1}^{M_n} d(A^{**}x^{**}, K_{n,m})$$ and the new norm $|\cdot|$ on $X^{**}$ to be the Minkowski functional of $C=\{x: f(x)\leqslant 1\}$.   We deduce as in the previous theorem that this is an equivalent norm on $X^{**}$ making $A^{**}:(X^{**}, |\cdot|)\to Y^{**}$ $w^*$-$\xi$-AUC.  We define the norm $|\cdot|_0$ on $X$ by $|x|_0=|\Phi x|$, where $\Phi:X\to X^{**}$ is the canonical embedding.   We must show that $|\cdot|_0^{**}=|\cdot|$.  If $C_0$ denotes the unit ball of $|\cdot|_0$ in $X$, by Goldstine's theorem, we only need to show that $\overline{\Phi C_0}^{w^*}=C$.  Of course, $\Phi C_0\subset C$, and since $C$ is $w^*$-compact, $\overline{\Phi C_0}^{w^*}\subset C$.  Since $\{x^{**}\in X^{**}: f(x^{**})<1\}$ is dense in $C$, we need to show that $\Phi C_0$ is $w^*$-dense in $\{x^{**}\in X^{**}: f(x^{**})<1\}$.  To that end, fix $x^{**}\in X^{**}$ with $f(x^{**})=1-\delta<1$.  Fix a $w^*$-neighborhood $V$ of $x^{**}$ and $x\in X$ with $\|x\|\leqslant \|x^{**}\|$ such that $\Phi x\in V$ and $\|A^{**}x^{**}-A x\|<\delta$, which we may do by Proposition \ref{easy obs}.   Then $$f(\Phi x)\leqslant f(x^{**})+\sum_{n=1}^\infty \frac{\ee_n}{M_n}\sum_{m=1}^{M_n}\delta <1-\delta+\delta=1.$$  Thus $x\in C_0$ and $\Phi x\in V$, yielding the $w^*$-density of $\Phi C_0$.

\end{proof}

We have yet another modification, which we will need later.  Given an ordinal $\xi$, a constant $\ee>0$, and a topology $\tau$ on $X$, and an operator $A:X\to Y$, let $d^\tau_\xi(\ee;A)$ be the supremum of all $\delta\geqslant 0$ such that for every tree $T$ with $o(T)=\omega^\xi+1$, every $\tau$-closed, $(A, \ee)$-separated $(x_t)_{t\in T}\subset B_X$, and every $\mu>0$, there exists a subtree $S$ of $T$ with $o(S)=o(T)$ such that $(x_t)_{t\in S}$ is $\tau$-closed and $$\|Ax_\varnothing\|\leqslant (1-\delta+\mu)\underset{t\in MAX(S)}\inf \|Ax_t\|.$$ We observe that by Lemma \ref{szlenk characterization}, the set of all such $\delta$ always includes $0$.  If $A$ is an adjoint and $\tau=w^*$ is the $w^*$-topology on $X$ coming from an understood predual, we say $A:X\to Y$ is \emph{co}-$w^*$-$\xi$-\emph{AUC} provided that $d^{w^*}_\xi(\ee;A)>0$.  We remark here that if $A:X\to Y$ is weakly compact, then whether the operator $A^{**}:X^{**}\to Y$ is co-$w^*$-$\xi$-AUC depends on the norm of $Y$ and is invariant under renorming $X$. 

\begin{theorem}\begin{enumerate}[(i)]\item Suppose $A:X\to Y$ is a weakly compact operator such that $Sz(A^*)\leqslant \omega^{\xi+1}$.  Then there exists an equivalent norm $|\cdot|$ on $Y$ such that $A^{**}:X^{**}\to (Y, |\cdot|)$ is co-$w^*$-$\xi$-AUC.  \item If $A:X\to Y$ $Sz(A)\leqslant \omega^{\xi+1}$, then there exists an equivalent norm $|\cdot|$ on $X$ such that $A^*:Y^*\to X^*$ is co-$w^*$-$\xi$-AUC.

\end{enumerate}

\label{coAUC}
\end{theorem}

\begin{proof}[Sketch]$(i)$ As before, we may assume that $Sz(A^*)=\omega^{\xi+1}$.   Fix a sequence of positive numbers $(\ee_n)$ decreasing to zero such that $a\sum \ee_n<1$ and $s_{\ee_1}^{\omega^\xi}(A^{**}B_{X^{**}})\neq \varnothing$. Let $\Phi:Y\to Y^{**}$ denote the canonical embedding and note that since $A$ is weakly compact, $A^{**}X^{**}\subset Y$.    Let $K_{n, 0}= n A^{**}B_{X^{**}}$ and let $K_{n, m+1}= \overline{\text{co}}^{w^*}s_{\ee_n}^{\omega^\xi}(K_{n,m})$.  By Corollary \ref{main corollary}, for each $n\in \nn$, there exists $M_n\in \nn$ such that $K_{n, M_n}\neq \varnothing$ and $K_{n, M_n+1}=\varnothing$.   We let $f:Y\to \rr$ be given by $$f(y)=\|y\|+\sum_{n=1}^\infty \frac{\ee_n}{M_n}\sum_{m=1}^{M_n} d(y, \Phi^{-1}K_{n,m}).$$   Let $C=\{y\in Y: f(y)\leqslant 1\}$.  Then $C$ is the unit ball of an equivalent norm on $Y$ such that $C\subset B_Y \subset 2C$.   Let $|\cdot|$ denote the Minkowski functional of $C$.   

Fix $\ee>0$, a tree $T$ with $o(T)=\omega^\xi+1$, and a collection $(x_t^{**})_{t\in T}\subset B_{X^{**}}$ as in the definition of $d_\xi^{w^*}(\ee;A^{**}:X^{**}\to (Y, |\cdot|))$.   Fix $\eta\in (0, \ee/2)$.     By passing to a subtree, using Lemma \ref{szlenk characterization}, we may assume that there exists $\mu\geqslant 0$ such that for every $t\in MAX(T)$, $$\mu-\eta \leqslant |Ax_t|\leqslant \mu.$$   For any $t\in MAX(T)$, if $s=t^-$, $$A^{**}x_s^{**}\in \overline{\{Ax_v: v^-=s\}}^w\subset \mu C.$$   Then $$\ee\leqslant |A^{**}x^{**}_s-A^{**}x^{**}_t|\leqslant 2\mu,$$ whence $\mu\geqslant \ee/2$.  Note also that $\mu\leqslant \eta+\inf_{t\in MAX(T)}|A^{**}x_t^{**}|\leqslant  \eta+\sup_{t\in T}|A^{**}x^{**}_t|\leqslant \eta+2\|A\|$.    We then argue as in \cite{LPR} with the tree $(\mu^{-1} x_t^{**})_{t\in T}\subset k B_{X^{**}}$ for a sufficiently large $k\in \nn$ (depending on $\ee$) to deduce that $f(\mu^{-1}A^{**}x^{**}_\varnothing)\leqslant 1-\delta$ for some $\delta=\delta(\ee)>0$.    From this it follows that $$f((1-\delta\ee/2\|A\|)A^{**}x^{**}_\varnothing) \leqslant f(\mu^{-1}A^{**}x^{**}_\varnothing)+\frac{\delta\ee}{2\mu \|A\|}\|A^{**}x^{**}_\varnothing\|\leqslant 1-\delta + \frac{\delta\ee}{2\|A\|}\cdot \frac{2}{\ee}\|A\| \leqslant 1,$$   whence $|A^{**}x^{**}_\varnothing|\leqslant \mu(1-\delta\ee/2\|A\|)^{-1}=\mu(1-\underline{\delta})$ for some constant $\underline{\delta}=\underline{\delta}(\ee)>0$.  From this it follows that $$|A^{**}x^{**}_\varnothing|\leqslant \mu(1-\underline{\delta}) \leqslant \frac{\mu}{\mu-\eta} (1-\underline{\delta})\inf_{t\in MAX(T)}|A^{**}x^{**}_t|.$$ Since the numbers $\delta$ and $\underline{\delta}$ did not depend on $\mu$ or $\eta$, since $\mu\geqslant \ee/2$, and since this process can be completed for arbitrarily small $\eta$, we deduce that $d^{w^*}_\xi(\ee;A^{**}:X^{**}\to Y)\geqslant \underline{\delta}$.    

$(ii)$ The proof is the same as $(i)$, with the $n A^{**}B_{X^{**}}$ replaced by $nA^*B_{Y^*}$.  One needs only note that the resulting set $C$ is $w^*$-compact, so that the norm $|\cdot|$ is the dual norm to a norm on $X$.

\end{proof}

Next, we turn to the use of the Asplund averaging method to show, among other things, that for certain operators from a Banach space into itself, which admits AUC and AUS norms, it admits a norm which is simultaneously AUC and AUS.  We first require a technical piece.

\begin{lemma} Suppose $A:X\to Y$ is an operator and $r_1, \ldots, r_k$ are equivalent norms on $Y$ such that $\text{\emph{eq}}(\|\cdot\|_Y, r_n)\leqslant C$ for each $1\leqslant n\leqslant k$.  Suppose also that $\sigma_0, \tau>0$ are such that for any $0<\sigma\leqslant \sigma_0$ and $1\leqslant n\leqslant k$, $\rho^w_\xi(\sigma C^2;A:X\to (Y, r_n)) \leqslant \sigma\tau$.    Then for any $B$-tree $B$ with $o(B)=\omega^\xi$, any $y\in Y$ with $\|y\|\geqslant 1/C$, any $\ee>0$, and any weakly null collection $(x_t)_{t\in B}\subset B_X$, there exists $t\in B$ and a convex combination $x$ of $\text{\emph{co}}(x_s: \varnothing\prec s\preceq t)$ such that for each $1\leqslant n\leqslant k$, $r_n(y+\sigma Ax) < r_n(y)(1+\sigma\tau+\ee)$.     

\label{crazy lemma}

\end{lemma}

For the proof, we recall some definitions. Given a $B$-tree $T$, we let $\Pi T=\{(s,t)\in T\times T: s\preceq t\}$.     Given $B$-trees $S,T$ and a directed set $D$, we say a pair $(\theta, e)$ of maps with $\theta:S D\to T D$ and $e:MAX(S D)\to MAX(T D)$ is an \emph{extended pruning} provided that $\theta:S D\to T D$ is a pruning and for every $t\in MAX(S D)$, $\theta(t)\preceq e(t)$.  By an abuse of notation, we write $(\theta, e):S D\to T D$.

\begin{proof} Suppose that the conclusion fails, and let $0<\sigma\leqslant \sigma_0$, $B$, $y\in Y$, $(x_t)_{t\in B}$ witness the failure of the claim.    Let $D$ be the set of all weakly open sets in $X$ that contain $0$, ordered by reverse inclusion. By Proposition \ref{normal prop}, we may assume that $B=B_{\omega^\xi} D$. Let $\Gamma_\xi$ be the $B$-tree with $o(\Gamma_\xi)=\omega^\xi$ defined in \cite{Causey2}. It was shown there that $o(\Gamma_\xi)=\omega^\xi$, so that there exists a pruning $\theta:\Gamma_\xi D\to B_{\omega^\xi}D$.  By replacing $x_t$ by $x_{\theta(t)}$ and relabeling, we obtain a collection $(x_t)_{t\in\Gamma_\xi D}$ witnessing the failure of the claim.     

In \cite{Causey2}, a function $\mathbb{P}_\xi:\Gamma_\xi D\to [0,1]$ was defined such that for every $t\in MAX(\Gamma_\xi D)$, $\sum_{\varnothing\prec s\preceq t}\mathbb{P}_\xi(s)=1$.   For each $t\in MAX(\Gamma_\xi D)$, let $$y_t=\sum_{\varnothing\prec s\preceq t}\mathbb{P}_\xi(s)x_s\in \text{co}(x_s: \varnothing\prec s\preceq t).$$   For each $1\leqslant n\leqslant k$, fix $y^*_{n,t}\in B_{Y^*}^{r_n}$ such that $$\text{Re\ }y^*_{n,t}(y+\sigma Ay_t)= r_n(y+\sigma Ay_t).$$   For each $1\leqslant n\leqslant k$ and each $(s,t)\in \Pi(\Gamma_\xi D)$, let $$g_n(s,t)= \frac{\text{Re\ }y^*_{n,t}(y+\sigma Ax_s)}{r_n(y)}.$$  For each $(s,t)\in \Pi (\Gamma_\xi D)$, let $g(s,t)=\max_{1\leqslant n\leqslant k}g_n(s,t)$.   Our contradiction assumption yields that for any $t\in MAX(\Gamma_\xi D)$, there exists some $1\leqslant n\leqslant k$ such that \begin{align*} 1+\sigma\tau+\ee & \leqslant \frac{r_n(y+\sigma Ay_t)}{r_n(y)}   = r_n(y)^{-1} \text{Re\ }y^*_{n,t}(y+\sigma Ay_t) \\ & = r_n^{-1}(y)\text{Re\ }y^*_{n,t}\sum_{\varnothing\prec s\preceq t}\mathbb{P}_\xi(s)(y+\sigma Ax_s)\\ & = \sum_{\varnothing\prec s\preceq t}\mathbb{P}_\xi(s)g_n(s,t) \leqslant \sum_{\varnothing\prec s\preceq t}\mathbb{P}_\xi(s)g(s,t). 
\end{align*}

By \cite{Causey2}, there exists an extended pruning $(\theta, e):\Gamma_\xi D\to \Gamma_\xi D$ such that for every $(s,t)\in \Pi(\Gamma_\xi D)$, $g(\theta(s), e(t))\geqslant 1+\sigma\tau+\ee/2$.     We next define a function $N:\Pi(\Gamma_\xi D)\to \{1, \ldots, k\}$  by letting $N(s,t)$ be the minimum $ n\in \{1, \ldots, k\}$ such that $g(\theta(s), e(t))=g_n(\theta(s), e(t))$.    By \cite[Proposition $4.6(iii)$]{Causey1}, there exists an extended pruning $(\theta', e'):\Gamma_\xi D\to \Gamma_\xi D$ and $n\in \{1, \ldots, k\}$ such that for every $(s,t)\in \Pi(\Gamma_\xi D)$, $N(\theta'(s), e'(t))=n$.    Let $\phi=\theta\circ \theta'$ and $f=e\circ e'$.  Then $(\phi, f):\Gamma_\xi D\to \Gamma_\xi D$ is an extended pruning.  Since $\phi$ is a pruning, it follows that $(x_{\phi(s)})_{s\in \Gamma_\xi D}\subset B_X$ is a weakly null collection.  Fix $(s,t)\in \Pi(\Gamma_\xi D)$ and note that $$1+\sigma\tau+\ee/2\leqslant g(\phi(s), f(t))=g_n(\phi(s), f(t))= \frac{\text{Re\ }y^*_{n,f(t)}(y+\sigma Ax_{\phi(s)})}{r_n(y)}.$$   For this $t\in MAX(\Gamma_\xi D)$, it follows that for every convex combination $x=\sum_{\varnothing \prec s\preceq t}a_sx_{\phi(s)}$ of $\text{co}(x_{\phi(s)}: \varnothing\prec s \preceq t)$, \begin{align*} 1+\sigma\tau+\ee/2 & \leqslant \sum_{\varnothing \prec s\preceq t}a_s \text{Re\ } y^*_{n,f(t)}\Bigl(\frac{y}{r_n(y)}+\frac{\sigma}{r_n(y)} Ax_{\phi(s)}\Bigr) \\ & =\text{Re\ }y^*_{n, f(t)}\Bigl(\frac{y}{r_n(y)}+ \frac{\sigma}{r_n(y)} Ax\Bigr) \leqslant r_n\Bigl(\frac{y}{r_n(y)}+\frac{\sigma}{r_n(y)} Ax\Bigr),\end{align*}   since $y^*_{n,t}\in B_{Y^*}^{r_n}$.   

However, since $\|y\|\geqslant 1/C$, $r_n(y)\geqslant 1/C^2$, so that $\sigma/r_n(y)\leqslant \sigma C^2$.   Then the vector $y/r_n(y)$ and the collection $(x_{\phi(s)})_{s\in \Gamma_\xi D}\subset B_X$ contradict the assumption that $\rho^w_\xi(\sigma C^2; A:X\to (Y, r_n))\leqslant 1+\sigma \tau$.    

\end{proof}

\begin{proposition} Suppose $A:X\to Y$ is an operator.  Suppose also that $r_n$ is a sequence of norms on $Y$ such that $\sup_n \text{\emph{eq}}(\|\cdot\|_Y, r_n)<\infty$.  Fix a sequence $(\ee_n)$ of positive numbers with $\sum \ee_n\leqslant 1$ and let $r=\sum \ee_n r_n$.  \begin{enumerate}[(i)]\item If for some ordinal $\xi$ and every $n\in \nn$, $A:(X, \|\cdot\|_X)\to (Y, r_n)$ is $\xi$-AUS, then $A:(X, \|\cdot\|_X)\to (Y, r)$ is $\xi$-AUS.  \item If $A$ is weakly compact, $\xi$ is an ordinal, and $p$ is an equivalent norm on $Y$ such that $\text{\emph{eq}}(p, r_n)\underset{n\to \infty}{\to} 1$ and $A^{**}:(X^{**}, \|\cdot\|_X)\to (Y, p)$ is co-$w^*$-$\xi$-AUC, then $A^{**}:(X^{**}, \|\cdot\|_X)\to (Y, r)$ is co-$w^*$-$\xi$-AUC.  \item If $X=Y$, $A$ is weakly compact, $p$ is a norm on $X$ such that $A^{**}:(X^{**},p)\to (X, \|\cdot\|)$ is $w^*$-$\xi$-AUC, and $\text{\emph{eq}}(p, r_n)\underset{n\to \infty}{\to} 1$, then $A^{**}:(X^{**}, r)\to (X, \|\cdot\|_X)$ is $w^*$-$\xi$-AUC. \end{enumerate}
\label{Asplund}
\end{proposition}

\begin{remark} Recall that the the properties $\xi$-AUS and co-$w^*$-$\xi$-AUC of the operator $A^{**}$ is invariant under renorming the domain, and the property $w^*$-$\xi$-AUC of $A^{**}$ is invariant under renorming the range. Therefore in Proposition \ref{Asplund}, each statement remains valid if we replace either of the original norms $\|\cdot\|_X, \|\cdot\|_Y$ with any equivalent norms.

\end{remark}

\begin{proof}$(i)$ It is sufficient to check that for any $\tau>0$, there exists $\sigma>0$ such that for any $y\in S_Y^r$, any $B$-tree $B$ with $o(B)=\omega^\xi$, any $\ee>0$, and weakly null collection $(x_t)_{t\in B}\subset B_X$, there exists some $t\in B$ and some convex combination $x\in \text{co}(x_s: \varnothing\prec s\preceq t)$ such that $r(y+\sigma Ax)\leqslant 1+\sigma\tau+\ee$. To that end, suppose $\tau, \ee>0$, $y\in S^r_Y$, $B$, and $(x_t)_{t\in B}$ are fixed as above.    Fix $k\in \nn$ such that $b\sum_{n=k+1}^\infty \ee_n<\tau/2$, where $b>\sup_n \|A:X\to (Y, r_n)\|$.  Note that there exists a constant $C$ such that for any $n\in \nn$, $\text{eq}(\|\cdot\|, r), \text{eq}(\|\cdot\|, r_n)\leqslant C$, so that $\|y\|\geqslant 1/C$.   There exists some number $\sigma_0$ such that for each $1\leqslant n\leqslant k$ and each $0<\sigma\leqslant \sigma_0$, $\rho^w_\xi(\sigma/C^2;A:X\to (Y, r_n))\leqslant \sigma\tau/2$.   By Lemma \ref{crazy lemma}, for any $0<\sigma\leqslant \sigma_0$, there exists $t\in B$ and $x\in \text{co}(x_s: \varnothing \prec s\preceq t)$ such that for each $1\leqslant n\leqslant k$, $r_n(y+\sigma Ax)\leqslant r_n(y)(1+\sigma\tau/2+\ee)$.   Then \begin{align*} r(y+\sigma Ax) & = \sum_{n=1}^k \ee_n r_n(y+\sigma Ax) + \sum_{k=n+1}^\infty \ee_nr_n(y+\sigma Ax) \\ & \leqslant \sum_{n=1}^k \ee_n[r_n(y)(1+\sigma\tau/2+\ee)] +\sum_{n=k+1}^\infty \ee_n(r_n(y)+\sigma b) \\ & \leqslant \sum_{n=1}^\infty \ee_n r_n(y) + \sum_{n=1}^k \ee_n(\sigma\tau/2+\ee)+\sigma b\sum_{n=k+1}^\infty \ee_n \\ & \leqslant 1+\sigma\tau/2+\ee+\sigma\tau/2=1+\sigma\tau+\ee .\end{align*}

$(ii)$ Fix $\ee>0$.    Fix a tree $T$ with $o(T)=\omega^\xi+1$ and a $w^*$-closed $(x^{**}_t)_{t\in T}\subset B_{X^{**}}$ such that $(x^{**}_t)_{t\in T}$ is $(A^{**}, \ee)$-separated (where the norm on $Y$ is $r$).    Fix $C$ such that for all $n\in \nn$, $\text{eq}(r,p), \text{eq}(p, r_n)\leqslant C$.  Let $\delta=d^{w^*}_\xi(\ee/C;A^{**}:X^{**}\to (Y,p))>0$ and $k\in \nn$ such that $\text{eq}(p, r_k)^2 (1-\delta)<1-\delta/2$.  Fix $\mu>0$ such that $\text{eq}(p, r_k)^2(1-\delta+\mu)<1-\delta/2$. Our choice of $C$ yields that  $(x^{**}_t)_{t\in T}$ is $(A^{**}, \ee/C)$-separated in $(Y, p)$, and by the definition of $d_\xi^{w^*}(\ee/C;A^{**}:X^{**}\to (Y,p))$, there exists a subtree $S$ of $T$ such that $o(S)=o(T)$, $(x^{**}_t)_{t\in T}$ is still $w^*$-closed, and \begin{align*} r_k(A^{**}x^{**}_\varnothing) & \leqslant \text{eq}(p, r_k)p(A^{**}x^{**}_\varnothing)\leqslant \text{eq}(p, r_k)(1-\delta+\mu)\underset{t\in MAX(S)}{\inf} p(A^{**}x^{**}_t) \\ & \leqslant (1-\delta/2)\underset{t\in MAX(S)}{\inf}r_k(A^{**}x^{**}_t).\end{align*} Fix a number $b$ which exceeds $\sup_n\|A:X\to (Y,r_n)\|$ and note that $b>\|A:X\to (Y,r)\|$. By passing to a full subtree, we may assume that $0< \underset{t\in MAX(S)}{\inf} r(A^{**}x^{**}_t)$.  Indeed, by Lemma \ref{szlenk characterization}$(ii)$, we may assume that for every $t\in MAX(S)$, for some $a_0\in \rr$, $$a_0\leqslant r(A^{**}x^{**}_t) \leqslant a_0+\ee/4.$$  From this it follows that for any $t\in MAX(S)$, $$\ee\leqslant r(A^{**}x^{**}_t-A^{**}x^{**}_{t^-}) \leqslant 2(a_0+\ee/4)$$ and $a_0>0$.

Next, let $\mu=\delta\ee_k/12C^2$ and fix $m\in \nn$ with $m>k$ such that $Cb\sum_{n=m+1}^\infty \ee_n < a_0\mu$.  By passing to a subtree once more, we may assume that we have $a, a_1, \ldots, a_m\in \rr$ such that for every $1\leqslant n\leqslant m$, $$a_n= \underset{t\in MAX(S)}{\inf} r_n(A^{**}x^{**}_t), \underset{t\in MAX(S)}{\sup}r_n(A^{**}x^{**}_t)\leqslant a_n+\mu a_0,$$ $$ a=\underset{t\in MAX(S)}{\inf} r(A^{**}x^{**}_t), \underset{t\in MAX(S)}{\sup} r(A^{**}x^{**}_t)\leqslant a+\mu a_0.$$   Noting that $a_k\geqslant a/C^2$, we see that \begin{align*} r(A^{**}x^{**}_\varnothing) & = \sum_{n=1}^\infty \ee_n r_n(A^{**}x^{**}_\varnothing) \leqslant \sum_{n=1}^m \ee_n r_n(A^{**}x^{**}_\varnothing)+\mu a_0 \\ & \leqslant (1-\delta/2)\ee_k  a_k+\sum_{k\neq n=1}^m \ee_n (a_n+\mu a_0) + \mu a_0 \\ & \leqslant \sum_{n=1}^m \ee_n a_n + 2\mu a_0 -\delta\ee_ka_k/2 \\ & \leqslant \underset{t\in MAX(S)}{\sup} r(A^{**}x^{**}_t) + 2\mu a_0 - \delta\ee_k a_k/2 \\ & \leqslant \underset{t\in MAX(S)}{\inf} r(A^{**}x^{**}_t) + 3\mu a_0 - \delta\ee_k a_k/2 \\ & \leqslant a -\Bigl(\frac{\delta\ee_k}{2}\cdot a_k - \frac{\delta \ee_k}{4C^2}\cdot a\Bigr) \leqslant a\Bigl(1-\frac{\delta\ee_k}{4C^2}\Bigr).  \end{align*} From this it follows that $d^{w^*}_\xi(\ee;A^{**}:X^{**}\to (Y,r))\geqslant \ee_k \delta/4C^2>0$.

$(iii)$ Fix a $B$-tree $B$, a vector $x^{**}\in X^{**}$ with $r(x^{**})\geqslant 1$, and a $w^*$-null, $(A^{**}, 1)$-large collection $(x_t^{**})_{t\in B}\subset X^{**}$. Again, let $C$ be such that $\text{eq}(r,p), \text{eq}(p, r_n)\leqslant C$ for all $n\in \nn$.  For $\tau>0$, there exists $\delta=\delta_\xi^{w^*}(\tau/C;A^{**}:(X^{**},p)\to X)>0$ such that $$\sup_{t\in B}p\Bigl(\frac{x^{**}}{p(x^{**})}+\frac{\tau}{p(x^{**})}\sum_{\varnothing\prec s\preceq t} x_s^{**}\Bigr) \geqslant 1+\delta,$$ whence for any $k\in \nn$ such that $\text{eq}(r_k, p)^{-2}(1+\delta)>1+\delta/2$, \begin{align*}\sup_{t\in B} r_k\bigl(x^{**}+\tau \sum_{\varnothing\prec s\preceq t} x_s^{**}\bigr) & \geqslant \text{eq}(r_k, p)^{-1} \sup_{t\in B}p\bigl(x^{**}+\tau \sum_{\varnothing\prec s\preceq t} x^{**}_s\bigr) \\ & \geqslant \text{eq}(r_k, p)^{-1}p(x^{**})(1+\delta) \geqslant \text{eq}(r_k, p)^{-2}r_k(x^{**})(1+\delta) \\ & \geqslant r_k(x^{**})(1+\delta/2) \geqslant r_k(x^{**})+\delta/2C. \end{align*} 

Fix some $k$ such that $\text{eq}(r_k, p)^{-2}(1+\delta)>1+\delta/2$.  Then \begin{align*} \sup_{t\in B} r\bigl(x^{**}+\tau\sum_{\varnothing\prec s\preceq t} x_s^{**}\bigr)   -1 & \geqslant \sup_{t\in B} r\bigl(x^{**}+\tau\sum_{\varnothing\prec s\preceq t}x_s^{**}\bigr)-r(x^{**}) \\ & \geqslant \ee_k \Bigl[\sup_{t\in B} r_k\bigl(x^{**}+\tau\sum_{\varnothing\prec s\preceq t}x_s^{**}\bigr) - r_k(x^{**})\Bigr] \geqslant \ee_k\delta/2C.\end{align*} 

\end{proof}

\begin{remark}Items $(ii)$ and $(iii)$ of Proposition \ref{Asplund} remain true if we only consider $A:X\to (Y, p)$ in $(ii)$ or $A:(X,p)\to X$ in $(iii)$ and replace co-$w^*$-$\xi$-AUC or $w^*$-$\xi$-AUC with co-$\xi$-AUC or $\xi$-AUC.  

\end{remark}

\begin{corollary} Let $A:X\to Y$ be a weakly compact operator and suppose $Sz(A)\leqslant \omega^{\xi+1}$ and $Sz(A^*)\leqslant \omega^{\zeta+1}$.    There exists an equivalent norm $r$ on $Y$ such that $A^{**}:X^{**}\to (Y, r)$ is $w^*$-$\xi$-AUS and co-$w^*$-$\zeta$-AUC.   

\label{super troopers}
 
\end{corollary}

\begin{proof}We may fix norms $p, q$ on $Y$ such that $A:X\to (Y, q)$ is $\xi$-AUS and $A^{**}:X^{**}\to (Y, p)$ is co-$w^*$-$\zeta$-AUC. For this we are using  Theorems \ref{renorming1} and \ref{coAUC}.  By our remark above, for every $n\in \nn$, $r^*_n=p^*+\frac{1}{n}q^*$ is a norm on $Y^*$ which is the dual norm of some equivalent norm $r_n$ on $Y$ and such that $A^*:(Y^*, r_n^*)\to X^*$ is $w^*$-$\xi$-AUC. Of course, $\lim_n \text{eq}(p, r_n)=1$. It follows from Theorem \ref{duality} that $A:X\to (Y, r_n)$ is $\xi$-AUS and, by Corollary \ref{obs corollary}, $A^{**}:X^{**}\to (Y, r_n)$ is $w^*$-$\xi$-AUS.   Then with $r=\sum \ee_n r_n$ for some sequence $(\ee_n)$ of positive numbers with $\sum \ee_n\leqslant 1$, we deduce by Proposition \ref{Asplund} that $A^{**}:X^{**}\to (Y, r)$ is co-$w^*$-$\zeta$-AUC and $w^*$-$\xi$-AUS.  

\end{proof}

In the same way we can prove the following corollary, which implies Theorem \ref{main2}$(ii)$.  

\begin{corollary} Let $A:X\to X$ be weakly compact and let $\xi, \zeta$ be ordinals such that $Sz(A)\leqslant \omega^{\xi+1}$ and $Sz(A^*)\leqslant \omega^{\zeta+1}$.  Then there exists an equivalent norm $r$ on $X$ such that $A^{**}:(X^{**}, r)\to (X, r)$ is $w^*$-$\xi$-AUS and $w^*$-$\zeta$-AUC.

\end{corollary}

\begin{proof} We fix norms $p,q$ on $X$ such that $A^{**}:(X^{**}, p)\to Y$ is $w^*$-$\zeta$-AUC and $A:X\to (X,q)$ is $\xi$-AUS.   For this we are using Theorem \ref{renorming1} and Corollary \ref{drew corollarymore}.  Then with $r_n^*=p^*+\frac{1}{n}q^*$ and $r=\sum \ee_nr_n$, we again use Proposition \ref{Asplund} to deduce that $A^{**}:(X^{**}, r)\to (X,r)$ is $w^*$-$\xi$-AUS and $w^*$-$\zeta$-AUC.  Indeed, each $r_n^*$ is dual to some norm $r_n$ on $X$ and $\text{eq}(r_n,p)\to 1$. As in the previous corollary, we deduce that for each $n\in \nn$,  $A^*:(X^*, r_n^*)\to X^*$ is $w^*$-$\xi$-AUC and $A:X\to (X,r_n)$ is $\xi$-AUS.   By Proposition \ref{Asplund}$(i)$, $A:X\to (X,r)$ is $\xi$-AUS.  By Corollary \ref{obs corollary}, $A^{**}:X^{**}\to (X,r)$ is $w^*$-$\xi$-AUS.  Since $w^*$-$\zeta$-AUC is not lost by renorming the range space, $A^{**}:(X^{**}, p)\to (X,r)$ is $w^*$-$\zeta$-AUC.  By Proposition \ref{Asplund}$(iii)$, $A^{**}:(X^{**}, r)\to (X,r)$ is $w^*$-$\zeta$-AUC.  Since $w^*$-$\xi$-AUS is not lost by renorming the domain, $A^{**}:(X^{**}, r)\to (X,r)$ is also $w^*$-$\xi$-AUS.

\end{proof}

\section{Property $(\beta)$ and non-linear characterizations}

Given a sequence $(y_n)$ in a Banach space $Y$, we write $\sep(y_n)$ to denote $\inf_{m\neq n}\|y_m-y_n\|$.  In the case that multiple norms on $Y$ have been specified, we will write $\sep_{\|\cdot\|}(y_n)$ to denote the separation with respect to the particular norm $\|\cdot\|$.  

Let us say that an operator $A:X\to Y$ has property $(\beta)$ provided that for any $\ee>0$, there exists $\beta=\beta(\ee)>0$ such that for any $x\in B_X$ and any $(x_n)\subset B_X$ with $\sep(Ax_n)\geqslant \ee$, there exists $n\in \nn$ such that $\|x+x_n\|\leqslant 2(1-\beta)$.  Note that this property is invariant under renorming $Y$, but not under renorming $X$. Moreover, this property is a direct generalization of property $(\beta)$ of Rolewicz when $A$ is an identity operator.  We say an operator $A:X\to Y$ is $(\beta)$-\emph{able} if there exists an equivalent norm $|\cdot|$ on $X$ such that $A:(X, |\cdot|)\to Y$ has property $(\beta)$.  In complete analogy to the case of Banach spaces, we have the following.

\begin{theorem} An operator $A:X\to Y$ is $(\beta)$-able if and only if $A$ is weakly compact, $Sz(A)\leqslant \omega$, and $Sz(A^*)\leqslant \omega$.   
\label{beta}
\end{theorem}

We begin with the positive $(\beta)$-renorming results, for which all of the preparatory work is done.  We remark that $Sz(A^{**})=Sz(A)$ when $A$ is weakly compact.  This is because if $\Phi:Y^*\to Y^{***}$ is the canonical embedding,  $\Phi(A^*B_{Y^*})=A^{***}B_{Y^{***}}$ and the $w^*$-topologies from $Y^*$ and $Y^{***}$ coincide on this set. Therefore the function $\Phi$ commutes with the Szlenk derivation.   In particular, if $A$ is weakly compact, the Szlenk index of $A$ is equal to the Szlenk index of every even adjoint of $A$, and the Szlenk index of $A^*$ is equal to the Szlenk index of every odd adjoint of $A$.   Therefore being $(\beta)$-able is a self-dual property.  We obtain the following.

\begin{proposition} If $A:X\to Y$ is weakly compact, $Sz(A)\leqslant \omega$, and $Sz(A^*)\leqslant \omega$, then $A$ is $(\beta)$-able.

\end{proposition}

\begin{proof}   By Corollary \ref{super troopers} with $0=\xi=\zeta$, there exists an equivalent norm $r$ on $Y$ such that $A^{**}:X^{**}\to (Y,r)$ is $w^*$AUS and co-$w^*$-AUC.    Let $p(x)^2=\|x\|^2+r(Ax)^2$. Fix $\ee>0$.  Since $A^{**}:X^{**}\to (Y,r)$ is co-AUC, there exists $\delta>0$ such that if $(x_n^{**})\subset B_{X^{**}}$, $\sep_r(A^{**}x_n^{**})\geqslant \ee$, and $A^{**}x^{**}_n\underset{w}{\to}y$, $r(y)\leqslant \underset{n}{\lim\inf\ } (1-\delta)r(A^{**}x^{**}_n)$.   Since $A^{**}:X^{**}\to (Y,r)$ is $w^*$-AUS, there exists $\sigma_0\in (0,1/2)$ such that for any $0<\sigma\leqslant \sigma_0$, for any $w^*$-null net $(x_\lambda^{**})\subset B_X$ and any $y\in B_Y^r$, $$\underset{\lambda}{\lim\sup\ }r(y+\sigma x_\lambda^{**}) \leqslant 1+  \sigma\tau,$$   where $\tau=\ee\delta/8$.

Fix $x\in B_X^p$ and a sequence $(x_n)\subset B_X^p\subset B_X$ with $\sep_r(Ax_n)\geqslant \ee$.   Without loss of generality, we may assume $\lim_n \|x_n\|$ exists, $\lim_n r(Ax_n)\to t\geqslant \ee/2$, and $Ax_n\underset{w}{\to}y$.  First suppose that $r(Ax)<\ee/4$.  Let $u=(\|x\|, r(Ax))\in B_{\ell_2^2}$ and assume by passing to a further subsequence that $v= \lim_n (\|x_n\|, r(Ax_n))\in B_{\ell_2^2}$.  Note that the second coordinate of $v$ is $t\geqslant \ee/2$, so that $\|u-v\|\geqslant \ee/4$.    Therefore \begin{align*} \underset{n}{\lim\sup\ } p\bigl(\frac{x+x_n}{2}\bigr) & = \underset{n}{\lim\sup}\Bigl[\bigl\|\frac{x+x_n}{2}\bigr\|^2+ r\bigl(\frac{Ax+Ax_n}{2}\bigr)^2\Bigr]^{1/2}   \\ & \leqslant \underset{n}{\lim\sup} \Bigl[\Bigl(\frac{\|x\|+\|x_n\|}{2}\Bigr)^2+ \Bigl(\frac{r(Ax)+r(Ax_n)}{2}\Bigr)^2\Bigr]^{1/2} = \bigl\|\frac{u+v}{2}\bigr\|_{\ell_2^2} \\ & = \Bigl[\frac{\|u\|^2+\|v\|^2}{2} - \|\frac{u-v}{2}\|^2\Bigr]^{1/2} \\ & \leqslant \bigl[1 - (\ee/8)^2\bigr]^{1/2}\leqslant 1-\ee^2/128. \end{align*}

Now consider the case that $r(Ax)\geqslant \ee/4$.        We may pass to a subnet $(x_\lambda)$ of $(x_n)$ which converges $w^*$ to some $x^{**}\in B_{X^{**}}^p$ and note that $A^{**}x^{**}=y$.   Then $Ax_\lambda \underset{w}{\to}y$ and $r(Ax_\lambda)\to t$. Note that $r(y)\leqslant (1-\delta)t$.   Let $\lambda=\ee\sigma_0/16$ and let $$\mu= (1-\lambda)r(Ax)+\lambda (1-\delta)t\geqslant \ee/8.$$  Note that $r((1-\lambda)Ax+\lambda y)\leqslant \mu$, $2\lambda/\mu\leqslant \sigma_0$, and  $(x_\lambda-x^{**})$ is a $w^*$-null net, so \begin{align*} \underset{\alpha}{\lim\sup\ } r((1-\lambda)Ax +\lambda Ax_\alpha) & = \mu \underset{\alpha}{\ \lim\sup\ } r\Bigl(\frac{(1-\lambda)Ax+\lambda y}{\mu} + \frac{2\lambda}{\mu} A^{**} \bigl(\frac{x_\alpha-x^{**}}{2}\bigr)\Bigr) \\ & \leqslant \mu\Bigl(1+ \frac{2\lambda\tau}{\mu}\Bigr) = (1-\lambda)r(Ax) + (1-\delta)\lambda t + 2\lambda \tau \\ & = (1-\lambda)r(Ax) + \lambda t + (2\lambda \tau-\delta \lambda t) \\ & \leqslant (1-\lambda)r(Ax)+\lambda t -\ee\delta\lambda/4 .  \end{align*} From this it follows that with $\eta=\ee\delta\lambda/8$, $$\underset{\alpha}{\lim\sup\ }r\Bigl(\frac{Ax+Ax_\alpha}{2}\Bigr) \leqslant r(Ax)/2+t/2-\eta.$$

From this, it follows that \begin{align*} \underset{\alpha}{\lim\sup} p\bigl(\frac{x+x_\alpha}{2}\bigr) & = \underset{\alpha}{\lim\sup} \Bigl[\bigl\|\frac{x+x_\alpha}{2}\bigr\|^2 + r\bigl(\frac{Ax+Ax_\alpha}{2}\bigr)^2\Bigr]^{1/2} \\ & \leqslant \lim_\alpha \Bigl[\Bigl(\frac{\|x\|+\|x_\alpha\|}{2}\Bigr)^2+ \Bigl(\frac{r(Ax)+r(Ax_\alpha)}{2} -\eta\Bigr)^2\Bigr]^{1/2} \\ & \leqslant (1-\eta^2)^{1/2} \leqslant 1-\eta^2/2.  \end{align*} Here we have used the fact that for any vector $h=(h_1, h_2)\in \ell_2^2$ and any $\theta\geqslant 0$ such that  $0\leqslant \theta\leqslant h_2$, $$\|h-(0, \theta)\|_{\ell_2^2}\leqslant (\|h\|^2_{\ell_2^2}-\theta^2)^{1/2}.$$  

Since $\eta^2/2\leqslant \ee^2/128$, we see that $A:(X,p)\to (Y,r)$ has property $(\beta)$ with $\beta(\ee)\geqslant \eta^2/2$.  Note that since property $(\beta)$ is invariant under renorming $Y$, $A:(X,p)\to Y$ has property $(\beta)$.

\end{proof}

 Note that we may repeat this construction to obtain a dual norm $q^*$ on $Y^*$ such that $A^*:(Y^*, q^*)\to X$ has property $(\beta)$.   From this it follows that we may renorm both $X$ and $Y$ so that $A:X\to Y$ and $A^*:Y^*\to X^*$ simultaneously have property $(\beta)$.

The remainder of this section is devoted to the other direction of Theorem \ref{beta}.     The easiest piece of this direction is due to classical characterizations of weak compactness due to James \cite{James}.  

\begin{proposition} If $A$ is $(\beta)$-able, then $A$ is weakly compact.   

\end{proposition}

\begin{proof} Suppose that $A:X\to Y$ has property $(\beta)$ and $A$ is not weakly compact.  Then there exists $x^{**}\in S_{X^{**}}$ such that $A^{**}x^{**}\in Y^{**}\setminus Y$.   Then following James's recursive construction, we obtain sequences $(x_n)\subset B_X$ and $(x_n^*)\subset B_{X^*}$ such that for each $n\geqslant m$, $\text{Re\ }x^*_m(x_n)\geqslant 1-1/m$ and such that $(Ax_n)$ is seminormalized and basic.  Then $\sep(Ax_n)\geqslant \ee>0$ for some $\ee>0$ (which can be estimated as a function of $\|A^{**}y^{**}\|_{Y^{**}/Y}$).  However, there exists no $\beta=\beta(\ee)$ as in the definition of property $(\beta)$.  Indeed, if $\beta>1/m$, the member $x_m\in B_X$ and the sequence $(x_{m+n})_{n\in \nn}$ satisfy $\|\frac{x_m+x_{m+n}}{2}\|\geqslant 1-1/m>1-\beta$ for all $n\in \nn$, while $\sep(Ax_{n+m})\geqslant \ee$.

\end{proof}

The remainder of proving this direction will rely on non-linear results.   Given a collection $\mathcal{M}$ of metric spaces and an operator $A:X\to Y$, we say $\mathcal{M}$ \emph{factors through} $A$ provided that there exists $D>0$ such that for every $(G,d)\in \mathcal{M}$, there exists $f:G\to X$ such that for every $s,t\in G$, $$\frac{d(s,t)}{D}\leqslant \|Af(s)-Af(t)\|, \|f(s)-f(t)\|\leqslant d(s,t).$$  If $A$ is the identity on $X$, this is the same as the members of $\mathcal{M}$ admitting bi-Lipschitz embeddings into $X$ with uniform distortions.   If $\mathcal{M}=\{(G,d)\}$, we say $G$ factors through $A$ rather than $\mathcal{M}$ factors through $A$.   Of course, whether or not $\mathcal{M}$ factors through $A$ is invariant under renormings of $X$ and $Y$.  

  For each $n\in \nn$, let $T_n$ consist of all sequences of natural numbers having length not exceeding $n$, including the empty sequence $\varnothing$.   We let $T$ denote all finite sequences of natural numbers and treat each $T_n$ as a subset of $T$.  We endow $T$ with the distance $d(s,t)=|s|+|t|-2|u|$, where $u$ is the largest common initial segment of $s$ and $t$.  This is the graph distance on the graph which has vertex set $T$ and where the sequence $t$ is adjacent to $t\cat n$ for every $n\in \nn$.   We let $\mathcal{T}=\{(T_n, d):n\in \nn\}$. Of course, since each $T_n$ is a subset of $T$, if $T$ factors through $A$, then $\mathcal{T}$ factors through $A$.

     The following argument is a modification of the elegant argument of Baudier and Zheng \cite{BZ}, which is itself a modification of Kloeckner's argument \cite{Kl} for the non-embeddability of binary trees into uniformly convex spaces.

\begin{proposition} If $A:X\to Y$ is $(\beta)$-able, then $\mathcal{T}$ does not factor through $A$. 

\end{proposition}

\begin{proof} Suppose that $A:X\to Y$ has property $(\beta)$ and $\|A\|\leqslant 1$.  Fix $D\geqslant 1$.   We prove by induction on $n$ that if $f:T_{2^n}\to X$ is such that for every $s,t\in T_{2^n}$, $$d(s,t)/D\leqslant \|Af(s)-Af(t)\|\leqslant \|f(s)-f(t)\|\leqslant d(s,t),$$ then for any $m\in \nn$,  there exists $t=(m_i)_{i=1}^{2^n}\in T_{2^n}$ such that $m_1=m$ and $\|f(t)-f(\varnothing)\|\leqslant 2^n(1-\beta)^n$, where $\beta=\beta(2/D)$.   This will imply that $D\geqslant (1-\beta)^{-n}$ and yield the result.    

First, the base case.   Suppose $f:T_2\to X$ is as described in the previous paragraph.  Assume without loss of generality that $f(\varnothing)=0$.    Let $x=f((m))$ and $x_n=f((m,n))-x$. Then $x, x_n\in B_X$ and $$\sep(Ax_p)=\inf_{q\neq p} \|Ax_q-Ax_p\|= \inf_{q\neq p}\|Af((m,q))-Af((m,p))\| \geqslant 2/D.$$   By the definition of property $(\beta)$, there exists $p\in \nn$ such that $$\|f((m,p))-f(\varnothing)\|= \|x+x_p\|\leqslant 2(1-\beta),$$ finishing the base case.

Next, suppose that $f:T_{2^{n+1}}\to X$ satisfies the hypotheses and assume the result for $n$.  Assume also that $f(\varnothing)=0$ and fix $m\in \nn$.     Consider $f|_{T_{2^n}}\to X$ and apply the inductive hypothesis for each $p\in \nn$ to obtain a sequence $t_p$ with $|t_p|=2^n$ such that the first member of $t_p$ is $p$ and such that $\|f(t_p)\|\leqslant 2^n(1-\beta)^n$ for each $p\in \nn$.    Next, for each $p\in \nn$, let $f_p:T_{2^n}\to X$ be given by $f_p(t)=f(t_p\cat t)$.   For each $q\in \nn$, we may fix $s_{p,q}\in T_{2^n}$ with first member $q$ and such that $|s_{p,q}|=2^n$ and $\|f(t_p\cat s_{p,q})-f(t_p)\| \leqslant 2^n(1-\beta)^n$.  Let $t_{p,q}=t_p\cat s_{p,q}$.  Then one easily checks that $g:T_2\to X$ given by $g(\varnothing)=f(\varnothing)$, $g((p))=f(t_p)/2^n(1-\beta)^n$, and $g((p,q))=f(t_{p,q})/2^n(1-\beta)^n$ satisfies the hypotheses of the base case.   By the base case, there exists $(m,q)\in T_2$ such that $\|g((m,q))\|\leqslant 2(1-\beta)$, from which it follows that $\|f(t_{m,q})\|\leqslant 2^{n+1}(1-\beta)^{n+1}$.    

\end{proof}

We next aim to show that if $Sz(A^*)>\omega$, then $T$ factors through $A$, $A^*$, and $A^{**}$.  The portion of the proof which shows that $T$ factors through $A^*$ and $A^{**}$ will not depend in any way on the fact that $A^*$ is an adjoint, so the same proof can be used to show that if $Sz(A)>\omega$, then $T$ factors through $A$ and $A^*$.  To avoid repeating lengthy computations, we omit a direct proof of this implication and leave it to the reader to modify the following results.  

Let $N:T\to \nn_0$ be a bijection such that if $s\prec t$, $N(s)<N(t)$.   Let us define the linear order $<$ on $T$ by letting $s<t$ if and only if $N(s)<N(t)$.    Let $r_n=2^n-1$ for $n\in \nn_0$, $\ell(t)=\max\{n\in \nn_0: r_n\leqslant |t|\}$, $L(t)=\max\{\ell(s): \varnothing\leqslant s\leqslant t\}$, and $L_n=\{t\in T: \ell(t)=n\}$.   

\begin{lemma} Suppose that $\ee>0$ is such that $0\in s_{4\ee}^n(A^{**}B_{Y^{**}})$ for every $n\in \nn$.    Then for any $(\ee_n)\subset (0,1)$ and any $\delta>0$, there exist collections $(x_t)_{t\in T}\subset B_X$, $(y^*_t)_{t\in T}\subset B_{Y^*}$, $(x^{**}_t)_{t\in T}\subset B_{X^{**}}$ such that for every $t\in T$, \begin{enumerate}[(i)]\item $\text{\emph{Re}\ }A^{**}x^{**}_t(y^*_t)>\ee-\delta$,  \item for each $\varnothing<s<t$, $|A^{**}x^{**}_s(y^*_t)|, |y^*_t(Ax_s)|< \delta/ 2^{3 L(t)+2},$ \item if $|t|=r_n$, then for each $\varnothing<s<t$, $|A^{**}x^{**}_t(y^*_s)|<  \delta/ 2^{3L(t)+2},$ \item if $r_n<|t|<r_{n+1}$, for each $\varnothing<s<t$, $|(A^{**}x^{**}_t- A^{**}x^{**}_{t^-})(y^*_s)|< \delta/ 2^{3L(t)+2},$ \item for each $\varnothing<s\leqslant t$, $|A^{**}x^{**}_t(y^*_s)- y^*_s(Ax_t)|< \delta/ 2^{3L(t)+2}.$ \end{enumerate}

\label{hardlemma}

\end{lemma}

\begin{proof} We define $x_t$, $y_t^*$, and $x^{**}_t$ by induction on $N(t)$.  We let $x_\varnothing=y^*_\varnothing=x^{**}_\varnothing=0$.

Assume that for some $t\neq \varnothing$, $x_s$, $y_s^*$, and $x^{**}_s$ have been defined for each $s<t$.   Moreover, assume that for each $s<t$, if $|s|=r_n+l$ for some $0\leqslant l<2^n$, $A^{**}x^{**}_x\in s_{4\ee}^{2^n-l}(A^{**}B_{X^{**}})$.    Fix integers $n,l$ such that $0\leqslant l<2^n$ and $|t|=r_n+l$. If $l=0$, let $y^{**}=0\in s_{4\ee}^{2^n+1}(A^{**}B_{X^{**}})=s_{4\ee}(s_{4\ee}^{2^n}(A^{**}B_{X^{**}}))$.  If $l>0$, let $y^{**}=A^{**}x^{**}_{t^-}$.  In the second case, our assumption guarantees that $y^{**}\in s_{4\ee}^{2^n-l+1}(A^{**}B_{X^{**}})=s_{4\ee}(s_{4\ee}^{2^n-l}(A^{**}B_{X^{**}}))$.  Let $$U=\Bigl\{y^*\in Y^*: (\forall \varnothing<s<t)(|A^{**}x^{**}_s(y^*)|,|y^*(Ax_s)|< \delta/2^{3L(t)+2} )\Bigr\}.$$  Let $$V=\Bigl\{u^{**}\in Y^{**}: (\forall \varnothing<s<t)(|y^{**}(y^*_s)-u^{**}(y^*_s)|<\delta/2^{3L(t)+2})\Bigr\}.$$   By Lemma \ref{lemma1} and the choice of $y^{**}$, there exists $y^{**}_0\in s_{4\ee}^{2^n-l}(A^{**}B_{X^{**}})\cap V$ and $y^*\in B_{Y^*}\cap U$ such that $\text{Re\ }y^{**}_0(y^*)>\ee-\delta$.   Let $y_t^*=y^*$.   Fix $x_t^{**}\in B_{X^{**}}$ such that $A^{**}x^{**}_t=y^{**}_0$. We see that with these choices, $(i)$ is satisfied.   Our choice of $U$ guarantee that $(ii)$ is satisfied.   Our choice of $V$ and $y^{**}$ guarantee that the $(iii)$ is satisfied if $l=0$ and $(iv)$ is satisfied if $l>0$.   We choose by Goldstine's theorem some $x_t\in B_X$ to satisfy $(v)$.

\end{proof}

In order to make the remainder of the work more readable, we introduce some terminology and notation.  A \emph{segment} $\mathfrak{s}$ will be a subset of $T$ of the form $\mathfrak{s}=\{w: u\preceq w\preceq t\}$ for some $u\preceq t\in T$. The notations $[u,s]$ and $(u,s]$ will denote the obvious segments.       For a segment $\mathfrak{s}$ and $v\in T$, we will write $\mathfrak{s}\perp v$ if no member of $\mathfrak{s}$ is comparable to $v$.    We will write $\mathfrak{s}\sqsubset \mathfrak{t}$ if for every $s\in \mathfrak{s}$ and $t\in \mathfrak{t}$, $s\prec t$ and $\ell(s)<\ell(t)$.   We will write $\mathfrak{s}\sqsubset v$ to denote $\mathfrak{s}\sqsubset \{v\}$ and $v\sqsubset \mathfrak{s}$ to denote $\{v\}\sqsubset \mathfrak{s}$.

\begin{proposition} Let $(x_t)_{t\in T}$, $(y^*_t)_{t\in T}$, and $(x^{**}_t)_{t\in T}$ be as in Lemma \ref{hardlemma}.   \begin{enumerate}[(i)]\item For any $i\in \nn$, if $\varnothing\neq \mathfrak{s}\subset L_i$, if $v$ is the $\prec$-minimal member of $\mathfrak{s}$, then $$\text{\emph{Re}}\sum_{w\in \mathfrak{s}} A^{**}x_w^{**}(y^*_v),\text{\emph{Re}} \sum_{w\in \mathfrak{s}}y^*_v(Ax_w) \geqslant (\ee-3\delta)|\mathfrak{s}|.$$  \item For any $i\in \nn$, for any $\varnothing\neq \mathfrak{s}\subset L_i$, if $v$ is the $\prec$-maximal member of $\mathfrak{s}$,  $$ \text{\emph{Re}}\sum_{w\in \mathfrak{s}}A^{**}x^{**}_v(y^*_w)\geqslant (\ee-3\delta)|\mathfrak{s}|.$$ \item If $v,w$ are incomparable or $\ell(w)\neq \ell(v)$, $$|A^{**}x^{**}_w(y^*_v)|, |y^*_v(Ax_w)|<\delta/2^{\max\{\ell(v), \ell(w)\}}.$$    \item For any segment $\mathfrak{s}$ and $v\in T$ such that either $\mathfrak{s}<v$, $v<\mathfrak{s}$, or $\mathfrak{s}\perp v$,   $$\sum_{w\in \mathfrak{s}}|A^{**}x_w^{**}(y^*_v)|, \sum_{w\in \mathfrak{s}}|y^*_v(Ax_w)|, \sum_{w\in \mathfrak{s}}|A^{**}x^{**}_v(y^*_w)|< \delta.$$   \end{enumerate}
\label{prop1}
\end{proposition}

\begin{proof}$(i)$ Let $\mathfrak{s}=[v,s]$.   For $v\preceq u\preceq s$, using properties $(i)$ and $(iv)$ from Lemma \ref{hardlemma}, $$\text{Re\ }A^{**}x^{**}_u(y^*_v) \geqslant \text{Re\ }A^{**}x^{**}_v(y^*_v) -\sum_{v\prec w\preceq u}|(A^{**}x^{**}_w-A^{**}w^{**}_{w^-})(y^*_v)| \geqslant \ee-\delta - 2^n\delta/2^{3n+2} >\ee-2\delta,$$ where $n=\ell(v)$.  Using property $(v)$, $\text{Re\ }Ax_u(y^*_v)> \text{Re\ }A^{**}x^{**}_u(y^*_v)- \delta > \ee-3\delta$.  Summing over the appropriate $u$ gives $(i)$.  

$(ii)$ This is similar to $(i)$, summing over the appropriate range.   

$(iii)$ Let $m=\max\{\ell(v), \ell(w)\}$. Assume that either $v,w$ are incomparable or $\ell(w)\neq \ell(v)$.     If $w<v$, $L(v)\geqslant m$ and $$|A^{**}x^{**}_w(y_v^*)|, |y^*_v(Ax_w)|<\delta/2^{3L(v)+2}<\delta/2^{2m}$$ by Lemma \ref{hardlemma}$(iii)$.  For the remainder of $(iii)$, we assume $v<w$. Let $n=\ell(w)$ and let $g_1$ be the initial segment of $w$ with $|g_1|=r_n$. Note that $\ell(g_1)=\ell(w)$.  Let $g_2$ be the minimal initial segment of $w$ such that $g_2^-< v$.  Let $g$ be the larger of the two initial segments $g_1, g_2$.  Note that if $g=g_2$, then $g_2^-<v<g_2$ and $\ell(g_2)=\ell(w)$.  That $\ell(g_2)=\ell(w)$ in the case that $g=g_2$ follows from the fact that in this case, $\ell(w)=\ell(g_1)\leqslant \ell(g_2)\leqslant \ell(w)$.  If $v$ is not an initial segment of $w$, $g_2^-<v<g_2$. If $v$ is an initial segment of $w$, then we must be in the case that $\ell(w)\neq \ell(v)$ and $v<w$.  This means that $\ell(v)<\ell(w)$. Since $\ell(g_2)=\ell(w)$, $\ell(v)<\ell(g_2)$. In this case, since both $v$ and $g_2$ are initial segments of $w$ and $\ell(v)<\ell(g_2)$, $v\prec g_2$  and $v<g_2$.

Suppose that $g_1\prec g=g_2$, so that $g_2^-<v< g_2$ and $\ell(w)=\ell(g_2)$. In this case, since $g_1\prec g_2$, $g_1\preceq g_2^-<v$.  From this it follows that $\ell(w)=\ell(g_1) \leqslant L(v)$.  Since $\ell(v)\leqslant L(v)$, $m\leqslant L(v)$.  Therefore $|A^{**}x^{**}_{g_2^-}(y^*_v)|<\delta/2^{3L(v)+2}\leqslant \delta/2^{3m+2}$ by Lemma \ref{hardlemma}$(ii)$ applied with $s=g_2^-$ and $t=v$.    Then by Lemma \ref{hardlemma}$(iv)$, \begin{align*} |A^{**}x^{**}_w(y^*_v)| & \leqslant |A^{**}x^{**}_{g_2^-}(y^*_v)| + \sum_{g_2\preceq h\preceq w}|A^{**}x^{**}_h(y^*_v)-A^{**}x^{**}_{h^-}(y^*_v)| \\ & < \frac{\delta}{2^{3m+2}} +\frac{\delta 2^n}{2^{3m+2}} < \delta/2^{2m+1}. \end{align*} Here we have used the fact that $\ell(w)=n\leqslant m$ and $L(h)\geqslant m$ for each $g_2\preceq h \preceq w$.  

Now assume that $g=g_1$. In this case, $v<g_1$.  Indeed, $v=g_1$ is impossible since either $v$ and $w$ are either incomparable or $\ell(v)\neq \ell(w)$, and $g_1<v$ is impossible, since then $g_1\prec g_2$. Since $v<g_1$, $L(g_1) \geqslant \ell(v)$ and $L(g_1)\geqslant \ell(g_1)=\ell(w)$.  Therefore $L(g_1)\geqslant m$. By Lemma \ref{hardlemma}$(iii)$, $|A^{**}x^{**}_{g_1}(y^*_v)|<\delta/2^{3m+2}$.  Since $L(h)\geqslant m$ for each $g_1\preceq h\preceq w$, we deduce by Lemma \ref{hardlemma}$(iv)$ that \begin{align*} |A^{**}x^{**}_w(y^*_v)| & \leqslant |A^{**}x^{**}_{g_1}(y^*_v)| + \sum_{g_1\prec h\preceq w}|A^{**}x^{**}_h(y^*_v)-A^{**}x^{**}_{h^-}(y^*_v)| \\ & < \frac{\delta}{2^{3m+2}} +\frac{\delta 2^n}{2^{3m+2}} < \delta/2^{2m+1}. \end{align*}  

Recalling that in either of the cases $g_1\prec g$ and $g=g_1$ above, $L(w)\geqslant m$, we use Lemma \ref{hardlemma}$(v)$ to deduce that $$|y^*_v(Ax_w)|<|A^{**}x^{**}_w(y^*_v)|+\delta/2^{3L(w)+2} <\delta/2^{2m+1}+\delta/2^{3m+2}<\delta/2^{2m}.$$

$(iv)$ Since $x_\varnothing=y^*_\varnothing=x^{**}_\varnothing=0$, we may assume $0\notin \mathfrak{s}$ without loss of generality. If $v, \mathfrak{s}$ are such that either $v\sqsubset\mathfrak{s}$, $\mathfrak{s}\sqsubset v$, or $v\perp \mathfrak{s}$, we may write $\mathfrak{s}$ as the union $\mathfrak{s}=\cup_{i=1}^\infty \mathfrak{s}_i$ where $\mathfrak{s}_i=\mathfrak{s}\cap L_i$.  Then $|\mathfrak{s}_i|\leqslant 2^i$, and $v, \mathfrak{s}_i$ satisfy the same relationship as $v, \mathfrak{s}$. That is, if $v\sqsubset \mathfrak{s}$, then $v\sqsubset \mathfrak{s}_i$ for each $i$, etc.  By $(iii)$, for each $i\in \nn$ and $w\in \mathfrak{s}_i$, $|A^{**}x^{**}_w(y^*_v)|, |y^*_v(Ax_w)|<\delta/2^{2i}$. This follows from the fact that in each case, either $w$ and $v$ are incomparable (if $v\perp \mathfrak{s}$) or $\ell(w)\neq \ell(v)$ (if either $v\sqsubset \mathfrak{s}$ or $\mathfrak{s}\sqsubset v$) and $\max\{\ell(w), \ell(v)\}\geqslant \ell(w)\geqslant i$ when $w\in \mathfrak{s}_i$.  Similarly we deduce by $(iii)$, by reversing the roles of $v$ and $w$, that $|A^{**}x^{**}_v(y^*_w)|<\delta/2^{2i}$ for $w\in \mathfrak{s}_i$.         Summing as $w$ ranges over $\mathfrak{s}=\cup_{i=1}^\infty \mathfrak{s}_i$ yields  $$\sum_{i=1}^\infty\sum_{w\in \mathfrak{s}_i}|A^{**}x^{**}_w(y^*_v)|, \sum_{i=1}^\infty \sum_{w\in \mathfrak{s}_i}|y^*_v(Ax_w)|, \sum_{i=1}^\infty \sum_{w\in \mathfrak{s}_i}|A^{**}x^{**}_v(y^*_w)|\leqslant \sum_{i=1}^\infty \frac{\delta |\mathfrak{s}_i|}{2^{2i}} \leqslant \sum_{i=1}^\infty \frac{2^i \delta}{2^{2i}}=\delta.$$   

\end{proof}

\begin{corollary} Let $A:X\to Y$ be an operator.  \begin{enumerate}[(i)]\item If  $Sz(A^*)>\omega$, then $T$ factors through $A$, $A^*$, and $A^{**}$.  \item If $Sz(A)>\omega$, then $T$ factors through $A$ and $A^*$.  \end{enumerate}   
\label{maincorollary}
\end{corollary}

\begin{proof}

$(i)$ We first note that if $Sz(A^*)>\omega$, then there exists $\ee>0$ such that $0\in s_{4\ee}^n(A^{**}B_{X^{**}})$ for all $n\in \nn$, which was shown in \cite{Lancien2}.   Fix any $\ee_0\in (0, \ee)$ and $\delta>0$ such that $\ee-6\delta>\ee_0$.   Let $(x_t)_{t\in T}$, $(y^*_t)_{t\in T}$, $(x^{**}_t)_{t\in T}$ be as in Lemma \ref{hardlemma}.    Define $\Phi:T\to X$ by $\Phi(t)=\sum_{\varnothing\prec s\preceq t} x_s$, and let $\Phi^*:T\to Y^*$ and $\Phi^{**}:T\to X^{**}$ be defined similarly. The triangle inequality easily yields that $\Phi$, $\Phi^*$, and $\Phi^{**}$ are $1$-Lipschitz.   Fix distinct $s,t\in T$ and let $u$ be the largest common initial segment of $s$ and $t$.    We claim that, by switching $s$ and $t$ if necessary, there exist segments $\mathfrak{s}_1$, $\mathfrak{s}$, $\mathfrak{s}_2$ such that $|\mathfrak{s}|\geqslant d(s,t)/8$, $\mathfrak{s}_1\cup \mathfrak{s}\cup \mathfrak{s}_2=(u,s]$, $\mathfrak{s}\subset L_i$ for some $i\in \nn$, and for any $v\in \mathfrak{s}$, $\mathfrak{s}_1\sqsubset v\sqsubset \mathfrak{s}_2$.    We first assume the claim and establish the desired lower estimates.  If $v$ is the $\prec$-minimal member of $\mathfrak{s}$, then by Proposition \ref{prop1}, \begin{align*} \|A\Phi(s)-A\Phi(t)\| & \geqslant \text{Re\ }y^*_v\Bigl(\sum_{w\in \mathfrak{s}} Ax_w\Bigr) - \sum_{w\in \mathfrak{s}_1}|y^*_v(Ax_w)| \\ & - \sum_{w\in \mathfrak{s}_2}|y^*_v(Ax_w)| - \sum_{w\in (u,t]} |y^*_v(Ax_w)| \\ & \geqslant (\ee-3\delta)|\mathfrak{s}| - 3\delta \geqslant (\ee-6\delta)|\mathfrak{s}| \\ & \geqslant d(s,t)\ee_0/8. \end{align*} The lower estimate on $\|A^{**}\Phi^{**}(s)-A^{**}\Phi^{**}(t)\|$ is established similarly.  To establish the lower estimate on $\|A^*\Phi^*(s)-A^*\Phi^*(t)\|$, we consider the action of the functional $x^{**}_v$, where $v$ is the $\prec$-maximal member of $\mathfrak{s}$.

We return to the claim.  Let  $\max\{\ell(s), \ell(t)\}=n+1$.  

Case $1$: $\ell(u)<n$. By switching $s$ and $t$, we may assume $n+1=\ell(s)$.  Then $$d(s,t)\leqslant |s|+|t|\leqslant 2r_{n+2}<2^{n+3}.$$  Let $\mathfrak{s}=\{w\in (u,s]: \ell(w)=n\}$, so that $|\mathfrak{s}|=2^n\geqslant d(s,t)/8$.  We let $\mathfrak{s}_1=\{w\in (u,s]: \ell(w)<n\}$ and $\mathfrak{s}_2=\{w\in (u,s]: \ell(w)>n\}$.    

Case $2$:  $\ell(u)=n+1$.  Then $(u, s], (u,t]\subset L_{n+1}$ and $d(s,t)=|(u, s]|+|(u,t]|$.  By switching $s$ and $t$, we may assume $|(u,s]|\geqslant d(s,t)/2$.  We let $\mathfrak{s}=(u,s]$, $\mathfrak{s}_1=\mathfrak{s}_2=\varnothing$.   

Case $3$: $\ell(u)=n$. Then we may write $(u,s]=(u,s']\cup(s',s]$ and $(u,t]=(u,t']\cup (t', t]$, where $(u, s'], (u,t']\subset L_n$, $(s', s], (t', t]\subset L_{n+1}$.  Note that some of these segments may be empty, and $d(s,t)$ is the sum of the cardinalities of these four segments.  By switching $s$ and $t$, we may assume that one of the two segments $(u, s']$, $(s',s]$ has cardinality at least $d(s,t)/4$.    If $|(u, s']|\geqslant d(s,t)/4$, we let $\mathfrak{s}_1=\varnothing$, $\mathfrak{s}=(u, s']$, and $\mathfrak{s}_2=(s',s]$.   Otherwise we let $\mathfrak{s}_1=(u, s']$, $\mathfrak{s}=(s', s]$, and $\mathfrak{s}_2=\varnothing$.

$(ii)$ As mentioned above, this follows from repeating the computations above with $A^*:Y^*\to X^*$ replaced by $A:X\to Y$, $A^{**}:X^{**}\to Y^{**}$ replaced by $A^*:Y^*\to X^*$, and omitting all reference to $A,X$, and $Y$.

\end{proof}

\section{Remarks}

In our definition of $\xi$-AUS, we considered $B$-trees $B$ with $o(B)=\omega^\xi$.  One may ask whether more precise results are possible if we consider the modulus $$r_\xi^w(\sigma;A)=\sup\Bigl\{\inf\{\|y+\sigma Ax\|: t\in B, x\in \text{co}(x_s:\varnothing\prec s\preceq t)\}\Bigr\},$$ where the supremum is taken over all $y\in B_Y$, all $B$-trees $B$ with $o(B)=\xi$, and weakly null collections $(x_t)_{t\in B}\subset B_X$. With these definitions, $r_{\omega^\xi}^w$ coincides with $\rho^w_\xi$. Similarly, we may define $p_\xi^{w^*}(\sigma;A^*)$ in a way such that $p_{\omega^\xi}^{w^*}(\sigma;A^*)$ coincides with $\delta_\xi^{w^*}(\sigma;A^*)$.        However, there are no better renorming results to be obtained by considering these finer gradations.  Indeed, it follows from the description of the Szlenk index of an operator which was given in \cite{Causey1} that if $\xi\geqslant Sz(A)$, then for any $B$-tree $B$ with $o(B)=\xi$ and any weakly null collection $(x_t)_{t\in B}\subset B_X$, $$\inf\{\|Ax\|: t\in B, x\in \text{co}(x_s: \varnothing\prec s\preceq t)\}=0,$$   whence $$\inf\{\|y+\sigma Ax\|-1: t\in B, x\in \text{co}(x_s: \varnothing\prec s\preceq t)\}\leqslant \|y\|-1\leqslant 0.$$    Thus for a given operator $A$, it is trivial to consider $r^w_\xi$ for $\xi\geqslant Sz(A)$.  However, if $\xi<\omega^\gamma<Sz(A)$, we claim that it is also trivial to consider $r^w_\xi$. Indeed, one may prove that if $r^w_\xi(\sigma;A)\leqslant \sigma\tau$ for some $0<\sigma, \tau\leqslant 1$, then $p^{w^*}_\xi(6\tau;A^*)\geqslant \sigma\tau$ as in Proposition \ref{dualprop}$(i)$.  This is because the only time special properties of the ordinals of the form $\omega^\gamma$ were used in our proofs were to obtain Proposition \ref{dualprop}$(ii)$.  One can then deduce that if $r^w_\xi(\sigma;A)/\sigma\to 0$ as $\sigma\to 0$, then $p^{w^*}_\xi(\tau;A^*)>0$ for every $\tau>0$.  Then one argues as in Proposition \ref{finish prop} to deduce that for every $\ee>0$, there exists $\delta\in (0,1)$ such that $s_\ee^\xi(A^*B_{Y^*})\subset (1-\delta)A^*B_{Y^*}$.  A standard homogeneity argument yields that $Sz(A)\leqslant \xi\omega$. Then if $\xi<\omega^\gamma$, $\xi\omega\leqslant \omega^\gamma$ and $Sz(A)\leqslant \omega^\gamma$.   Thus if $\xi<\omega^\gamma<Sz(A)$, it is trivial to consider $r^w_\xi$.    If $Sz(A)=\omega^\zeta$ for a limit ordinal $\zeta$, then for any $\xi<Sz(A)$, there exists an ordinal $\gamma$ such that $\xi<\omega^\gamma<Sz(A)$, which means that there are no $\xi$ for which one can deduce any positive renorming results regarding $r^w_\xi$.   Since any Asplund operator $A$ has $Sz(A)=\omega^\zeta$ for some ordinal $\zeta$, it follows that there can only exist positive, non-trivial renorming results if $\zeta$ is a successor, say $\zeta=\gamma+1$.   Then our discussion here implies that the only $\xi$ for which we could obtain positive results would be $\xi\in [\omega^\gamma, \omega^{\gamma+1})$.  As we already mentioned, however, the strongest possible result in this case would be $\xi=\omega^\gamma$, which was covered by our main renorming theorem.  Thus we see that our positive renorming results are a complete solution to the question of Asplund renormings.   

It follows from the work of Brooker \cite{BrookerAsplund} that for every ordinal $\xi$, there is a Banach space $X_\xi$ with $Sz(X_\xi)=\omega^{\xi+1}$. Thus for every ordinal $\xi$, there do exist Banach spaces to which our main renorming theorem applies.    Moreover, these spaces can be taken to be reflexive, and to satisfy the property that $Sz(X_\xi^*)=\omega$ for every ordinal $\xi$ (we will prove this later).    Since $Sz(X\oplus Y)=\max\{Sz(X), Sz(Y)\}$ for every pair of Banach spaces $X,Y$, it follows that $W_{\xi, \zeta}=X_\xi\oplus X_\zeta^*$ is reflexive, $Sz(W_{\xi, \zeta})=\omega^{\xi+1}$, and $Sz(W_{\xi, \zeta}^*)=\omega^{\zeta+1}$.   Thus for every $\xi, \zeta$, the class of spaces to which Corollary \ref{super troopers} applies is non-empty.

For $1<p<\infty$, we say that an operator $A:X\to Y$ is $\xi$-AUS \emph{with power type} $p$ if there exists a constant $C$ such that for every $\sigma\geqslant 0$, $\rho^w_\xi(\sigma;A)\leqslant C\sigma^p$.   It is known that if $X$ is a Banach space with $Sz(X)\leqslant \omega$, then $X$ admits an equivalent AUS norm with power type $p$ for some $p>1$.  Moreover, the best possible $p$ is known in terms of the behavior of the indices $Sz(B_{X^*}, \ee)$.   The positive results regarding AUS norms has to do with the submultiplicative nature of the Szlenk index of a Banach space.  That is, for any $\delta, \ee>0$, $Sz(B_{X^*}, \delta\ee)\leqslant Sz(B_{X^*}, \delta)Sz(B_{X^*}, \ee)$.  However, this inequality fails if we replace $B_{X^*}$ with other $w^*$-compact sets.  Moreover, one can construct examples of operators with Szlenk index $\omega$ which admit no power type renorming.  One may fix a sequence $p_n\downarrow 1$ and a corresponding sequence of positive numbers $\theta_n$ tending very slowly to $0$ (depending on $p_n$) and let $A:\bigl(\oplus_n \ell_{p_n}\bigr)_{\ell_2}\to \bigl(\oplus_n \ell_{p_n} \bigr)_{\ell_2}$ be given by $A|_{\ell_{p_n}}=\theta_n I_{\ell_{p_n}}$.    It is a consequence of the work of Brooker \cite{BrookerAsplund} that this operator has Szlenk index $\omega$, while one easily checks that for a sufficiently slowly vanishing choice of $(\theta_n)$, there can be no power type renorming.  

Let us say for a constant $C\geqslant 0$, some $p>1$, and an ordinal $\xi$, that a Banach space $X$ has \emph{property} $(\xi,p,C)$ provided that for any constant $r>0$, any $y\in X$, any $B$-tree $B$ with $o(B)=\omega^\xi$, and any weakly null collection $(x_t)_{t\in B}\subset rB_X$, $$\inf\{\|y+x\|^p: t\in B, x\in \text{co}(x_s:\varnothing \prec s\preceq t)\} \leqslant \|y\|^p+C r^p.$$  An elementary computation shows that $X$ has property $(\xi, p,C)$ for some $C$ if and only if $X$ is $\xi$-AUS with power type $p$, which in turn implies that $Sz(X)\leqslant \omega^{\xi+1}$.   If $\xi=0$, of course it is sufficient to check that for any $y\in X$, any $r>0$, and any net $(x_\lambda)\subset rB_X$, $$\underset{\lambda}{\lim\inf} \|y+x_\lambda\|^p \leqslant \|y\|^p+Cr^p. $$ We say that $X$ has \emph{property} $(\xi,p)$ if it has property $(\xi, p)$ for some $C\geqslant 0$.    Also, given an Asplund Banach space $X$ and $\ee>0$, let $o(X, \ee)$ denote the supremum over all ordinals $\xi$ such that there exists a $B$-tree $B$ and a weakly null collection $(x_t)_{t\in B}\subset B_X$ such that for every $t\in B$ and every $x\in \text{co}(x_s:\varnothing\prec s\preceq t)$, $\|x\|\geqslant \ee$.  It follows from \cite{Causey1} that $o(X, \ee)<Sz(X)$.           

We next observe that for any $1<p<\infty$ and any ordinal $\xi$, there exists a reflexive Banach space $S_{\xi,p}$ such that \begin{enumerate}[(i)]\item $S_{\xi,p}$ has property $(\xi, p,1)$, \item $S_{\xi,p}$ cannot be renormed to have property $(\xi,r)$ for any $p<r<\infty$, \item $o(S_{\xi,p},1)\geqslant \omega^\xi$, \item $S_{\xi,p}^*$ has property $(0,q,1)$, where $1/p+1/q=1$.     \end{enumerate} Item $(i)$ implies that $Sz(S_{\xi, p})\leqslant \omega^{\xi+1}$, while item $(iii)$ implies that $Sz(X)>\omega^\xi$, so that $Sz(S_{\xi, p})=\omega^{\xi+1}$.   Item $(iv)$ implies that $Sz(S^*_{\xi, p})\leqslant \omega$, and since $S_{\xi, p}^*$ will be infinite dimensional, $Sz(S^*_{\xi, p})=\omega$ for all $\xi$.   Thus $S_{\xi, 2}\oplus S_{\zeta,2}^*$ furnishes an example of a reflexive Banach space with Szlenk index $\omega^{\xi+1}$ and the Szlenk index of the dual of which is $\omega^{\zeta+1}$.

We define the spaces.  We begin with $S_{0,p}=\ell_p$.   Assuming $S_{\xi, p}$ has been defined, we let $$S_{\xi+1,p}=\bigl(\oplus_n \ell_1^n(S_{\xi,p})\bigr)_{\ell_p}.$$   Last, if $\xi$ is a limit ordinal and if $S_{\zeta,p}$ has been defined for every $\zeta<\xi$, we let $$S_{\xi,p}=\bigl(\oplus_{\zeta<\xi} S_{\zeta, p}\bigr)_{\ell_p([0, \xi))}.$$   It follows from the work of Brooker \cite{BrookerDirect} that $Sz(S_{\xi, p})\leqslant \omega^{\xi+1}$ for every $\xi$.  It is easy to see that for any Banach spaces $X,Y$, $o(X\oplus_1 Y,1)\geqslant o(X,1)+o(Y,1)$, from which it follows that $o(\ell_1^n(S_{\xi,p}),1)\geqslant o(S_{\xi,p},1)n$ and $o(S_{\xi+1,p},1)\geqslant \sup_n o(\ell_1^n(S_{\xi,p}),1)\geqslant \omega^{\xi+1}$.  Furthermore, when $\xi$ is a limit ordinal, $o(S_{\xi,p},1)\geqslant \sup_{\zeta<\xi} o(S_{\zeta+1,p},1)\geqslant \omega^\xi$.   From this it follows that $o(S_{\xi,p},1)>\omega^\xi$ for every ordinal, whence $Sz(S_{\xi,p},1)=\omega^{\xi+1}$ for every $\xi$. 

Furthermore, it is easy to see that $\ell_q$ has property $(0,q,1)$, if $X,Y$ have property $(0,q,1)$, so does $X\oplus_\infty Y$, and if $(X_i)_{i\in I}$ have property $(0,q,1)$, so does $(\oplus_{i\in I} X_i)_{\ell_q(I)}$.  From this we deduce $(iv)$ by induction.    

Next, we want to show that $S_{\xi,p}$ has $(\xi, p,1)$.  If $\xi=0$, this is clear.  Recall that $S_{\xi+1, p}=\bigl(\oplus_n \ell_1^n(S_{\xi,p})\bigr)_{\ell_p}$.  Fix $N\in \nn$ and let $X_N:=\bigl(\oplus_{n=1}^N \ell_1^n(S_{\xi,p})\bigr)_{\ell_p^m}$.    Let $P:S_{\xi+1,p}\to X_N$ denote the projection onto $X_N$.   By the work of H\'{a}jek and Lancien \cite{HL}, it follows that $Sz(X_N)=Sz(S_{\xi, p})=\omega^{\xi+1}$.  By the characterization of Szlenk index given in \cite{Causey1}, for any $B$-tree $B$ with $o(B)=\omega^{\xi+1}$, any $r>0$, any weakly null collection $(x_t)_{t\in B}\subset r B_{S_{\xi+1, p}}$, and any $\ee>0$, there exists $t\in B$ and $x\in \text{co}(x_s:\varnothing\prec s\preceq t)$ such that $\|P_Nx\|<\ee$.   Now fix $r>0$, a $B$-tree $B$ with $o(B)=\omega^{\xi+1}$, $y\in S_{\xi+1, p}$, and a weakly null collection $(x_t)_{t\in B}\subset bB_{S_{\xi+1, p}}$.   Fix $N\in \nn$ such that $\|y-P_Ny\|<\ee$. Fix $t\in B$ and $x\in \text{co}(x_s:\varnothing\prec s\preceq t)$ such that $\|P_Nx\|<\ee$.  Then $$\|y+x\|\leqslant 2\ee+ \|P_N y+ (I-P_N)x\| \leqslant =2\ee +(\|P_Ny\|^p+\|(I-P_N)y\|^p)^{1/p} \leqslant 2\ee+(\|y\|^p+r^p)^{1/p}.$$   Since $\ee>0$ was arbitrary, we deduce that $$\inf \{\|y+z\|^p: t_0\in B, z\in \text{co}(x_s:\varnothing\prec s\preceq t)\} \leqslant \|y\|^p+r^p,$$ and $S_{\xi+1, p}$ has property $(\xi+1, p,1)$.    Proving that $S_{\xi,p}$ has property $(\xi, p,1)$ when $\xi$ is a limit ordinal is similar.   We now let $P_\gamma:S_{\xi,p}\to X_\gamma=\bigl(\oplus_{\zeta<\gamma} S_{\zeta, p}\bigr)$.  It follows now from \cite{BrookerAsplund} that $$Sz(X_\gamma) \leqslant \bigl(\sup_{\zeta<\gamma} Sz(S_{\zeta, p})\bigr)\omega \leqslant \omega^\gamma\omega=\omega^{\gamma+1}  <\omega^\xi$$   Again using \cite{Causey1}, for any $B$-tree $B$ with $o(B)=\omega^\xi$, any $r>0$, any weakly null collection $(x_t)_{t\in B}\subset r B_{S_{\xi, p}}$, and any $\ee>0$, there exists $t\in B$ and $x\in \text{co}(x_s: \varnothing\prec s\preceq t)$ such that $\|P_\gamma x\|<\ee$.  We then argue that $S_{\xi, p}$ has property $(\xi, p, 1)$ as in the successor case.

Finally, we argue that since $o(S_{\xi, p}, 1)\geqslant \omega^\xi$, $o(S_{\xi, p}, 1/n^{1/q}) \geqslant \omega^\xi n$.    First we define ``addition'' of well-founded $B$-trees.  If $C_1$, $C_2, \ldots, C_n$ are $B$-trees, we let $C_1+C_2+\ldots +C_n$ denote the $B$-tree consisting of all sequences $t_1\cat \ldots \cat t_m$ such that $m\leqslant n$, $t_i\in C_i$ for each $1\leqslant i\leqslant m$, and $t_i\in MAX(C_i)$ for each $1\leqslant i<m$. It is easy to see that the representation of a sequence in $C_1+\ldots +C_n$ as a concatenation $t_1\cat\ldots \cat t_m$ satisfying these conditions must be unique.  It is easy to see (see, for example, \cite{Causey}) that if $o(C_i)=\omega^\xi$, then $o(C_1+\ldots +C_n)=\omega^\xi n$.   Next, suppose $X_1, \ldots, X_n$ are Banach spaces such that $o(X_i, 1)\geqslant \omega^\xi$ for each $1\leqslant i\leqslant n$.  Then $o\bigl((\oplus_{i=1}^n X_i)_{\ell_p^n}, 1/n^{1/q}\bigr)\geqslant \omega^\xi n$.  Indeed, for each $i$, there exists a $B$-tree $C_i$ with $o(C_i)=\omega^\xi$ and a weakly null collection $(x^i_t)_{t\in C_i}$ such that for every $t\in C_i$ and every $x\in \text{co}(x^i_s:\varnothing \prec s\preceq t)$, $\|x\|=1$.   For $t\in C_1+\ldots +C_n$, let $t=t_1\cat \ldots \cat t_i$ be a representation of the form given above and let $x_t= x^i_{t_i}\in (\oplus_{i=1}^n X_i)_{\ell_p^n}$, where $X_i$ is identified with the obvious subspace of $(\oplus_{i=1}^n X_i)_{\ell_p^n}$.     Then for every $t\in C_1+\ldots +C_n$, say $t=t_1\cat \ldots \cat t_m$, and any $x\in \text{co}(x_s: \varnothing \prec s\preceq t)$, $x=\sum_{i=1}^m a_ix^i$, where $x_i\in \text{co}(x^i_s:\varnothing \prec s\preceq t_i)$, $a_i\geqslant 0$, and $\sum_{i=1}^m a_i=1$.   Therefore $$\|x\|\geqslant \bigl(\sum_{i=1}^m a_i^p\bigr)^{1/p}\geqslant 1/n^{1/q}.$$     Therefore in order to see that $o(S_{\xi, p}, 1/n^{1/q}) \geqslant \omega^\xi n$, we need only to see that we may write $S_{\xi, p}=(\oplus_{i=1}^n X_i)_{\ell_p^n}$ for some $X_1, \ldots, X_n$ with $o(X_i)\geqslant \omega^\xi$.  If $\xi=0$, we let $M_1, \ldots, M_n$ be disjoint, infinite subsets of $\nn$ and let $X_i=[e_j: j\in M_i]$, where $(e_i)$ is the $\ell_p$ basis.    If $\xi$ is a successor, say $\xi=\gamma+1$, we let $$X_i=\bigl(\oplus_{n\in M_i} \ell_1^n(S_{\gamma, p})\bigr)_{\ell_p}.$$   If $\xi$ is a limit, we let $A_1, \ldots, A_n$ be disjoint subsets of $[0, \xi)$ such that $\sup A_i=\xi$ for each $i$.    Then we let $$X_i=\bigl(\oplus_{\zeta\in A_i} S_{\zeta, p}\bigr)_{\ell_p(A_i)}.$$  

Note that the inequality $o(S_{\xi, p},1/n^{1/q})\geqslant \omega^\xi n$ yields that for any equivalent norm $|\cdot|$ on $S_{\xi, p}$, there exists a constant $C>0$ such that  $o((S_{\xi, p}, |\cdot|),1/Cn^{1/q}) \geqslant \omega^\xi n$.    We last show that if $X$ is any Banach space having property $(\xi, r)$, there exists a constant $D$ such that for any $n\in \nn$, $o(X, D/n^{1/s})\leqslant  \omega^\xi n$, where $1/r+1/s=1$.   Then if $r>p$, since $o(X, 1/Cn^{1/q})\geqslant \omega^\xi n$ and $o(X, D/n^{1/s})\leqslant \omega^\xi n$ cannot both hold for all $n\in \nn$,     we deduce that $S_{\xi, p}$ cannot be renormed to have property $(\xi, r)$ for any $r>p$.     Suppose that $X$ has property $(\xi, r, c)$ for some $c\geqslant 1$.   Let $B$ be a $B$-tree with $o(B)=\omega^\xi n$ and suppose $(x_t)_{t\in B}\subset B_X$ is a weakly null collection. Fix $\ee\in (0,c)$.   Fix any $t_1$ maximal  in $B^{\omega^\xi(n-1)}$ and let $y_1=x_{t_1}$.    Next, assume that for $1\leqslant m<n$,  $t_1, \ldots, t_m$, $y_1, \ldots, y_m$ have been defined such that $t_1\cat \ldots \cat t_{m-1}$ is a maximal member of $B^{\omega^\xi (n-m)}$, $y_i\in \text{co}(x_s: t_1\cat \ldots \cat t_{i-1}\prec s\preceq t_1\cat \ldots \cat t_i)$ for each $1\leqslant i\leqslant m$, and $$\|\sum_{j=1}^i y_j\|^r \leqslant \ee + \|\sum_{j=1}^{i-1} y_j\|^r + c$$ for each $1<i\leqslant m$.    Then let $C$ denote those non-empty sequences $t$ such that $t_1\cat \ldots \cat t_m\cat t\in B^{\omega^\xi(n-m-1)}$.   Since $t_1\cat \ldots \cat t_m$ is a maximal member of $B^{\omega^\xi(n-m)}$, $o(C)=\omega^\xi$.  Applying the definition of property $(\xi, r, c)$ to the vector $y=\sum_{i=1}^m y_i$ and the weakly null collection $(x_{t_1\cat \ldots \cat t_m\cat t})_{t\in C}$, we deduce the existence of $t_{m+1}\in C$ (which we may assume is maximal by replacing the sequence with a proper extension if necessary) such that $t_1\cat \ldots \cat t_{m+1}$ is a maximal member of $B^{\omega^\xi(n-m-1)}$ and a vector $y_{m+1}\in \text{co}(x_s: t_1\cat \ldots \cat t_m\prec s\preceq t_1\cat \ldots \cat t_{m+1})$ such that $$\|\sum_{i=1}^{m+1} y_i\|^r \leqslant \ee + \|\sum_{i=1}^m y_i\|^r + c.$$  This recursively defines $t_1, \ldots, t_n$ and $y_1, \ldots, y_n$ in such a way that $y=\frac{1}{n}\sum_{i=1}^n y_i\in \text{co}(x_s: \varnothing\prec s\preceq t_1\cat \ldots \cat t_n)$ and, by construction, \begin{align*} \|y\|^r & = n^{-r}\|\sum_{i=1}^{n-1} y_i\|^r \\ & < n^{-r}\Bigl[\ee+\|\sum_{i=1}^n y_i\|^r + c\Bigr] <\ldots \\ & < n^{-r}(n\ee + cn) < \frac{2cn}{n^r}.\end{align*}    Taking $r^{th}$ roots yields that for any constant $D>(2c)^{1/r}$, $o(X, C/n^{1/s}) \leqslant \omega^\xi n$ for every $n\in \nn$.

Thus we have shown by the previous examples that for every ordinal $\xi$ and every $1<p<\infty$, there exists a Banach space $S_{\xi, p}$ which has a $\xi$-AUS norm with power type $p$, and cannot be renormed to have power type better than $p$.  Moreover, it is easy to see using Proposition \ref{dualprop} that for an operator $A:X\to Y$, there exists a constant $C$ such that for every $\sigma>0$, $\rho^w_\xi(\sigma;A)\leqslant C\sigma^p$ if and only if there exists a constant $C'$ such that for every $\tau>0$,  $\delta^{w^*}_\xi(\tau;A^*)\geqslant C'\tau^{1/q}$, where again $1/p+1/q=1$.   Since the $S_{\xi, p}$ spaces are reflexive, we have produced among the spaces $S_{\xi, p}^*$ examples of spaces which are $\xi$-AUC with power type $q$ but which cannot be renormed to have power type better than $q$.  

Note that, arguing as above, the collection of spaces $S_{\xi, \infty}$ given by $S_{\xi, 0}=c_0$, $S_{\xi+1, \infty}=\bigl(\oplus_n \ell_1^n(S_{\xi, \infty})\bigr)_{c_0}$, and $S_{\xi, \infty}=\bigl(\oplus_{\zeta<\xi} S_{\zeta, \infty}\bigr)_{c_0([0, \xi))}$, we obtain examples of Banach spaces with $Sz(S_{\xi, \infty})=\omega^{\xi+1}$ and with $\rho_\xi(\sigma;S_{\xi, \infty})=0$ for $0<\sigma<1$.

In \cite{Causey2}, a characterization was given of which ordinals occur as the Szlenk index of a Banach space.  Let $\alpha>0$ be any ordinal, $\beta=1$, and $\beta_n=n$ for every $n\in \nn$.  For $n\in \nn$, let $\theta_n^{-1}=\log_2(n+1)$.   In \cite{Causey2}, a Banach space $\mathfrak{G}$ was constructed such that $Sz(\mathfrak{G})=\omega^{\alpha+1}$ and for every $n\in \nn$ and $\ee\in (0, \theta_n)$, $Sz(B_{\mathfrak{G}^*}, \ee)>\omega^\alpha n$.   It follows by the relationship between $o(\mathfrak{G}, \cdot)$ and $Sz(B_{\mathfrak{G}^*}, \cdot)$ given in \cite{Causey1} that there exists a constant $c>0$ such that for any $\ee>0$ and $n\in \nn$, if $Sz(B_{\mathfrak{G}^*}, c\ee)>\omega^\alpha n$, $o(\mathfrak{G}, \ee)\geqslant \omega^\alpha n$.  This means that for any $n\in \nn$ and $\ee\in (0, \theta_n/c)$, $o(\mathfrak{G}, \ee)\geqslant \omega^\alpha n$.  By replacing $c$ with any strictly larger number, we may assume that for any $\ee\in (0, \theta_n/c]$, $o(\mathfrak{G}, \ee)\geqslant \omega^\alpha n$.    Then for any equivalent norm $|\cdot|$ on $\mathfrak{
S}$, there exist constants $C,D$ such that for any natural number $n\in \nn$, $o((\mathfrak{G}, |\cdot|), \theta_n/cC)\geqslant \omega^\alpha n$, while, arguing as above, $o((\mathfrak{G}, |\cdot|), D/n^{1/q}) <\omega^\alpha n$, where $q$ is conjugate to $p$.  However, since $n^{1/q}\theta_n>cCD$ for sufficiently large $n$, we obtain a contradiction.  Thus we deduce that for every ordinal $\alpha>0$, there exists a Banach space $\mathfrak{G}$ such that $Sz(\mathfrak{G})=\omega^{\alpha+1}$ which cannot be renormed  to have $\rho^w_\alpha$-modulus of non-trivial power type.   We summarize this discussion, which yields a complete picture, by the following.   In what follows, for convenience, we say $A:X\to Y$ is $\xi$-AUS with power type $\infty$ if there exists a constant $C\geqslant 0$ such that for any $y\in Y$, any $b>0$, any $B$-tree $B$ with $o(B)=\omega^\xi$, and any weakly null collection $(x_t)_{t\in B}\subset b B_X$, $$\inf \{\|y+\sigma Ax\|: t\in B, x\in \text{co}(x_s:s\preceq t)\} \leqslant \max\{\|y\|, Cb \}.$$   

\begin{theorem} Fix an ordinal $\xi$. \begin{enumerate}[(i)]\item For any $1<p\leqslant \infty$, there exists a Banach space which is $\xi$-AUS of power type $p$ and which cannot be renormed to be $\xi$-AUS of power type better than $p$. This Banach space necessarily has Szlenk index $\omega^{\xi+1}$. \item There exists a Banach space with Szlenk index $\omega^{\xi+1}$ which cannot be renormed to be $\xi$-AUS with any power type if and only if $\xi>0$. \item There exists an operator with Szlenk index $\omega^{\xi+1}$ which cannot be renormed to be $\xi$-AUS with any power type. \item All of the appropriate dual statements for $w^*$-$\xi$-AUC Banach spaces and operators also hold. \end{enumerate}

\end{theorem}

We discuss the reason for the difference between the $\xi=0$ and $\xi>0$ cases in $(ii)$ of the preceding theorem.  Lancien \cite{Lancien1} showed that for any $\ee, \delta>0$ and any Banach space $X$, $Sz_{\ee\delta}(B_{X^*})\leqslant Sz_\ee (B_{X^*})Sz_\delta(B_{X^*})$ (which may fail if $B_{X^*}$ is replaced by another $w^*$-compact set).   This yields subgeometric growth of $Sz_{2^{-n}}(B_{X^*})$, which yields non-trivial results when $Sz(X)=\omega^{\omega^\xi}$ for some $\xi$.  This is because the ordinals $\omega^{\omega^\xi}$ are precisely the infinite ordinals $\delta$ which have the property that for any $\alpha, \beta<\delta$, $\alpha \beta<\delta$.  However, our familiar homogeneity argument yields that a non-trivial, $\xi$-asymptotically smooth norm on a Banach space $X$ implies that $Sz(X)=\omega^{\xi+1}$.  But the only ordinal which is of the form $\omega^{\omega^\xi}$ and of the form $\omega^{\zeta+1}$ is $\omega$.

\end{document}